\documentclass[11pt,a4paper]{amsart}
\pdfoutput=1
\usepackage{amsfonts, amsthm, amsmath, amssymb,inputenc,geometry,todonotes,mathtools,physics,bbm,tikz-cd,comment}
\usepackage[toc,page]{appendix}
\usepackage{kpfonts}
\usepackage[mathscr]{euscript}
\definecolor{winered}{rgb}{0.5,0,0}
\usepackage{hyperref}
\hypersetup{citecolor = winered, linkcolor = winered, menucolor = winered, colorlinks = true}

\usepackage[style = alphabetic]{biblatex}
\renewbibmacro{in:}{}
\usepackage[foot]{amsaddr}

\newtheoremstyle{theoremdd}
{\topsep}{\topsep}{\upshape}{0pt}{\bfseries}{.}{ }{\thmname{#1}\thmnumber{ #2}\thmnote{ (#3)}}

\theoremstyle{definition}
\newtheorem{Th}{Theorem}[section]
\newtheorem{Lemma}[Th]{Lemma}
\newtheorem{Cor}[Th]{Corollary}
\newtheorem{Prop}[Th]{Proposition}
\newtheorem{Def}[Th]{Definition}

\newtheorem{Rem}[Th]{Remark}

\newtheorem{Ex}[Th]{Example}

\newcommand{\Hom}{\text{Hom}}

\newcommand{\R}{\mathbb{R}}

\renewcommand{\D}{\mathscr{D}}

\newcommand{\B}{\mathbf{B}}
\newcommand{\conj}{\overline}

\newcommand{\cat}{\mathsf}
\renewcommand{\u}{\underline}
\newcommand{\co}{\text{co}}
\newcommand{\colim}{\text{colim}}

\newcommand{\pre}{\cat{Pre}}
\newcommand{\spre}{\cat{sPre}}

\newcommand{\cechpre}{\cat{s}\check{\cat{P}}\cat{re}}

\newcommand{\cart}{\cat{Cart}}

\newcommand{\site}{\mathscr{C}}
\newcommand{\good}{\text{good}}
\newcommand{\open}{\text{open}}
\newcommand{\sub}{\text{sub}}
\newcommand{\Man}{\cat{Man}}
\newcommand{\Open}{\cat{Open}}

\makeatletter
\newtheorem*{rep@theorem}{\rep@title}
\newcommand{\newreptheorem}[2]{%
\newenvironment{rep#1}[1]{%
 \def\rep@title{#2 \ref{##1}}%
 \begin{rep@theorem}}%
 {\end{rep@theorem}}}
\makeatother

\newreptheorem{theorem}{Theorem}

\newreptheorem{lemma}{Lemma}

\newreptheorem{cor}{Corollary}

\usetikzlibrary{calc}
\usetikzlibrary{decorations.pathmorphing}

\tikzset{curve/.style={settings={#1},to path={(\tikztostart)
    .. controls ($(\tikztostart)!\pv{pos}!(\tikztotarget)!\pv{height}!270:(\tikztotarget)$)
    and ($(\tikztostart)!1-\pv{pos}!(\tikztotarget)!\pv{height}!270:(\tikztotarget)$)
    .. (\tikztotarget)\tikztonodes}},
    settings/.code={\tikzset{quiver/.cd,#1}
        \def\pv##1{\pgfkeysvalueof{/tikz/quiver/##1}}},
    quiver/.cd,pos/.initial=0.35,height/.initial=0}

\tikzset{tail reversed/.code={\pgfsetarrowsstart{tikzcd to}}}
\tikzset{2tail/.code={\pgfsetarrowsstart{Implies[reversed]}}}
\tikzset{2tail reversed/.code={\pgfsetarrowsstart{Implies}}}

\title{Diffeological Principal Bundles and Principal Infinity Bundles}

\author{Emilio Minichiello}
\address{CUNY Graduate Center, Email address:  \href{mailto:eminichiello@gradcenter.cuny.edu}{eminichiello@gradcenter.cuny.edu}}

\begin{document}

\maketitle

\begin{abstract}
In this paper, we study diffeological spaces as certain kinds of discrete simplicial presheaves on the site of cartesian spaces with the coverage of good open covers. The \v{C}ech model structure on simplicial presheaves provides us with a notion of $\infty$-stack cohomology of a diffeological space with values in a diffeological abelian group $A$. We compare $\infty$-stack cohomology of diffeological spaces with two existing notions of \v{C}ech cohomology for diffeological spaces in the literature \cite{krepski2021sheaves} \cite{iglesias2020vcech}. Finally, we prove that for a diffeological group $G$, that the nerve of the category of diffeological principal $G$-bundles is weak homotopy equivalent to the nerve of the category of $G$-principal $\infty$-bundles on $X$, bridging the bundle theory of diffeology and higher topos theory.
\end{abstract}

\tableofcontents

\section{Introduction}

Principal $G$-bundles and \v{C}ech cohomology are important tools in the study of smooth manifolds. However, in recent years, the desire to expand the typical objects of study in differential geometry has led to various frameworks in which one can define a ``generalized smooth space." In this paper, we construct a bridge between two such frameworks. One of them is diffeology, as popularized in the textbook \cite{iglesias2013diffeology}. A diffeological space consists of a set $X$, and a set $\D_X$ of functions $\R^n \to X$, for varying $n \geq 0$, satisfying three simple conditions. While the definition of a diffeological space is simple, a large number of interesting spaces outside the purview of classical differential geometry can be given a diffeology. Every finite dimensional smooth manifold inherits a canonical diffeology, as does the set $C^\infty(X,Y)$ of smooth maps between any two diffeological spaces. In fact the category of diffeological spaces is complete, cocomplete and cartesian closed. More precisely it is a quasi-topos \cite{baez2009convenient}. This is of course not the case for the category of finite dimensional smooth manifolds, and thus provides a ``better" category in which to work. Various important constructions in classical differential geometry have been defined for diffeological spaces, like differential forms, deRham cohomology, fiber bundles, tangent spaces \cite{christensen2015tangent}, and recently \v{C}ech cohomology \cite{iglesias2020vcech} \cite{krepski2021sheaves} \cite{ahmadi2023diffeological}.

The second framework is higher topos theory. Here the objects of interest are $\infty$-stacks over the site $\cart$ of cartesian spaces, or more colloquially known as smooth $\infty$-groupoids. Many constructions of classical differential geometry can be extended to smooth $\infty$-groupoids, with interesting applications. One such extension is the notion of a principal $\infty$-bundle, as defined in \cite{NSSGeneral} and \cite{NSSPresentations}. Classical principal bundles and non-abelian bundle gerbes are particular examples of principal $\infty$-bundles. Principal $\infty$-bundles allow for a robust framework wherein one can study twisted, equivariant or differential refinements of generalized cohomology theories. For more on this theory we recommend the texts \cite{fiorenza2011cech}, \cite{amabel2021differential}, \cite{schreiber2013differential}, \cite{bunke2013differential}. In this paper, we will use the presentation of this theory by simplicial presheaves. Thus, the reader does not need to be comfortable with the language of $\infty$-categories in order to read this paper.

The following theorem is the main result of this paper, which shows that when $G$ is a diffeological group, then homotopically speaking, the notion of diffeological principal $G$-bundle over a diffeological space $X$ and the notion of $G$-principal $\infty$-bundle over $X$ coincide.

\begin{reptheorem}{th diffeological principal bundles are infinity bundles}
Given a diffeological space $X$ and a diffeological group $G$, there is a weak homotopy equivalence, 
\begin{equation*}
    N \cat{Prin}_G^\infty(X) \simeq N \cat{DiffPrin}_G(X).
\end{equation*}
where $N \cat{DiffPrin}_G(X)$ is the nerve of the category of diffeological principal $G$-bundles on $X$ and $N \cat{Prin}_G^{\infty}(X)$ is the nerve of the category of $G$-principal $\infty$-bundles on $X$.
\end{reptheorem}

Some explanations are in order. Baez-Hoffnung proved in \cite{baez2009convenient} that the category of diffeological spaces is equivalent to the category $\cat{ConSh}(\Open)$ of concrete sheaves on $\Open$, the category of open subsets of cartesian spaces and smooth maps, equipped with the coverage of open covers. This means that diffeological spaces can be thought of as certain kinds of sheaves. This is a very powerful point of view, especially with respect to studying the homotopy theory of diffeological spaces \cite{pavlov2022projective}. In particular, the category of concrete sheaves on $\cart$ embeds fully faithfully into the category $\spre(\cart)$ of simplicial presheaves on $\cart$, whose objects are functors $X: \cart^{op} \to \cat{sSet}$. Just as some presheaves of sets are sheaves, some simplicial presheaves are $\infty$-stacks. Higher topos theory can also be called homotopical sheaf theory, and $\infty$-stacks are the main objects of study. Sheaves of sets are discrete $\infty$-stacks, and stacks of groupoids are $1$-truncated $\infty$-stacks, see Section \ref{section smooth higher stacks} for a more detailed discussion. 

Embedding the category of diffeological spaces into the category of simplicial presheaves puts diffeological spaces into a homotopical framework that is both easy to work with and connects powerfully with the underlying homotopy theory of simplicial sets. For instance, using this embedding, we can immediately define a notion of cohomology on diffeological spaces, which we call $\infty$-stack cohomology. If $X$ is a diffeological space, and $A$ is an $\infty$-stack, whose $k$-fold delooping $\conj{W}^k A$ exists (notions we will explain in Section \ref{section smooth higher stacks}), then the $\infty$-stack cohomology of $X$ with coefficients in $A$ is given by the connected components of the derived mapping space 
\begin{equation*}
    \check{H}^k_\infty(X,A) = \pi_0 \R \Hom(X,\conj{W}^k A) = \pi_0 \u{\spre(\cart)}(QX, \conj{W}^k A),
\end{equation*}
where $Q$ is a cofibrant replacement functor for the \v{C}ech model structure. This functor $Q$ has a simple and explicit description in terms of the plots of $X$. In fact, it is the nerve of a diffeological category, see Proposition \ref{prop QX is nerve of a diff category}. In Corollary \ref{cor explicit description of infinity stack cohomology} we construct an explicit cochain complex whose cohomology gives the $\infty$-stack cohomology of $X$ with values in $A$ when $A$ is a diffeological abelian group.

There are two other examples of diffeological categories of diffeological spaces that act as ``resolutions" that already exist in the literature, namely the \v{C}ech groupoid $\check{C}(X)$ of \cite{krepski2021sheaves} and the gauge monoid $B//M$ of \cite{iglesias2020vcech}. These resolutions are used to define different notions of \v{C}ech cohomology for diffeological spaces, the Krepski-Watts-Wollbert diffeological \v{C}ech cohomology $\check{H}^k_{KWW}$ and the PIZ diffeological \v{C}ech cohomology $\check{H}^k_{PIZ}$ respectively. In Section \ref{section smooth higher stacks} we compare these various notions of diffeological \v{C}ech cohomologies and find that
\begin{equation*}
    \check{H}^0_{\infty}(X,A) \cong \check{H}^0_{KWW}(X,A) \cong \check{H}^0_{PIZ}(X,A), \qquad \check{H}^1_{\infty}(X,A) \cong \check{H}^1_{KWW}(X,A).
\end{equation*}
A full summary of known results about these cohomology theories is given after (\ref{eq diagram in cohomology}). It is currently unknown if all three of these cohomology theories are isomorphic in all degrees.

Now if $G$ is a diffeological group, there exists an $\infty$-stack called $\B G$, see Example \ref{ex delooping stack}, and $\infty$-stack cohomology of a diffeological space $X$ with values in $\B G$ is denoted $\check{H}^1_{\infty}(X, G)$. An immediate consequence of Theorem \ref{th diffeological principal bundles are infinity bundles} is the following result.

\begin{repcor}{cor nonabelian cohomology}
Given a diffeological space $X$ and a diffeological group $G$, there is an isomorphism of pointed sets
\begin{equation*}
    \check{H}^1_{\infty}(X, G) \cong \pi_0 \cat{DiffPrin}_G(X),
\end{equation*}
where $\pi_0 \cat{DiffPrin}_G(X)$ denotes the set of isomorphism classes of diffeological principal $G$-bundles on $X$, pointed at the isomorphism class of trivial bundles.
\end{repcor}

Thus $\infty$-stack cohomology of diffeological spaces also encompasses nonabelian cohomology. In the course of proving Theorem \ref{th diffeological principal bundles are infinity bundles} we will prove a bundle construction-type theorem, Theorem \ref{th diff cocycle theorem}, which gives an equivalence between a category of plotwise defined cocycles on $X$ and the category of diffeological principal $G$-bundles over $X$. This result may be of independent interest, and thus we prove it in Section \ref{section diff principal bundles}, before we need to introduce any abstract machinery.

It is our view that by applying tools from higher topos theory can be beneficial to the still young subject of diffeological spaces. In particular we believe that while the machinery of $\infty$-stack cohomology may come from an abstract framework, it can ultimately output important and down-to-earth results. Furthermore, higher topos theory already has definitions for higher principal bundles (called bundle gerbes) and connections on such objects inherently built into it. Pulling these definitions over to diffeological spaces and analyzing the results are the subject of future work. 

The paper is organized as follows. In Section \ref{section diffeological spaces}, we will give some background information about diffeological spaces. In Section \ref{section diff principal bundles} we turn to diffeological principal $G$-bundles. We define $G$-cocycles and prove a bundle construction-type theorem, Theorem \ref{th diff cocycle theorem}. In Section \ref{section diffeological spaces as concrete sheaves}, we give a brisk introduction to sheaf theory, and explain the Baez-Hoffnung Theorem \cite[Proposition 24]{baez2009convenient}. In Appendix \ref{section comparison of site structures}, we compare several categories of concrete sheaves on various sites, and show that they are all equivalent, proving that the category of diffeological spaces as given in Definition \ref{def diffeological space} is equivalent to the usual category of diffeological spaces considered in the literature. In Section \ref{section smooth higher stacks}, we will review the \v{C}ech model structure on simplicial presheaves over cartesian spaces. Proposition \ref{prop QX is nerve of a diff category} provides a cofibrant replacement of a diffeological space as the nerve of a diffeological category. We compare this diffeological category to two other diffeological categories $\check{C}(X)$ and $B//M$, which have been introduced in \cite{krepski2021sheaves} and \cite{iglesias2020vcech}, respectively. From these three diffeological categories, we obtain three separate notions of \v{C}ech cohomology for diffeological spaces, and compare them in Section \ref{section diffeological cech cohomologies}. In Section \ref{section principal diffeological bundles as principal infinity bundles}, we turn to the main result of this paper, that if $G$ is a diffeological group and $X$ is a diffeological space, then the nerve of the category of principal $G$-bundles on $X$ is weak homotopy equivalent to the nerve of the category of $G$-principal $\infty$-bundles over $X$.

\subsection*{Acknowledgements}

The author would like to thank Mahmoud Zeinalian for being a wonderful and kind advisor. He also thanks Patrick Iglesias-Zemmour, Jordan Watts and Matthew Cushman for various lively and helpful discussions.

\section{Diffeological Spaces} \label{section diffeological spaces}

In this section we give some background on diffeological spaces, all of which can be found in the textbook \cite{iglesias2013diffeology}.

\begin{Def} \label{def cart, good covers, parametrizations}
Let $M$ be a finite dimensional smooth manifold\footnote{We will assume throughout this paper that manifolds are Hausdorff and paracompact.}. We say a collection of subsets $\mathcal{U} = \{U_i \subseteq M \}_{i \in I}$ is an \textbf{open cover} if each $U_i$ is an open subset of $M$, and $\bigcup_{i \in I} U_i = M$. If $U$ is a finite dimensional smooth manifold diffeomorphic to $\R^n$ for some $n \in \mathbb{N}$, we call $U$ a \textbf{cartesian space}. We call $\mathcal{U} = \{U_i \subseteq M \}$ a \textbf{cartesian open cover} of a manifold $M$ if it is an open cover of $M$ and every $U_i$ is a cartesian space. We say that $\mathcal{U}$ is a \textbf{good open cover} if it is a cartesian open cover, and further every finite non-empty intersection $U_{i_0 \dots i_k} = U_{i_0} \cap \dots \cap U_{i_k}$ is a cartesian space. 

Let $\cat{Man}$ denote the category whose objects are finite dimensional smooth manifolds and whose morphisms are smooth maps. Let $\cart$ denote the full subcategory whose objects are cartesian spaces. Given a set $X$, let $\cat{Param}(X)$ denote the set of \textbf{parametrizations} of $X$, namely the collection of set functions $p: U \to X$, where $U \in \cat{Cart}$.
\end{Def}

\begin{Def} \label{def diffeological space}
A \textbf{diffeology} on a set $X$, consists of a collection $\mathscr{D}$ of parametrizations $p: U \to X$ satisfying the following three axioms:
\begin{enumerate}
    \item $\mathscr{D}$ contains all points $\R^0 \to X$,
    \item If $p: U \to X$ belongs to $\mathscr{D}$, and $f: V \to U$ is a smooth map, then $pf: V \to X$ belongs to $\mathscr{D}$, and
    \item If $\{U_i \subseteq U \}_{i \in I}$ is a good open cover of a cartesian space $U$, and $p: U \to X$ is a parametrization such that $p|_{U_i} : U_i \to X$ belongs to $\mathscr{D}$ for every $i \in I$, then $p \in \mathscr{D}$.
\end{enumerate}
A set $X$ equipped with a diffeology $\mathscr{D}$ is called a \textbf{diffeological space}. Parametrizations that belong to a diffeology are called \textbf{plots}. We say a set function $f: X \to Y$ between diffeological spaces is \textbf{smooth} if for every plot $p: U \to X$ in $\mathscr{D}_X$, the composition $pf: U \to Y$ belongs to $\mathscr{D}_Y$. We often denote the set of smooth maps from $X$ to $Y$ by $C^\infty(X,Y)$.

\begin{Rem}
In what follows, when we say that $f: X \to Y$ is a map of diffeological spaces or a smooth map, we mean that it is a smooth function in the above sense.
\end{Rem}

Denote the category whose objects are diffeological spaces and morphisms are smooth maps between them by $\cat{Diff}$. An isomorphism in this category is called a \textbf{diffeomorphism}.
\end{Def}

\begin{Rem} \label{rem why this def of diff spaces}
Note that Definition \ref{def diffeological space} is not the exact definition of diffeological spaces as usually given in the literature, such as \cite[Article 1.5]{iglesias2013diffeology}. However it is precisely the definition of diffeological space as defined in \cite[Definition 2.7]{pavlov2022numerable} and \cite[Notation 3.3.15]{sati2022equivariant}, as we will prove in Theorem \ref{thm diff = consh(cart)}. We will call the diffeological spaces defined in \cite[Article 1.5]{iglesias2013diffeology} \textbf{classical diffeological spaces} and denote their category by $\cat{Diff}'$. 

In Section \ref{section diffeological spaces as concrete sheaves}, leveraging \cite{baez2009convenient}, we will explain how to think of diffeological spaces as concrete sheaves on the site $(\cart, j_\good)$, namely cartesian spaces with the coverage of good open covers. Leveraging this perspective we show that $\cat{Diff}$ is equivalent to $\cat{Diff}'$ in Appendix \ref{section comparison of site structures}. 

However there are real advantages to using $\cat{Diff}$ over $\cat{Diff}'$, one of them being Lemma \ref{lem equiv bundles def}, which is false for $\cat{Diff}'$. There are other more technical advantages as well. In Section \ref{section smooth higher stacks}, we will consider $\cat{Diff}$ embedded into the category $\spre(\cart)$, which can be given the \v{C}ech model structure $\cechpre(\cart)$. If one uses $j_\open$ instead of $j_\good$ on $\cart$, and $\mathcal{U}$ is an arbitrary cartesian open cover of a cartesian space $U$, then there is no guarantee that $\check{C}(\mathcal{U})$ will be projective cofibrant. Using good open covers ensures that it is projective cofibrant, which is necessary for much of the theory to be developed. Using $j_\good$ also allows us to leverage Theorem \ref{th delooping sheaf of groups is infinity stack} and Theorem \ref{th delooping infinity stack is infinity stack}, which are vital to our results.
\end{Rem}

\begin{Ex} \label{ex manifold diffeology}
If $M$ is a finite dimensional smooth manifold, then the set of parametrizations $p: U \to M$ that are smooth in the sense of classical differential geometry forms a diffeology \cite[Article 4.3]{iglesias2013diffeology}. We call this the \textbf{manifold diffeology} of the underlying set of $M$. Further a map $f: M \to N$ of smooth manifolds is smooth in the classical sense if and only if it is smooth as a map of diffeological spaces. This implies that the functor $\D_{\cat{Man}}: \cat{Man} \to \cat{Diff}$ that sends a manifold to its underlying set equipped with the manifold diffeology is fully faithful.
\end{Ex}

\begin{Def}
We say that a map $i : X \to Y$ of diffeological spaces is an \textbf{induction} if for every plot $p: U \to Y$ there exists a plot $q: U \to X$ such that $p = i q$. 
\end{Def}

\begin{Def} \label{def subduction}
We say that a map $\pi : X \to Y$ of diffeological spaces is a \textbf{subduction} if it is surjective, and for every plot $p: U \to Y$, there exists a good open cover $\{U_i \subseteq U \}$, and plots $p_i: U_i \to X$ making the following diagram commute
\begin{equation}
    \begin{tikzcd}
	{U_i} & X \\
	{U} & Y
	\arrow["\pi", from=1-2, to=2-2]
	\arrow["{p_i}", from=1-1, to=1-2]
	\arrow[hook, from=1-1, to=2-1]
	\arrow["{p}"', from=2-1, to=2-2]
\end{tikzcd}
\end{equation}
\end{Def}

\begin{Lemma} \label{lem universal property of subduction}
	Let $X \xrightarrow{\pi} Y$ be a subduction, then a function $Y \xrightarrow{f} Z$ is smooth if and only if $f \pi$ is, and $f$ is a subduction if and only if $f \pi$ is.
\end{Lemma}

\begin{Lemma} \label{lem maps with sections are subductions}
If $f: X \to Y$ is a smooth map of diffeological spaces such that there exists a section, i.e. a smooth map $\sigma : Y \to X$ such that $f \sigma = 1_Y$, then $f$ is a subduction.
\end{Lemma}

\begin{Def}
Let $(X, \D_X)$ be a diffeological space and $A \xhookrightarrow{i} X$ a subset. Then consider the collection $\D^X_A$ of parametrizations $p : U \to A$ such that $ip : U \to X$ is a plot of $X$. It is not hard to see that this collection is a diffeology, which we call the \textbf{subspace diffeology} on $A$. Note that $i$ is an induction when $A$ is equipped with the subspace diffeology.
\end{Def}

\begin{Def}
Let $(X, \D_X)$ be a diffeological space and suppose that $\sim$ is an equivalence relation on the underlying set $X$. Let $X \xrightarrow{\pi} X/{\sim}$ denote the quotient function taking a point $x \in X$ to its equivalence class $[x] \in X /{\sim}$. Then consider the collection $\D_{\sim}^X$ of parametrizations $p : U \to X/{\sim}$ such that there exists a good open cover $\{U_i \subseteq U \}$ and plots $p_i: U_i \to X$ making the following diagram commute
\begin{equation}
        \begin{tikzcd}
	{U_i} & X \\
	{U} & X/{\sim}
	\arrow["\pi", from=1-2, to=2-2]
	\arrow["{p_i}", from=1-1, to=1-2]
	\arrow[hook, from=1-1, to=2-1]
	\arrow["{p}"', from=2-1, to=2-2]
\end{tikzcd}
\end{equation}
It is not hard to see that this forms a diffeology, which we call the \textbf{quotient diffeology} on $X/{\sim}$. Note that $\pi$ is a subduction when $X/{\sim}$ is equipped with the quotient diffeology.
\end{Def}

The category $\cat{Diff}$ of diffeological spaces is complete and cocomplete. Suppose $F: \cat{J} \to \cat{Diff}$ is a diagram of diffeological spaces. Then a parametrization $p: U \to \lim F$, (where we are taking $\lim F$ to be the limit of the underlying sets of the $F_j$) is a plot if and only if the composite $U \xrightarrow{p} \lim F \to F_j$ is a plot for every $j \in \cat{J}$.

Similarly a parametrization $p: \cat{J} \to \colim \, F$ is a plot if and only if there exists a good open cover $\{ U_i \to U \}$ and plots $U_i \xrightarrow{p_i} F_{j_i}$ for each $i$, such that the following diagram commutes:
\begin{equation*}
    \begin{tikzcd}
	{U_i} & F_{j_i} \\
	{U} & \colim \, F
	\arrow["\pi", from=1-2, to=2-2]
	\arrow["{p_i}", from=1-1, to=1-2]
	\arrow[hook, from=1-1, to=2-1]
	\arrow["{p}"', from=2-1, to=2-2]
\end{tikzcd}
\end{equation*}

\begin{Def}
Given any two diffeological spaces $X,Y$, consider the set $\cat{Diff}(X,Y)$ of smooth maps $f: X \to Y$. Let $\D_{X \to Y}$ denote the collection of parametrizations $p: U \to C^\infty(X,Y)$ such that the adjoint function $p^{\#} : U \times X \to Y$, defined by $p^{\#}(u,x) = p(u)(x)$ is smooth. This collection is a diffeology, which we call the \textbf{functional diffeology}.
\end{Def}

The functional diffeology makes $\cat{Diff}$ a Cartesian closed category. We will see in Section \ref{section diffeological spaces as concrete sheaves} that $\cat{Diff}$ is in fact a quasitopos.

To every diffeological space $X$, we can consider the category $\cat{Plot}(X)$, whose objects are plots $p_0: U_{p_0} \to X$ and whose morphisms $f_0 : p_1 \to p_0$ are smooth maps $f_0 : U_{p_1} \to U_{p_0}$ making the following diagram commute 
\begin{equation}
    \begin{tikzcd}
	{U_{p_1}} && {U_{p_0}} \\
	& X
	\arrow["{f_0}", from=1-1, to=1-3]
	\arrow["{p_1}"', from=1-1, to=2-2]
	\arrow["{p_0}", from=1-3, to=2-2]
\end{tikzcd}
\end{equation}

\begin{Lemma}[{\cite[Proposition 2.7]{CWS2014}}] \label{lem diff space is colimit of plots}
Given a diffeological space $X$, let $q: \cat{Plot}(X) \to \cat{Diff}$ denote the functor that sends a plot $p : U \to X$ to the cartesian space $U$ considered as a diffeological space with its manifold diffeology. Then $X \cong \colim \, q$.
\end{Lemma}

\section{Diffeological Principal Bundles} \label{section diff principal bundles}
In this section we introduce diffeological principal $G$-bundles, and prove a bundle construction-type theorem (Theorem \ref{th diff cocycle theorem}) that gives an equivalence between the groupoid of $G$-cocycles and the groupoid of diffeological principal $G$-bundles over a diffeological space $X$. This development will be needed in Section \ref{section principal diffeological bundles as principal infinity bundles}. One does not need to know anything about model categories or homotopy theory to understand this section.

\begin{Def}
A \textbf{diffeological group} is a group $G$ equipped with a diffeology $\D_G$ such that the multiplication map $m: G \times G \to G$, 	and inverse map $i: G \to G$ are smooth.
\end{Def}

\begin{Def}
A right \textbf{diffeological group action} of a diffeological group $G$ on a diffeological space $X$ is a smooth map $\rho: X \times G \to X$ such that $\rho(x, e_G) = x$, and $\rho(\rho(x,g), h) = \rho(x, gh)$, where $e_G$ denotes the identity element of $G$. We will often keep $\rho$ implicit and denote such an action by $x \cdot g$.
\end{Def}

\begin{Def} \label{def diff princ bundle}
Let $G$ be a diffeological group, and $P$ be a diffeological right $G$-space. A map $\pi: P \to X$ of diffeological spaces is a \textbf{diffeological principal $G$-bundle} if:
\begin{enumerate}
    \item the map $\pi: P \to X$ is a subduction, and
    \item the map $\text{act}: P \times G \to P \times_X P$ defined by $(p,g) \mapsto (p, p \cdot g)$, which we call the \textbf{action map} is a diffeomorphism.
\end{enumerate}
\end{Def}

A map of diffeological principal $G$-bundles $P \to P'$ over $X$ is a diagram
\begin{equation*}
    \begin{tikzcd}
	P && {P'} \\
	& X
	\arrow["{\pi'}", from=1-3, to=2-2]
	\arrow["\pi"', from=1-1, to=2-2]
	\arrow["f", from=1-1, to=1-3]
\end{tikzcd}
\end{equation*}
where $f$ is a $G$-equivariant smooth map. A diffeological principal $G$-bundle $P$ is said to be \textbf{trivial} if there exists an isomorphism $\varphi: X \times G \to P$, called a \textbf{trivialization}, where $\text{pr}_1: X \times G \to X$ is the product bundle. Let $\cat{DiffPrin}_G(X)$ denote the category of diffeological principal $G$-bundles over a diffeological space $X$. The following properties of diffeological principal $G$-bundles are not hard to prove, see \cite[Chapter 8]{iglesias2013diffeology}.

\begin{Lemma} \label{lem props of principal bundles}
Let $G$ be a diffeological group and $\pi: P \to X$ a diffeological principal $G$-bundle, then we have the following:
\begin{enumerate}
    \item if $f: Y \to X$ is a smooth map, then the pullback $f^* P \to Y$ is a diffeological principal $G$-bundle,
    \item if there is a section $s: X \to P$, namely $\pi \circ s = 1_X$, then $P$ is trivial,
    \item if $f: P \to P'$ is a map of diffeological principal $G$-bundles over a diffeological space $X$, then it is an isomorphism.
\end{enumerate}
\end{Lemma}

By Lemma \ref{lem props of principal bundles}.(3), the category $\cat{DiffPrin}_G(X)$ is a groupoid for every diffeological group $G$ and every diffeological space $X$.

\begin{Prop}[{\cite[Article 8.19]{iglesias2013diffeology}}] \label{prop bundles on cartesian spaces are trivial}
Given a diffeological group $G$ and a diffeological principal $G$-bundle $\pi : P \to U$, where $U$ is a cartesian space, there exists a trivialization $\varphi: U \times G \to P$.
\end{Prop}

\begin{Lemma} \label{lem equiv bundles def}
Given a diffeological group $G$ and a diffeological right $G$-space $P$, a map $\pi: P \to X$ is a diffeological principal $G$-bundle if and only if for every plot $p_0: U_{p_0} \to X$, the pullback $p_0^*P$ is trivial, and the map $\text{act}: P \times G \to P \times_X P$ defined by $(p,g) \mapsto (p,p \cdot g)$ is a diffeomorphism.
\end{Lemma}

\begin{proof}
$(\Rightarrow)$ If $\pi : P \to X$ is a diffeological principal $G$-bundle, and $p_0: U_{p_0} \to X$ is a plot, then by Lemma \ref{lem props of principal bundles} $p_0^* P$ is a diffeological principal $G$-bundle over $U_{p_0}$. Since $U_{p_0}$ is cartesian, by Proposition \ref{prop bundles on cartesian spaces are trivial}, $p_0^* P$ is trivial, with trivialization $\varphi_{p_0}: U_{p_0} \times G \to p_0^*P$.

$(\Leftarrow)$ If $p_0^* P$ is trivial for every plot $p_0 : U_{p_0} \to X$, then there exists a trivialization $\varphi_{p_0}: U_{p_0} \times G \to p_0^* P$, and thus we obtain the following commutative diagram
\begin{equation*}
    \begin{tikzcd}
	{U_{p_0} \times G} & {p_0^*P} & P \\
	& U_{p_0} & X
	\arrow["p_0"', from=2-2, to=2-3]
	\arrow["\pi", from=1-3, to=2-3]
	\arrow[from=1-2, to=2-2]
	\arrow["{\psi_{p_0}}", from=1-2, to=1-3]
	\arrow[from=1-1, to=2-2]
	\arrow["\varphi_{p_0}", from=1-1, to=1-2]
	\arrow["{1_U \times e_G}", curve={height=-12pt}, from=2-2, to=1-1]
\end{tikzcd}
\end{equation*}
Since this is true for every plot $p_0: U_{p_0} \to X$, then by taking $U_{p_0}$ as a good cover of itself, $pi$ can be seen to be a subduction.
\end{proof}

\begin{Lemma} \label{lem free transitive action equiv to diffeomorphism}
Condition (2) of Definition \ref{def diff princ bundle} is equivalent to the condition that $G$ acts on the fibers of $\pi$ freely and transitively.
\end{Lemma}

\begin{proof}
If $\text{act}: P \times G \to P \times_X P$ is a diffeomorphism, then $G$ clearly acts on the fibers of $\pi$ freely and transitively. Now suppose $G$ acts on the fibers of $\pi$ freely and transitively. The map $\text{act}: P \times G \to P \times_X P$ is smooth. Since the action is free and transitive it means also that the map $\text{act}$ is a bijection. We need then only to show that the inverse function is smooth. Namely if $\langle q, q' \rangle : U \to P \times_X P$ is a plot, where $q, q': U \to P$ are plots, then we wish to show that the composite function $\text{act}^{-1} \langle q, q' \rangle$ is a plot of $P \times G$. Now this is the set map $u \mapsto (q(u), \text{diff}(q(u), q'(u)))$, where $\text{diff}: P \times_X P \to G$ is the composite map $\text{proj}_2 \text{act}^{-1}$. This is a plot of $P \times G$ if it is a plot in both factors. Obviously it is a plot of the first factor, so we need only show that it is a plot of the second factor. This is the map that we will denote by $\tau$, namely $\tau(u) = \text{diff}(q(u), q'(u))$. Now, let $r = \pi q = \pi q'$ denote the plot $r: U \to X$. We have the following commutative cube
\begin{equation}
\begin{tikzcd}
	U &&& {U \times G} \\
	& {(U \times G) \times_U (U \times G)} &&&& P \\
	&& {P \times_X P} & U \\
	& {U \times G} &&&& X \\
	&& P
	\arrow["\pi", from=2-6, to=4-6]
	\arrow["\pi"', from=5-3, to=4-6]
	\arrow[from=3-3, to=5-3]
	\arrow[from=3-3, to=2-6]
	\arrow["{\varphi_{q'}}", from=1-4, to=2-6]
	\arrow["r"{pos=0.3}, from=3-4, to=4-6]
	\arrow["{\varphi_q}"{description}, from=4-2, to=5-3]
	\arrow[from=2-2, to=4-2]
	\arrow[from=4-2, to=3-4]
	\arrow[from=1-4, to=3-4]
	\arrow[from=2-2, to=1-4]
	\arrow["k"{description}, from=2-2, to=3-3]
	\arrow["{\langle q, 1_U \rangle}"{description}, curve={height=18pt}, from=1-1, to=4-2]
	\arrow["{\langle 1_U, q'\rangle}"{description}, curve={height=-18pt}, from=1-1, to=1-4]
	\arrow["h"{description}, dashed, from=1-1, to=2-2]
\end{tikzcd}
\end{equation}
and $(U \times G) \times_U (U \times G) \cong U \times G \times G$. Note that $\langle q, q' \rangle: U \to P \times_X P$ factors as $kh$, where $k$ is the map $U \times G \times G \to P \times_X P$ induced by the plotwise trivializations of $P$ along $q$ and $q'$. Now we can see that $\tau$ factors as
\begin{equation*}
    \begin{tikzcd}
	U & {P \times_X P} & {P \times G} & G \\
	& {U \times G \times G} & {U \times G}
	\arrow["{\langle q,q'\rangle}"', from=1-1, to=1-2]
	\arrow["{\text{act}^{-1}}"', from=1-2, to=1-3]
	\arrow["{\text{proj}_2}"', from=1-3, to=1-4]
	\arrow["\tau", curve={height=-18pt}, from=1-1, to=1-4]
	\arrow["h"', curve={height=12pt}, from=1-1, to=2-2]
	\arrow["k"', from=2-2, to=1-2]
	\arrow["n"', from=2-2, to=2-3]
	\arrow["{\text{proj}_2}"', curve={height=12pt}, from=2-3, to=1-4]
\end{tikzcd}
\end{equation*}
where $n: U \times G \times G \to U \times G$ is the smooth map $(u,g, g') \mapsto (u, (g')^{-1} g)$. Thus $\tau$ is a composite of smooth maps, and therefore is a plot.
\end{proof}

Diffeological principal $G$-bundles are a true generalization of classical principal $G$-bundles in the following sense. If $M$ is a finite dimensional smooth manifold, $G$ is a Lie group, and $\pi : P \to M$ is a classical principal $G$-bundle on $M$, then we can consider the map of diffeological space $\D_{\cat{Man}}(\pi) : \D_{\cat{Man}}(P) \to \D_{\cat{Man}}(M)$, given by the functor $\D_{\cat{Man}}: \cat{Man} \to \cat{Diff}$ from Example \ref{ex manifold diffeology}. It turns out that this map of diffeological spaces is a diffeological principal $\D_{\cat{Man}}(G)$-bundle, in fact more is true.

\begin{Prop}[{\cite[Theorem 3.1.7]{waldorf2012transgression}}]
If $M$ is a finite dimensional smooth manifold, $G$ is a Lie group, and $\cat{Prin}_G(M)$ denotes the groupoid of classical principal $G$-bundles over $M$, then the functor 
\begin{equation}
    \D_M : \cat{Prin}_G(M) \to \cat{DiffPrin}_G(M)
\end{equation}
that assigns to a classical principal $G$-bundle $\pi$ the corresponding diffeological principal $\D_{\cat{Man}}(G)$-bundle, is an equivalence of groupoids.
\end{Prop}

\begin{Rem}
In what follows we will not use the notation $\D_{\cat{Man}}(M)$ to distinguish between a manifold and its corresponding diffeological space, but will instead rely on context.
\end{Rem}

Classically, principal $G$-bundles over a smooth manifold $M$ are classified up to isomorphism by homotopy classes of maps from $M$ to a classifying space $BG$. There has been some work \cite{christensen2021smooth}, \cite{magnot2017diffeology} extending this result to diffeology. Since diffeological spaces are so much more general than smooth manifolds, one must only consider numerable principal $G$-bundles to classify them in this sense as in the above references.

However, there is another way of classifying principal $G$-bundles over a smooth manifold $M$, that produces the whole groupoid of principal $G$-bundles, rather than just the isomorphism classes, see Example \ref{ex delooping stack}. One goal of this paper is to extend this idea to diffeological principal $G$-bundles. This will be achieved in Section \ref{section principal diffeological bundles as principal infinity bundles}. In order to classify diffeological principal $G$-bundles in this way, we must understand how to construct bundles from their cocycles and vice versa. 

\begin{Def} \label{def cocycle}
Given a diffeological space $X$ and a diffeological group $G$, call a collection $g = \{g_{f_0} \}$ of smooth maps $g_{f_0} : U_{p_1} \to G$ indexed by maps of plots $f_0: U_{p_1} \to U_{p_0}$ of $X$ a \textbf{$G$-cocycle} if for every pair of composable plot maps of $X$
\begin{equation*}
    U_{p_2} \xrightarrow{f_1} U_{p_1} \xrightarrow{f_0} U_{p_0}
\end{equation*}
it follows that
\begin{equation} \label{eqn diff cocycle condition}
    g_{f_0 f_1} = (g_{f_0} \circ f_1) \cdot g_{f_1}.
\end{equation}
We call (\ref{eqn diff cocycle condition}) the \textbf{diffeological $G$-cocycle condition}. 

Given two $G$-cocycles, $g,g'$, we say a collection $h = \{h_{p_0}\}$ of smooth maps $h_{p_0}: U_{p_0} \to G$ indexed by plots of $X$ is a \textbf{morphism of $G$-cocycles} $h : g \to g'$ if for every map $f_0 :U_{p_1} \to U_{p_0}$ of plots of $X$, it follows that
\begin{equation} \label{eqn map of diff cocycles}
    g'_{f_0} \cdot h_{p_1} = (h_{p_0} \circ f_0) \cdot g_{f_0}.
\end{equation}
\end{Def}

\begin{Rem}
The definition of diffeological $G$-cocycles is reminiscent of classical $G$-cocycles (usually written $g_{ij}$ for some cover $\mathcal{U} = \{U_i \subseteq M \}$ of a manifold), but also seems to have come from nowhere. In Section \ref{section principal diffeological bundles as principal infinity bundles} we will see how one comes to this definition purely from the framework of higher topos theory.
\end{Rem}

Let $\cat{Coc}(X,G)$ denote the category whose objects are cocycles of $X$ with values in $G$ and whose morphisms are morphisms of cocycles. Composition is defined as follows. If $h: g \to g'$ and $h': g' \to g''$ are morphisms of cocycles, then let $(h' \circ h)$ denote the morphism of cocycles defined plotwise by $(h' \circ h)_{p_0} = h'_{p_0} \cdot h_{p_0}$ for a plot $p_0 : U_{p_0} \to X$. Let us show that $(h' \circ h)$ is actually a morphism of cocycles. A morphism of cocycles $h: g \to g'$ implies that if $f_0: U_{p_1} \to U_{p_0}$ is a map of plots, then 
\begin{equation*}
    g'_{f_0} h_{p_1} = (h_{p_0} \circ f_0) g_{f_0}
\end{equation*}
and $h': g' \to g''$ implies that
\begin{equation*}
    g''_{f_0} h'_{p_1} = (h'_{p_1} \circ f_0) g'_{f_0}.
\end{equation*}
Thus 
\begin{equation*}
g''_{f_0} h'_{p_1} = (h'_{p_0} \circ f_0) (h_{p_0} \circ f_0) g_{f_0} h_{p_1}^{-1}.
\end{equation*}
and therefore
\begin{equation*}
 g''_{f_0} (h'_{p_1} \cdot h_{p_1}) = ((h'_{p_0} \cdot h_{p_0}) \circ f_0) g_{f_0}   
\end{equation*}
This proves that $(h' \circ h)$ is a morphism of cocycles. Note that $\cat{Coc}(X,G)$ is a groupoid by taking $(h^{-1})_p = h_p^{-1}$.

Given a diffeological space $X$ and a $G$-cocycle $g$ on $X$, we wish to construct a diffeological principal $G$-bundle $\pi : P \to X$, such that we can recover the $G$-cocycle $g$ by looking at plotwise trivializations of $P$. We will do this as follows. Consider the diffeological space
\begin{equation}
    \widehat{P} = \coprod_{p_0 \in \cat{Plot}(X)} U_{p_0} \times G.
\end{equation}

We label the elements of $\widehat{P}$ by $(x_{p_0},k_0)$, where $p_0: U_{p_0} \to X$ is a plot, $x_{p_0} \in U_{p_0}$, and $k_0 \in G$. We write $(x_{p_1}, k_1) \sim (x_{p_0}, k_0)$ if there exists a map $f_0: U_{p_1} \to U_{p_0}$ of plots such that $f_0(x_{p_1}) = x_{p_0}$ and $k_0 = g_{f_0}(x_{p_1}) \cdot k_1$. This relation is reflexive and transitive, but not symmetric, so abuse notation by letting $\sim$ also denote the smallest equivalence relation containing $\sim$. In other words we say that $(x_{p_1}, k_1) \sim (x_{p_0}, k_0)$ if and only if there exists a finite zig-zag of plot maps
\begin{equation} \label{eq zig-zag of plot maps}
(x_{p_1}, k_1) \xrightarrow{f_0} (x_{q_0}, h_0) \xleftarrow{f_1} (x_{q_1}, h_1) \xrightarrow{f_2} \dots \xleftarrow{f_n} (x_{q_n}, h_n) \xrightarrow{f_{n+1}} (x_{p_0}, k_0),
\end{equation}
such that $f_0(x_{p_1}) = x_{q_0}$, $h_0 = g_{f_0}(x_{p_1}) \cdot k_1$, $f_1(x_{q_1}) = x_{q_0}$, $h_0 = g_{f_1}(x_{q_1}) \cdot h_1$, $\dots$, $f_{n+1}(x_{q_n}) = x_{p_0}$, and $k_0 = g_{f_{n+1}}(x_{q_n}) \cdot h_n$. By setting $f_0$ or $f_{n+1}$ equal to the identity one can obtain any kind of zig-zag from one of the form above.

Let $P = \widehat{P} / \sim$, and denote its elements by $[x_{p_0}, k_0]$. There is a smooth map $\pi : P \to X$ given by $\pi[x_{p_0}, k_0] = p_0(x_{p_0})$. This map is well-defined, because if $(x_{p_1}, k_1) \sim (x_{p_0}, k_0)$, then there is a finite zig-zag of plot maps connecting them as above, so 
\begin{equation}
    p_1(x_{p_1}) = q_0(f_0(x_{p_1})) = q_0(x_{q_0}) = q_0(f_1(x_{q_1})) = q_1(x_{q_1}) = \dots = q_n(x_{q_n}) = q_n(f_{n+1}(x_{q_n})) = p_0(x_{p_0}).
\end{equation}
Thus $\pi : P \to X$ is well defined. We let $\pi = \cat{Cons}(g)$, short for construction. 

As sets, we can think of $\widehat{P}$ as the objects of a category, and a morphism in this category looks like
\begin{equation*}
    (x_{p_0}, k_0) \xrightarrow{f_0} (x_{p_1}, k_1)
\end{equation*}
where $f_0: U_{p_1} \to U_{p_0}$ is a plot map such that $f_0(x_{p_1}) = x_{p_0}$ and $k_1 = g_{f_0}(x_{p_1}) \cdot k_0$. Then $P \cong \pi_0 \widehat{P}$. This might be a helpful way to think about this construction, and we will say more about this observation in Section \ref{section principal diffeological bundles as principal infinity bundles}.

\begin{Prop} \label{prop construction of bundle}
Given a diffeological group $G$, a diffeological space $X$ and a $G$-cocycle $g$ on $X$, the map $\pi : P \to X$ where $\pi = \cat{Cons}(g)$ is a diffeological principal $G$-bundle.
\end{Prop}

\begin{proof}
First let us show that there is an action of $G$ on $P$. Let the action be defined by $[x_{p_0}, k_0] \cdot g = [x_{p_0}, k_0 \cdot g]$. This action is well defined, as suppose a zig-zag of the form (\ref{eq zig-zag of plot maps}) identifies $(x_{p_1}, k_1)$ with $(x_{p_0}, k_0)$. Then the same zig-zag with $h_i$ replaced with $h_i \cdot g$ will identify $(x_{p_1}, k_1 \cdot g)$ with $(x_{p_0}, k_0 \cdot g)$.

Now let us show that if $p_1: U_{p_1} \to X$ is a plot, then there exists a $G$-equivariant diffeomorphism $\varphi_{p_1} : U_{p_1} \times G \to p_1^* P$. First note that
\begin{equation*}
p_1^*P = \left \{ (x_{p_1},[x_{p_0},k_0]) \in U_{p_1} \times P \, : \, p_1(x_{p_1}) = \pi[x_{p_0}, k_0] \right \}
\end{equation*}
So given a point $(x_{p_1},[x_{p_0}, k_0]) \in p_1^*P$, we have a commutative diagram
\begin{equation*}
\begin{tikzcd}
	{*} & U_{p_0} \\
	U_{p_1} & X
	\arrow["p_1"', from=2-1, to=2-2]
	\arrow["p_0", from=1-2, to=2-2]
	\arrow["x_{p_1}"', from=1-1, to=2-1]
	\arrow["x_{p_0}", from=1-1, to=1-2]
	\arrow["x"{description}, from=1-1, to=2-2]
\end{tikzcd}
\end{equation*}
where $*$ denotes the plot $* \cong \R^0 \xrightarrow{x} X$. We can think of this diagram as a zig-zag
\begin{equation*}
    (x_{p_1}, g_{x_{p_1}}(*)g_{x_{p_0}}^{-1}(*) k_0) \xleftarrow{x_{p_1}} (*, g_{x_{p_0}}^{-1}(*)k_0) \xrightarrow{x_{p_0}} (x_{p_0}, k_0).
\end{equation*}
Thus we can identify 
\begin{equation*}
 [x_{p_1}, g_{x_{p_1}}(*) g_{x_{p_0}}^{-1}(*) k_0] = [x_{p_0}, k_0].   
\end{equation*}

Now let us define a map $\varphi_{p_1}: p_1^*P \to  U_{p_1} \times G$ by $\varphi(x_{p_1},[x_{p_0}, k_0]) = (x_{p_1}, g_{x_{p_1}}(*) g_{x_{p_0}}^{-1}(*)k_0)$. Let us show that this is well defined. First suppose that $(x_{q_0}, l_0) \sim (x_{p_0}, k_0)$ is identified by a single morphism, i.e. there exists a plot map $f_0: U_{q_0} \to U_{p_0}$ such that $f_0(x_{q_0}) = x_{p_0}$ and $k_0 = g_{f_0}(x_{q_0}) \cdot l_0$, then the following diagram commutes
\begin{equation*}
    \begin{tikzcd}
	{*} & U_{p_0} \\
	U_{q_0} & X
	\arrow["q_0"', from=2-1, to=2-2]
	\arrow["p_0", from=1-2, to=2-2]
	\arrow["x_{q_0}"', from=1-1, to=2-1]
	\arrow["x_{p_0}", from=1-1, to=1-2]
	\arrow["f_0"{description}, from=2-1, to=1-2]
\end{tikzcd}
\end{equation*}
However, we can also think of $x_{p_0}$ and $x_{q_0}$ as maps of plots. Thus from the cocycle condition (\ref{eqn diff cocycle condition}), we have 
\begin{equation*}
    g_{(f_0 \circ x_{q_0})}(*) = (g_{f_0} \circ x_{q_0})(*) \cdot g_{x_{q_0}}(*).
\end{equation*}
Now let us abuse notation for the rest of this proof by writing $g_{x_{p_0}}$ for $g_{x_{p_0}}(*)$. Thus we have
\begin{equation}
    g_{x_{p_0}} = g_{f_0}(x_{q_0}) \cdot g_{x_{q_0}}.
\end{equation}

Therefore
\begin{equation*}
    \varphi(x_{p_1}, [x_{p_0}, k_0]) = (x_{p_1}, g_{x_{p_1}} g_{x_{p_0}}^{-1} k_0) = (x_{p_1}, g_{x_{p_1}} g_{x_{p_0}}^{-1} g_{f_0}(x_{q_0}) l_0) = (x_{p_1}, g_{x_{p_1}} g_{x_{q_0}}^{-1} l_0) = \varphi(x_{p_1}, [x_{q_0}, l_0]).
\end{equation*}

Now if $(x_{q_0}, l_0)$ and $(x_{p_0}, k_0)$ are connected by an arbitrary finite zig-zag, then using the above argument on every morphism in the zig-zag shows that $\varphi(x_{p_1}, [x_{q_0}, l_0]) = \varphi(x_{p_1}, [x_{p_0}, k_0])$. So $\varphi$ is well defined. It is also $G$-equivariant, as $\varphi(x_{p_1}, [x_{p_0}, k_0] \cdot g) = \varphi(x_{p_1}, [x_{p_0}, k_0 g]) = (x_{p_1}, g_{x_{p_1}} g_{x_{p_0}}^{-1} k_0 g) = (x_{p_1},g_{x_{p_1}} g_{x_{p_0}}^{-1} k_0) \cdot g$.

We define an inverse $\varphi^{-1}_{p_1}: U_{p_1} \times G \to p_1^*P$ by $(x_{p_1}, g) \mapsto (x_{p_1}, [x_{p_1}, g])$. This is clearly $G$-equivariant, and it is easy to see that $\varphi_{p_1} \circ \varphi^{-1}_{p_1} = 1_{U_{p_1} \times G}$ and $\varphi^{-1}_{p_1} \circ \varphi_{p_1} = 1_{p_1^* P}$. Thus $P$ is plotwise trivial.

Now let us show that $G$ acts on the fibers of $\pi$ freely and transitively. However this is immediate, as a fiber of $\pi$ is in particular a pullback:
\begin{equation*}
    \begin{tikzcd}
	{\pi^{-1}(x)} & P \\
	{*} & X
	\arrow["x"', from=2-1, to=2-2]
	\arrow["\pi", from=1-2, to=2-2]
	\arrow[from=1-1, to=2-1]
	\arrow[from=1-1, to=1-2]
	\arrow["\lrcorner"{anchor=center, pos=0.125}, draw=none, from=1-1, to=2-2]
\end{tikzcd}
\end{equation*}
and every constant map $* \to X$ is a plot, and we've already shown that for any plot this pullback is trivial, and thus $\pi^{-1}(x) \cong * \times G \cong G$, which acts freely and transitively on itself by right multiplication. Thus by Lemma \ref{lem free transitive action equiv to diffeomorphism}, the action map $P \times G \to P \times_X P$ is a diffeomorphism. Thus by Lemma \ref{lem equiv bundles def}, $P \xrightarrow{\pi} X$ is a diffeological principal $G$-bundle. 
\end{proof}

\begin{Rem} \label{rem plotwise trivialization convention}
As it is convenient for the rest of this section, we rename the plotwise trivialization $\varphi^{-1}_{p_0} : U_{p_0} \times G \to p_0^* \cat{Cons}(g)$ to $\varphi_{p_0}$. 
\end{Rem}

\begin{Prop} \label{prop construction of morphism of bundles}
Given a diffeological space $X$, a diffeological group $G$, $G$-cocycles $g$, $g'$ on $X$, with corresponding diffeological principal $G$-bundles $P = \cat{Cons}(g)$ and $P' = \cat{Cons}(g')$ and a morphism $h : g \to g'$ of $G$-cocycles, there is an induced morphism $\widetilde{h} = \cat{Cons}(h): P \to P'$ of diffeological principal $G$-bundles. Furthermore, for every plot $p_0: U_{p_0} \to X$, if $\widetilde{h}: P \to P'$ is a morphism of diffeological principal $G$-bundles, then we obtain a commutative diagram
\begin{equation} \label{eq map of bundles under trivialization}
    \begin{tikzcd}
	{U_{p_0} \times G} && {U_{p_0} \times G} \\
	{p_0^*P} && {p_0^*P'} \\
	P && {P'} \\
	& X
	\arrow["{\psi_{p_0}}"', from=2-1, to=3-1]
	\arrow["{\varphi_{p_0}}"', from=1-1, to=2-1]
	\arrow["{\widetilde{h}_{p_0}}", from=1-1, to=1-3]
	\arrow["{\widehat{h}_{p_0}}", from=2-1, to=2-3]
	\arrow["\pi"', from=3-1, to=4-2]
	\arrow["{\varphi'_{p_0}}", from=1-3, to=2-3]
	\arrow["{\psi'_{p_0}}", from=2-3, to=3-3]
	\arrow["{\pi'}", from=3-3, to=4-2]
	\arrow["{\widetilde{h}}", from=3-1, to=3-3]
\end{tikzcd}
\end{equation}
where $\widetilde{h}_{p_0}(x_{p_0}, k_0) = (x_{p_0}, h_{p_0}(x_{p_0}) \cdot k_0)$.
\end{Prop}

\begin{proof}
Define a map $\widetilde{h} = \cat{Cons}(h) : P \to P'$ as follows. Given a point $[x_{p_0}, k_0]$ in $P$, let $\widetilde{h}([x_{p_0}, k_0]) = [x_{p_0}, h_{p_0}(x_{p_0}) \cdot k_0]$. Let us show that this map is well defined. Suppose that $[x_{p_1}, k_1] = [x_{p_0}, k_0]$, we want to show that $[x_{p_1}, h_{p_1}(x_{p_1}) k_1] = [x_{p_0}, h_{p_0}(x_{p_0}) k_0]$. Suppose that $(x_{p_1}, k_1)$ and $(x_{p_0}, k_0)$ are connected by a single morphism, i.e. there exists a plot map $f_0: U_{p_1} \to U_{p_0}$ such that $f_0(x_{p_1}) = x_{p_0}$ and $k_0 = g_{f_0}(x_{p_1}) k_1$. Then as elements of $P'$, we have
\begin{equation}
    \begin{aligned}
        \widetilde{h}([x_{p_0}, k_0]) & = [x_{p_0}, h_{p_0}(x_{p_0}) k_0] \\
        & = [x_{p_0}, h_{p_0}(f_0(x_{p_1})) g_{f_0}(x_{p_1})k_1]\\
        &= [x_{p_0}, g'_{f_0}(x_{p_1}) h_{p_1}(x_{p_1}) k_1] \\
        &= [x_{p_1}, h_{p_1}(x_{p_1}) k_1] \\
        &= \widetilde{h}([x_{p_1}, k_1]),
    \end{aligned}
\end{equation}
where the third equality follows from (\ref{eqn map of diff cocycles}). Thus if $(x_{p_1}, k_1)$ and $(x_{p_0}, k_0)$ are connected by an arbitrary finite zig-zag, then using the above argument on every morphism in the zig-zag shows that $\widetilde{h}([x_{p_1}, k_1]) = \widetilde{h}([x_{p_0}, k_0])$. Thus $\widetilde{h}$ is well-defined, and clearly smooth. It is also $G$-equivariant, as 
\begin{equation}
    \begin{aligned}
        \widetilde{h}([x_{p_0}, k_0] \cdot g) & = \widetilde{h}([x_{p_0}, k_0 g]) \\
        & = [x_{p_0}, h_{p_0}(x_{p_0}) k_0 g] \\
        & = [x_{p_0}, h_{p_0}(x_{p_0}) k_0] \cdot g \\
        & = \widetilde{h}([x_{p_0}, k_0]) \cdot g.
    \end{aligned}
\end{equation}
Thus $\widetilde{h}$ is a morphism of diffeological principal $G$-bundles over $X$.

Now given a plot $p_0: U_{p_0} \to X$, we obtain the commutative diagram (\ref{eq map of bundles under trivialization}) by pulling back $\widetilde{h}$ along $p_0$ and the plotwise trivialization $\varphi$. We need only show that $\widetilde{h}_{p_0}(x_{p_0}, k_0) = (x_{p_0}, h_{p_0}(x_{p_0}) \cdot k_0)$. Since $\widetilde{h}_{p_0} = (\varphi'_{p_0})^{-1} \widehat{h}_{p_0} \varphi_{p_0}$ we have
\begin{equation}
\begin{aligned}
        \widetilde{h}_{p_0}(x_{p_0}, k_0) &= (\varphi'_{p_0})^{-1} \widehat{h}_{p_0} \varphi_{p_0}(x_{p_0}, k_0) \\
        &= (\varphi'_{p_0})^{-1} \widehat{h}_{p_0}(x_{p_0}, [x_{p_0}, k_0]) \\
        &= (\varphi'_{p_0})^{-1}(x_{p_0}, [x_{p_0}, h_{p_0}(x_{p_0}) \cdot k_0]) \\
        &= (x_{p_0}, h_{p_0}(x_{p_0}) \cdot k_0),
\end{aligned}
\end{equation}
where we have used the plotwise trivialization from the proof of Proposition \ref{prop construction of bundle} and the convention of Remark \ref{rem plotwise trivialization convention}.
\end{proof}

The content of Propositions \ref{prop construction of bundle} and \ref{prop construction of morphism of bundles} can be summarized by saying that we have a functor $\cat{Cons} : \cat{Coc}(X,G) \to \cat{DiffPrin}_G(X)$. Our goal now is to show that this functor is an equivalence, namely that it is fully faithful and essentially surjective.

Let us first show essential surjectivity. Suppose that $\pi : P \to X$ is a diffeological principal $G$-bundle. Suppose that for every plot $p_0: U_{p_0} \to X$ we choose a trivialization $\varphi_{p_0} : U_{p_0} \times G \to p_0^*P$, which is a $G$-equivariant diffeomorphism. Let $f_0: U_{p_1} \to U_{p_0}$ be a map of plots. Then we obtain the following commutative diagram
\begin{equation} \label{eq bundle under trivializations}
    \begin{tikzcd}
	{U_{p_1} \times G} && {U_{p_0} \times G} \\
	{p_1^*P} && {p_0^*P} \\
	{U_{p_1}} && {U_{p_0}} \\
	& P \\
	& X
	\arrow["{p_1}"', from=3-1, to=5-2]
	\arrow["{p_0}", from=3-3, to=5-2]
	\arrow["{\psi_{p_1}}"', from=2-1, to=4-2]
	\arrow["{\psi_{p_0}}", from=2-3, to=4-2]
	\arrow["{\varphi_{p_1}}"', from=1-1, to=2-1]
	\arrow[from=2-1, to=3-1]
	\arrow["{\widetilde{f}_0}", from=1-1, to=1-3]
	\arrow["{\widehat{f}_0}", from=2-1, to=2-3]
	\arrow["{f_0}", from=3-1, to=3-3]
	\arrow[from=4-2, to=5-2]
	\arrow[from=2-3, to=3-3]
	\arrow["{\varphi_{p_0}}", from=1-3, to=2-3]
\end{tikzcd}
\end{equation}
where $\widehat{f}_0$ is obtained from $f_0$ by pullback (taking pullbacks is functorial), and $\widetilde{f}_0$ is defined as $\widetilde{f}_0 = \varphi_{p_0}^{-1} \widehat{f}_0 \varphi_{p_1}$. Given $(x_{p_1}, k) \in U_{p_1} \times G$, it follows that 
\begin{equation} \label{eq what plot maps do under trivialization}
\widetilde{f}_0(x_{p_1}, k) = (f_0(x_{p_1}), g_{f_0}(x_{p_1}) \cdot k)
\end{equation}
for some map $g_{f_0}: U_{p_1} \to G$. This is because both $\widehat{f}_0$ and $\widetilde{f}_0$ are maps of diffeological principal $G$-bundles, and therefore $G$-equivariant. 

If furthermore we have a pair of composable plot maps $U_{p_2} \xrightarrow{f_1} U_{p_1} \xrightarrow{f_0} U_{p_0}$, then 
\begin{equation}
    (\widetilde{f_0 f_1})(x_{p_2}, k) = ((f_0 f_1)(x_{p_2}), g_{f_0 f_1}(x_{p_2}) \cdot k) = ((f_0 f_1)(x_{p_2}), g_{f_0}(f_1(x_{p_2}) \cdot g_{f_1}(x_{p_2}) \cdot k) = \widetilde{f_0} \widetilde{f_1}(x_{p_2}, k).
\end{equation}
From this we obtain the cocycle condition (\ref{eqn diff cocycle condition}).

Thus given a pair of a diffeological principal $G$-bundle $\pi : P \to X$, and a choice $\varphi$ of plotwise trivializations, we obtain a $G$-cocycle $g$. Denote this cocycle by $g = \cat{Ext}(P, \varphi)$ for extracting the cocycle from the principal bundle. We wish to show that there exists an isomorphism $\tau: \cat{Cons}( \cat{Ext}(P,\varphi)) \to P$ of diffeological principal $G$-bundles.

Let $Q = \cat{Cons}( \cat{Ext}(P,\varphi))$ and let $\widetilde{Q} = \coprod_{p_0 \in \cat{Plot}(X)} U_{p_0} \times G$. Let $q : \widetilde{Q} \to Q$ denote the quotient map. Let us define a set function $\tau: Q \to P$ as follows. Given a point $[x_{p_1}, k_1] \in Q$, let $\tau([x_{p_1}, k_1])$ be $\psi_{p_1} \varphi_{p_1}(x_{p_1}, k_1)$, where $\psi_{p_1}$ and $\varphi_{p_1}$ are the maps given in (\ref{eq bundle under trivializations}). Let us show that this function is well defined. Suppose that $[x_{p_1}, k_1] = [x_{p_0}, k_0]$ are connected by a single plot map $f_0$, i.e. that $f_0(x_{p_1}) = x_{p_0}$ and $k_0 = g_{f_0}(x_{p_1}) \cdot k_1$. Then by $(\ref{eq bundle under trivializations})$ and (\ref{eq what plot maps do under trivialization}) we have
\begin{equation}
\begin{aligned}
    \tau([x_{p_0}, k_0]) & = \psi_{p_0} \varphi_{p_0}(x_{p_0}, k_0) \\
    & = \psi_{p_0} \varphi_{p_0}(f_0(x_{p_1}), g_{f_0}(x_{p_1}) \cdot k_1) \\
    & = \psi_{p_0} \varphi_{p_0} \widetilde{f}_0 (x_{p_1}, k_1) \\
    & = \psi_{p_1} \varphi_{p_1}(x_{p_1}, k_1) \\
    & = \tau([x_{p_1}, k_1]).
\end{aligned}
\end{equation}
Thus if $(x_{p_1}, k_1)$ and $(x_{p_0}, k_0)$ are connected by an arbitrary finite zig-zag, then using the above argument on every morphism in the zig-zag shows that $\tau$ is well defined. To see that it is smooth, since $q$ is a submersion, by Lemma \ref{lem universal property of subduction} it is enough to show that $\tau q = \widetilde{\tau}$ is smooth. But this is just the map $\widetilde{\tau} : \coprod_{p_0} U_{p_0} \times G \to P$ that sends $(x_{p_0}, k_0)$ to $\psi_{p_0} \varphi_{p_0} (x_{p_0})$ and these are smooth maps, thus $\widetilde{\tau}$ is smooth and therefore $\tau$ is smooth. Since $\psi_{p_0}$ and $\varphi_{p_0}$ are $G$-equivariant, it is easy to see that $\tau$ is as well. Thus $\tau$ is a map of diffeological principal $G$-bundles, so by Lemma \ref{lem props of principal bundles}, it is an isomorphism. Therefore the functor $\cat{Cons} : \cat{Coc}(X,G) \to \cat{DiffPrin}_G(X)$ is essentially surjective.

Now let us show that $\cat{Cons}$ is fully faithful. Suppose that $h$ and $h'$ are maps of cocycles $h, h' : g \to g'$, and let $\widetilde{h} = \cat{Cons}(h)$ and $\widetilde{h}' = \cat{Cons}(h')$. Suppose that $\widetilde{h} = \widetilde{h}'$. Then using the canonical plotwise trivialization $\varphi$ of $P = \cat{Cons}(g)$ and $P' = \cat{Cons}(g')$ we obtain two copies of (\ref{eq map of bundles under trivialization}) for $\widetilde{h}$ and $\widetilde{h}'$. Since they are equal as maps, this implies that for every plot $p_0: U_{p_0} \to X$, $\widetilde{h}_{p_0} = \widetilde{h}'_{p_0}$. But by Proposition \ref{prop construction of morphism of bundles}, $\widetilde{h}_{p_0}$ determines $h_{p_0}$ and $\widetilde{h}'_{p_0}$ determines $h'_{p_0}$. Thus $h = h'$. Thus $\cat{Cons}$ is a faithful functor.

Now suppose that $P = \cat{Cons}(g)$ and $P' = \cat{Cons}(g')$ and $\widetilde{h}: P \to P'$ is a map of diffeological principal $G$-bundles. We wish to construct a map $h: g \to g'$ of cocycles such that $\cat{Cons}(h) = \widetilde{h}$. We obtain such a morphism $h$ of cocycles by pulling $\widetilde{h}$ along the canonical plotwise trivialization $\varphi$ of $\cat{Cons}(g)$ as in (\ref{eq map of bundles under trivialization}), so that for every plot $p_0$ of $X$, we have $\widetilde{h}_{p_0}(x_{p_0}, k) = (x_{p_0}, h_{p_0}(x_{p_0}) \cdot k_0)$. One can check that $h$ is a morphism of $G$-cocycles by chasing around the left hand square in the following commutative diagram
\begin{equation}
\begin{tikzcd}
	&& P \\
	& {U_{p_0} \times G} && {P'} \\
	{U_{p_1} \times G} && {U_{p_0} \times G} \\
	& {U_{p_1} \times G}
	\arrow["{\widetilde{f}_0}", from=3-1, to=2-2]
	\arrow["{\psi_{p_0}\varphi_{p_0}}", from=2-2, to=1-3]
	\arrow["{\widetilde{h}}", from=1-3, to=2-4]
	\arrow["{\widetilde{h}_{p_0}}", from=2-2, to=3-3]
	\arrow["{\psi_{p_0}' \varphi_{p_0}'}"', from=3-3, to=2-4]
	\arrow["{\widetilde{h}_{p_1}}"', from=3-1, to=4-2]
	\arrow["{\widetilde{f}'_0}"', from=4-2, to=3-3]
\end{tikzcd}
\end{equation}

Now we wish to show that $\cat{Cons}(h) = \widetilde{h}$. Let $x \in X$, and consider the plot $x : * \to X$ that sends the point to $x$. If we let $p_0 = x$ in (\ref{eq map of bundles under trivialization}), then $U_{p_0} \times G \cong G$. Let $p = \psi_x \varphi_x(e_G)$. Then $\widetilde{h}(p) = \psi'_x \varphi'_x \widetilde{h}_x(e_G) = \psi'_x \varphi'_x (h_x(e_G))$. However $\cat{Cons}(h)(p)$ is also determined plotwise by $h_x$, i.e. $\cat{Cons}(h)(p) = \psi'_x \varphi'_x (h_x(e_G)) = \widetilde{h}{p}$. Since $\widetilde{h}$ and $\cat{Cons}(h)$ are $G$-equivariant, and the action of $G$ on $\pi^{-1}(x)$ is transitive, this implies that $\widetilde{h}$ and $\cat{Cons}(h)$ agree on $\pi^{-1}(x)$. Since $x$ was arbitrary, this implies that $\widetilde{h} = \cat{Cons}(h)$. Thus $\cat{Cons}$ is a full functor. In summary we have proven the following.

\begin{Th} \label{th diff cocycle theorem}
Given a diffeological space $X$ and a diffeological group $G$, the functor 
\begin{equation}
    \cat{Cons}: \cat{Coc}(X,G) \to \cat{DiffPrin}_G(X)
\end{equation}
is an equivalence of groupoids.
\end{Th}

\begin{Rem}
Any choice of plotwise trivialization $\varphi$ with $\varphi_{p_0}: U_{p_0} \times G \cong p_0^* P$ gives a quasi-inverse to the functor $\cat{Cons}$ above.
\end{Rem}

\begin{Rem}
Weaker, but somewhat similar results to Theorem \ref{th diff cocycle theorem} have been proven in \cite{watts2014diffeological} and \cite{ahmadi2023diffeological}. But notice in these papers that the correspondence was only proven on the level of isomorphism classes, with different notions of diffeological \v{C}ech cohomology, and they only establish a bijection of sets. We say more about the other notions of diffeological \v{C}ech cohomology in the literature in Section \ref{section diffeological cech cohomologies}.
\end{Rem}

\section{Diffeological Spaces as Concrete Sheaves} \label{section diffeological spaces as concrete sheaves}

A major development in the theory of diffeological spaces was made in \cite{baez2009convenient}, which showed that the category of diffeological spaces is equivalent to the category of concrete sheaves on the site of open subsets of cartesian spaces. Here we introduce the theory necessary to understand this result, and to prepare the ground for Section \ref{section smooth higher stacks}. In Appendix \ref{section comparison of site structures}, we will show that the category of concrete sheaves on several various sites of interest are equivalent, giving a justification for Definition \ref{def diffeological space}. Nothing in this section is new, but it may be helpful to those less familiar with topos theory.

\begin{Def}
Let $\site$ be a category, and $U \in \site$. A \textbf{family of morphisms} over $U$ is a set of morphisms $r = \{r_i: U_i \to U \}_{i \in I}$ in $\site$ with codomain $U$.

A \textbf{refinement} of a family of morphisms $t = \{ t_j: V_j \to U \}_{j \in J}$ over $U$ consists of a family of morphisms $r = \{ r_i: U_i \to U \}_{i \in I}$, a function $\alpha: I \to J$ and for each $i \in I$ a map $f_i: U_i \to V_{\alpha(i)}$ making the following diagram commute:
\begin{equation}
    \begin{tikzcd}
	{U_i} && {V_{\alpha(i)}} \\
	& U
	\arrow["{f_i}", from=1-1, to=1-3]
	\arrow["{r_i}"', from=1-1, to=2-2]
	\arrow["{t_{\alpha(i)}}", from=1-3, to=2-2]
\end{tikzcd}
\end{equation}
If $r$ is a refinement of $t$, with maps $f_i: U_i \to V_{\alpha(i)}$, then we write $f: r \to t$. 
\end{Def}

We wish to consider added structure to a category that generalizes the notion of a topology. We will use families of morphisms as a generalized notion of "open cover."

\begin{Def}
A \textbf{collection of families} $j$ on a category $\site$ consists of a set $j(U)$ for each $U \in \site$, whose elements $\{r_i: U_i \to U \} \in j(U)$ are families of morphisms over $U$.

We call a collection of families $j$ on $\site$ a \textbf{coverage} if it satisfies the following property: for every $\{ r_i: U_i \to U \} \in j(U)$, and every map $g: V \to U$ in $\site$, then there exists a family $\{ t_j: V_j \to V \} \in j(V)$ such that $g t_j$ factors through some $r_i$. Namely for every $t_j$ there exists some $i$ and some map $s_j: V_j \to U_i$ making the following diagram commute:
\begin{equation} \label{eqn coverage def}
    \begin{tikzcd}
	{V_j} & {U_i} \\
	V & U
	\arrow["{t_j}"', from=1-1, to=2-1]
	\arrow["{s_j}", from=1-1, to=1-2]
	\arrow["{r_i}", from=1-2, to=2-2]
	\arrow["g"', from=2-1, to=2-2]
\end{tikzcd}
\end{equation}
The families $\{ r_i: U_i \to U \} \in j(U)$ are called \textbf{covering families} over $U$. If a map $r_i: U_i \to U$ belongs to a covering family $r \in j(U)$, then we say that $r_i$ is a \textbf{covering map}.

If $\site$ is a category, and $j$ is a coverage on $\site$, then we call the pair $(\site, j)$ a \textbf{site}.
\end{Def}

\begin{Ex} \label{ex j of topological space coverage}
Let $X$ be a topological space and let $\mathcal{O}(X)$ denote the partially ordered set of open subsets of $X$. Let $j_X$ denote the collection of familes on $\mathcal{O}(X)$ such that $j_X(U)$ is the set of all open covers of $U$, namely $\{ U_i \subseteq U \} \in j_X(U)$ if $\bigcup_i U_i = U$.

This collection of families is a coverage, for suppose we have fixed an open cover $\{U_i \subseteq U \}$ and an open subset $V \subseteq U$. Then $\{ V \cap U_i \subseteq V \}$ is an open cover of $V$, and $V \cap U_i \subseteq U_i$. We call $j_X$ the \textbf{open cover coverage} of $X$.
\end{Ex}

\begin{Ex} \label{ex j open coverage}
Define a collection of families $j_\open$ on $\cat{Man}$ as follows: For $M \in \cat{Man}$, let $j_\open(M)$ denote the collection of open covers of $M$. Then $j_\open$ is a coverage. Indeed if $\{U_i \subseteq M \}$ is an open cover and $f: N \to M$ is a smooth map, then $\{ f^{-1}(U_i) \subseteq N \}$ is an open cover of $N$ satisfying (\ref{eqn coverage def}).

Now consider the following full subcategories
$$\cart \hookrightarrow \cat{Open} \hookrightarrow \cat{Man}.$$
Where $\cart$ is the full subcategory whose objects are cartesian spaces and $\cat{Open}$ is the full subcategory whose objects are diffeomorphic to open subsets of a cartesian space. The collection of families $j_\open$ can be restricted to $\Open$ and is clearly a coverage there as well.

Notice however that if we restrict $j_\open$ to $\cart$, and $U$ is a cartesian space, then an open cover $\{U_i \subseteq U \}$ is a covering family for $j_\open$ if and only if it is a cartesian open cover, otherwise it could not be a collection of morphisms in $\cart$. For $\cat{Man}$ and $\cat{Open}$ any open cover will do. However if $\{ U_i \subseteq U \}$ is a cartesian open cover and $f: V \to U$ is a smooth map, there is no reason that $\{ f^{-1}(U_i) \subseteq V \}$ will be a cartesian open cover. However as we will see in Example \ref{ex j good coverage}, every open cover can be refined by a cartesian open cover, and thus $j_\open$ is indeed a coverage on $\cart$.
\end{Ex}

\begin{Ex} \label{ex j good coverage}
Define a collection of families $j_\good$ on $\cat{Man}$ as follows: For $M \in \cat{Man}$, let $j_\good(M)$ denote the collection of good open covers as in Definition \ref{def cart, good covers, parametrizations} of $M$. Let us show that the good covers form a coverage. If $\{U_i \subseteq M \}$ is a good cover and $g: N \to M$ a smooth map, then $\{ g^{-1}(U_i) \subseteq N \}$ is an open cover, but not necessarily good. By \cite[Corollary 5.2]{bott1982differential}, this open cover can be refined by a good open cover $\{W_k \subseteq N \}$ so that for every $W_k$ in the good open cover, there exists a $U_i$ such that $W_k \subseteq g^{-1}(U_i)$, and thus the following diagram commutes:
\begin{equation*}
\begin{tikzcd}
	{W_k} & {g^{-1}(U_i)} & {U_i} \\
	N && M
	\arrow["g"', from=2-1, to=2-3]
	\arrow[hook, from=1-3, to=2-3]
	\arrow[hook, from=1-1, to=2-1]
	\arrow["{g|_{g^{-1}(U_i)}}", from=1-2, to=1-3]
	\arrow[hook, from=1-1, to=1-2]
\end{tikzcd}
\end{equation*}
Thus $j_\good$ is a coverage on $\cat{Man}$. Similarly it defines a coverage on $\cart$ and $\cat{Open}$.
\end{Ex}

\begin{Def} \label{def smooth sites}
Let $\cat{Smooth}$ denote a site of the form $(\site, j)$ with $\site \in \{ \cart, \cat{Open}, \cat{Man} \}$ and $j \in \{ j_\open, j_\good\}$. We will call any such site a \textbf{smooth site}.
\end{Def}

\begin{Ex} \label{ex coverage of subductions on Diff}
We note here that the collection of families $j_\sub$ of subductions on the category $\cat{Diff}$ of diffeological spaces is a coverage, because the pullback of a subduction is a subduction. We will not use this observation in this section, but it will come up in Section \ref{section resolutions of diffeological spaces} when we talk about diffeological categories.
\end{Ex}

Coverages are those collections of families with the least amount of structure with which we can define sheaves on $\site$.

\begin{Def}
A \textbf{presheaf} on a category $\site$ is a functor $F: \site^{op} \to \cat{Set}$. An element $x \in F(U)$ for an object $U \in \site$ is called a \textbf{section} over $U$. If $f: U \to V$ is a map in $\site$, and $x \in F(V)$,  then we sometimes denote $F(f)(x)$ by $x|_U$. If $\{r_i: U_i \to U \}_{i \in I}$ is a covering family, then a \textbf{matching family} is a collection $\{x_i \}_{i \in I}$, $x_i \in F(U_i)$, such that given a diagram in $\site$ of the form
\begin{equation*}
\begin{tikzcd}
	V & {U_j} \\
	{U_i} & U
	\arrow["{r_i}"', from=2-1, to=2-2]
	\arrow["{r_j}", from=1-2, to=2-2]
	\arrow["f"', from=1-1, to=2-1]
	\arrow["g", from=1-1, to=1-2]
\end{tikzcd}
\end{equation*}
then $F(f)(x_i) = F(g)(x_j)$ for all $i,j \in I$. An \textbf{amalgamation} $x$ for a matching family $\{ x_i \}$ is a section $x \in F(U)$ such that $x_i|_U = x$ for all $i$.
\end{Def}

\begin{Def}
Given a family of morphisms $r = \{ r_i: U_i \to U \}$ in a category $\site$, we say that a presheaf $F: \site^{op} \to \cat{Set}$ is a \textbf{sheaf on $r$} if every matching family $\{ s_i \}$ of $F$ over $r$ has a unique amalgamation. If $j$ is a coverage on a category $\site$, we call $F$ a \textbf{sheaf} on $(\site, j)$ if it is a sheaf on every covering family of $j$. Let $\cat{Sh}(\site, j)$ denote the full subcategory of $\cat{Pre}(\site)$ whose objects are sheaves on $(\site, j)$.
\end{Def}

\begin{Rem}
If $(\site, j)$ is a site that has pullbacks, then we can equivalently express the condition for $F$ being a sheaf as requiring that for every $U \in \site$ and every covering family $\{U_i \to U \} \in j(U)$, the diagram:
\begin{equation} \label{eqn sheaf condition}
    \begin{tikzcd}
	{F(U)} & {\prod_i F(U_i)} & {\prod_{i,j} F(U_i \times_U U_j)}
	\arrow[from=1-1, to=1-2]
	\arrow[shift right=1, from=1-2, to=1-3]
	\arrow[shift left=1, from=1-2, to=1-3]
\end{tikzcd}
\end{equation}
is an equalizer. This is how the sheaf condition is often presented in the literature.
\end{Rem}

\begin{Ex}
Given a smooth manifold $M$, the presheaf 
$$U \mapsto C^\infty(U,M)$$
which we denote by either $\u{M}$ or just $M$, is a sheaf on $\cat{Smooth}$. Another important example of a sheaf is
$$U \mapsto \Omega^k(U)$$ for any $k \geq 0$.
If $V$ is a cartesian space, we will denote its image under the Yoneda embedding by $yV$. This is the presheaf
$$U \mapsto C^\infty(U,V)$$
and as above is a sheaf. We call a site $(\site, j)$ where every representable presheaf is a sheaf \textbf{subcanonical}. It is not hard to see that $\cat{Smooth}$ is a subcanonical site.
\end{Ex}

We wish to single out those sheaves that are in some sense a set with extra structure.

\begin{Def}
A site $(\site, j)$ is \textbf{concrete} if:
\begin{enumerate}
    \item it is subcanonical,
    \item it has a terminal object $*$,
    \item the functor $\site(*,-): \site \to \cat{Set}$ is faithful, and
    \item for every covering family $\{U_i \to U \}$, the family of maps $\site(*, U_i) \to \site(*,U)$ is \textbf{jointly surjective}, namely the map $\coprod_i \site(*, U_i) \to \site(*,U)$ is surjective.
\end{enumerate}
\end{Def}

It is not hard to show that all of the smooth sites are concrete.

\begin{Def}
If $(\site, j)$ is a concrete site and $F$ is a presheaf, then we call $F(*)$ its \textbf{underlying set}, and for any $U \in \site$ there always exists a map
$$\phi_U: F(U) \to \cat{Set}(\site(*,U), F(*))$$
defined by $\phi_U(x) = \left( u \mapsto F(u)(x) \right)$. We say a sheaf $F$ is \textbf{concrete} if for every object $U \in \site$, the function $\phi_U$ is injective. Let $\cat{ConSh}(\site, j)$ denote the full subcategory of concrete sheaves on a concrete site $(\site,j)$.
\end{Def}

\begin{Ex}
For any smooth manifold $M$, the sheaf $\u{M}$ on $\cat{Smooth}$ is concrete. This is equivalent to saying that for every $U \in \cat{Smooth}$ the function
$$\phi_U: C^\infty(U,M) \to \cat{Set}(U(*),M(*))$$
is injective, which is the same thing as saying that the set of smooth maps from $U$ to $M$ is a subset of all set functions from $U$ to $M$.

Note that the sheaf $\Omega^k$ is not concrete on $\cat{Smooth}$. Indeed if $U \in \cat{Smooth}$, then $\phi_U$ takes the form
$$\phi_U: \Omega^k(U) \to \cat{Set}(U(*), \Omega^k(*))$$
but $\Omega^k(*) = \{ 0 \}$ is the zero vector space, thus $\cat{Set}(U(*), \{ 0 \}) = *$ is the singleton set. Since in general $\Omega^k(U)$ is nontrivial, this shows that $\Omega^k$ is not concrete.
\end{Ex}

\begin{Th}[{\cite[Proposition 24]{baez2009convenient}}] \label{thm diff' = consh(open)}
The category $\cat{Diff}'$ of classical diffeological spaces is equivalent to the category of concrete sheaves on $\cat{Open}$ with the open cover coverage,
$$\cat{Diff'} \simeq \cat{ConSh}(\cat{Open}, j_\open).$$
\end{Th}

However, the proof of \cite[Proposition 24]{baez2009convenient} can be applied nearly word for word to prove the following.

\begin{Th} \label{thm diff = consh(cart)}
The category $\cat{Diff}$ of diffeological spaces as defined in Definition \ref{def diffeological space} is equivalent to the category of concrete sheaves on $\cart$ with the good open cover coverage
$$\cat{Diff} \simeq \cat{ConSh}(\cart, j_\good).$$
\end{Th}

\begin{Rem}
In Appendix \ref{section comparison of site structures} we will prove $\cat{Diff} \simeq \cat{Diff}'$ by showing that $\cat{ConSh}(\cart, j_\good) \simeq \cat{ConSh}(\cat{Open}, j_\open)$.
\end{Rem}

Theorem \ref{thm diff = consh(cart)} allows us to make a perspective shift. Constructions made in $\cat{Diff}$ can be compared with already defined notions of sheaves. For example a differential $k$-form $\omega$ on a diffeological space $X$ \cite[Article 6.28]{iglesias2013diffeology}, is precisely a map
$$X \xrightarrow{\omega} \Omega^k$$
of sheaves on $\cart$. This viewpoint on diffeological spaces, namely as concrete sheaves on $\cart$, will also be the starting point for Section \ref{section smooth higher stacks}, where we consider the fully faithful embedding of presheaves into simplicial presheaves. Since concrete sheaves are in particular presheaves, this means that there is a fully faithful embedding of diffeological spaces into simplicial presheaves, where we have a powerful homotopy theory to leverage.

\section{Smooth Higher Stacks} \label{section smooth higher stacks}

\subsection{Model Structures on Simplicial Presheaves} \label{section model structures on simplicial presheaves}
For this section, we assume the reader is comfortable with simplicial homotopy theory as in \cite{goerss2012simplicial} and model categories as in \cite{hirschhorn2009model}.

\begin{Def}
Let $\cat{sSet}$ denote the category of simplicial sets, and $\cat{sPre}(\cart)$ the category whose objects are functors $X: \cart^{op} \to \cat{sSet}$ and whose morphisms are natural transformations. We call such functors \textbf{simplicial presheaves}.
\end{Def}

There is a fully faithful embedding $\cat{Set} \hookrightarrow \cat{sSet}$, which we denote by $S \mapsto {}^c S$, where $\left( {}^cS \right)_n = S$ for all $n \geq 0$, and all of the face and degeneracy maps are the identity on $S$. Similarly there is a fully faithful embedding $\pre(\cart) \hookrightarrow \spre(\cart)$, which we also denote by $F \mapsto {}^c F$, such that $\left( {}^cF \right)(U) = {}^c F(U)$ for all $U \in \cart$. We call simplicial presheaves of this form \textbf{discrete simplicial presheaves}. This functor has a left adjoint $\pi_0: \spre(\cart) \to \pre(\cart)$, defined objectwise by
$$ (\pi_0 X)(U) = \text{coeq} \left( \begin{tikzcd}
	{X(U)_0} & {X(U)_1}
	\arrow[shift right=2, from=1-2, to=1-1]
	\arrow[shift left=2, from=1-2, to=1-1]
\end{tikzcd} \right),$$
and a right adjoint $(-)_0: \spre(\cart) \to \pre(\cart)$ defined objectwise by $X_0(U) = X(U)_0$. Note that limits and colimits in $\spre(\cart)$ are computed objectwise.

There is also a functor $(-)_c : \cat{sSet} \to \spre(\cart)$ defined objectwise by $X_c(U) = X$ for every $U \in \cart$. We call simplicial presheaves of this form \textbf{constant simplicial presheaves}. 
The category of simplicial presheaves on $\cart$ is simplicially enriched. Let $X$ and $Y$ be simplicial presheaves, then let $\u{\spre(\cart)}(X,Y)$ denote the simplicially-enriched Hom, defined degreewise by
$$\u{\spre(\cart)}(X,Y)_n = \spre(\cart)(X \times \Delta^n_c, Y).$$

Compare this with the simplicial mapping space for simplicial sets, namely if $X$ and $Y$ are simplicial sets, then let $\u{\cat{sSet}}(X,Y)$ denote the simplicial set defined degreewise by $\u{\cat{sSet}}(X,Y)_n = \cat{sSet}(X \times \Delta^n, Y).$

If $K$ is a simplicial set and $X$ is a simplicial presheaf, then let $X^K$ denote the simplicial presheaf which is defined objectwise by $(X^K)(U) = \u{\cat{sSet}}(K, X(U))$. Then $\spre(\cart)$ is tensored and cotensored over $\cat{sSet}$ in the sense that for simplicial presheaves $X$ and $Y$ and simplicial set $K$, there is the following natural isomorphism
\begin{equation*}
    \u{\spre(\cart)}(X \times K_c, Y) \cong \u{\spre(\cart)}(X, Y^K).
\end{equation*}

The category $\cat{sPre}(\cart)$ admits many model structures. Here we will discuss two of them. Say a map $f: X \to Y$ of simplicial presheaves is an \textbf{objectwise weak equivalence} if $f: X(U) \to Y(U)$ is a weak equivalence of simplicial sets for every $U \in \cart$. Similarly a \textbf{objectwise fibration} is a map $f: X \to Y$ of simplicial presheaves such that $f: X(U) \to Y(U)$ is a Kan fibration of simplicial sets for every $U \in \cart$. 

\begin{Th}[{\cite[Page 314]{bousfield1972homotopy}}]
There is a cofibrantly generated, simplicial model structure, which we call the \textbf{projective model structure} or \textbf{Bousfield-Kan model structure} on $\spre(\cart)$, whose weak equivalences are the objectwise weak equivalences, and whose fibrations are the objectwise fibrations.
\end{Th}

\begin{Rem}
In fact, the projective model structure makes $\spre(\cart)$ a combinatorial model category, see \cite[Section A.2.6]{lurie2009higher}.
\end{Rem}

\begin{Rem}
There is a Quillen equivalent model structure on simplicial presheaves where the cofibrations and weak equivalences are objectwise, which is called the injective or Heller model structure. See \cite{blander2001local} for an overview of the different model structures on simplicial presheaves.
\end{Rem}

As is often the case with model structures, while the descriptions of weak equivalences and fibrations in the projective model structure are convenient, the cofibrations of the projective model structure are less simple to describe. However, the following result gives a sufficient condition for a simplicial presheaf to be cofibrant. 

\begin{Th}[{\cite[Corollary 9.4]{dugger2001universal}}] 
A simplicial presheaf $X$ is cofibrant in the projective model structure on simplicial presheaves if
\begin{enumerate}
    \item $X$ is degreewise a coproduct of representables, i.e. $X_n = \coprod_{i \in I} yU_i$ for every $n \geq 0$,
    \item $X$ is split, in the sense that as a functor $X: \site^{op} \to \cat{sSet}$ it factors through the category $\cat{sSet}_{nd}$ whose objects are simplicial sets and whose morphisms are those maps of simplicial sets that map non-degenerate simplices to non-degenerate simplices.
\end{enumerate}
We say that $X$ is a \textbf{projective cofibrant} simplicial presheaf.
\end{Th}

\begin{Cor}
If $U \in \cart$, then ${}^c yU$ is a projective cofibrant simplicial presheaf on $\cart$.
\end{Cor}

\begin{Ex}
If $M$ is a finite dimensional smooth manifold, and $\mathcal{U} = \{ U_i \}_{i \in I}$ is a good open cover, then consider the simplicial presheaf $\check{C}(\mathcal{U})$ defined in degree $n$ by
$$\check{C}(\mathcal{U})_n = \prod_{i_0, \dots, i_n} y\left( U_{i_0} \cap \dots \cap U_{i_n} \right),$$
with face and degeneracy maps given by inclusions of open sets. Since $\mathcal{U}$ is a good open cover, $\check{C}(\mathcal{U})$ is a projective cofibrant simplicial presheaf. We call it the \textbf{\v{C}ech nerve} on $\mathcal{U}$. There is a canonical map 
$$\check{C}(\mathcal{U}) \xrightarrow{\pi} {}^c M,$$
of simplicial presheaves on $\cart$. However this map is not an objectwise weak equivalence in general.
\end{Ex}

If $U \in \cart$, and $\mathcal{U} = \{U_i \}_{i \in I}$ is a good open cover of $U$, then we can consider the canonical map
$$\check{C}(\mathcal{U}) \xrightarrow{\pi} {}^c yU.$$
Let $\check{C}$ denote the class of such maps as $U$ varies over $\cart$ and $\mathcal{U}$ varies over good open covers of $U$.

\begin{Def}
Let $\cechpre(\cart)$ denote the left Bousfield localization of the projective model structure on $\spre(\cart)$ at the class of maps $\check{C}$. We call this the \textbf{\v{C}ech model structure} on $\spre(\cart)$.
\end{Def}

The \v{C}ech model structure is described in greater detail in \cite[Appendix A]{dugger2004hypercovers}. Since $\cechpre(\cart)$ is a left Bousfield localization of the projective model structure, it inherits the same cofibrations, and therefore cofibrant objects. We call its weak equivalences the \v{C}ech weak equivalences. Note that any objectwise weak equivalence of simplicial presheaves will be a \v{C}ech weak equivalence.

We call the fibrant objects of this model structure \textbf{$\infty$-stacks} on $\cart$. They can be characterized as follows.

\begin{Prop} \label{prop infinity stack}
A simplicial presheaf $X$ on $\cart$ is an $\infty$-stack on $\cart$ if and only if it is projective fibrant (objectwise a Kan complex), and if for every $U \in \cart$ and good open cover $\mathcal{U}$ of $U$, the map
\begin{equation} \label{eqn cech descent}
    \u{\spre(\cart)}(yU, X) \to \u{\spre(\cart)}(\check{C}(\mathcal{U}), X),
\end{equation}
is a weak equivalence of simplicial sets. We say that $X$ satisfies \textbf{\v{C}ech descent}.
\end{Prop}

\begin{proof}
This follows from the definition of left Bousfield localization.
\end{proof}

By a simplicially-enriched version of the Yoneda Lemma, $\u{\spre(\cart)}(yU,X) \cong X(U)$. We wish to better understand the right hand side of (\ref{eqn cech descent}). To do this we will exploit the following result.

\begin{Lemma}
Let $X$ be a simplicial presheaf. Then
$$ X \cong \int^{[n] \in \cat{\Delta}^{op}} \Delta^n_c \times {}^c X_n,$$
where the colimit is taken in the category of simplicial presheaves, $\Delta^n_c$ is the constant simplicial presheaf on $\Delta^n$ and ${}^c X_n$ is the discrete simplicial presheaf on the presheaf $X_n$.
\end{Lemma}

\begin{proof}
This follows from the corresponding fact for simplicial sets \cite[Lemma I.2.1]{goerss2012simplicial}. 
\end{proof}

Thus we have
$$\check{C}(\mathcal{U}) \cong \int^{[n] \in \cat{\Delta}^{op}} \Delta^n_c \times \coprod_{i_0 \dots i_n} yU_{i_0 \dots i_n}.$$

This implies that
\begin{equation} \label{eqn spre as Tot}
    \begin{aligned}
    \u{\spre(\cart)}(\check{C}(\mathcal{U}), X) & \cong \int_{[n]} \u{\spre(\cart)}(\Delta^n_c \times \coprod_{i_0 \dots i_n} yU_{i_0 \dots i_n}, X) \\
    & \cong \int_{[n]} \prod_{i_0 \dots i_n} \u{\spre(\cart)}\left( yU_{i_0 \dots i_n}, X^{\Delta^n}\right) \\
    & \cong \int_{[n]} \prod_{i_0 \dots i_n} \u{\cat{sSet}} \left( \Delta^n, X(U_{i_0 \dots i_n}) \right) \\
    & \cong \int_{[n]} \u{\cat{sSet}} \left( \Delta^n, \prod_{i_0 \dots i_n} X(U_{i_0 \dots i_n}) \right).
    \end{aligned}
\end{equation}
This kind of limit is special enough to have its own name.

\begin{Def} \label{def totalization}
Let $F$ be a cosimplicial simplicial set, namely a functor $F: \cat{\Delta} \to \cat{sSet}$. Then let $\cat{Tot}(F)$ denote the simplicial set given by the end
$$\cat{Tot}(F) = \int_{[n] \in \cat{\Delta}} \u{\cat{sSet}}(\Delta^n, F^n),$$
where $\u{\cat{sSet}}(X,Y)$ denotes the simplicial mapping space between two simplicial sets $X$ and $Y$, namely $\u{\cat{sSet}}(X,Y)_n = \cat{sSet}(X \times \Delta^n, Y)$. For a cosimplicial simplicial set $F$, $\cat{Tot}(F)$ is often called the \textbf{total object} or \textbf{totalization} of $F$.
\end{Def}

A more convenient way of looking at $\cat{Tot}(F)$ is as the simplicial mapping space $\u{\cat{csSet}}(\Delta, F)$, where $\Delta$ denotes the cosimplicial simplicial set $[m] \mapsto \Delta^m$. In other words, an $n$-simplex of the simplicial set $\cat{Tot}(F)$ is a map of cosimplicial simplicial sets $\Delta \times \Delta^n \to F$. The full data of such a map is a commutative diagram of the form
\begin{equation*}
\begin{tikzcd}
	{\Delta^n} & {\Delta^1 \times \Delta^n} & {\Delta^2 \times \Delta^n} & \dots \\
	{F^0} & {F^1} & {F^2} & \dots
	\arrow[from=1-1, to=2-1]
	\arrow[shift left=2, from=1-1, to=1-2]
	\arrow[shift right=2, from=1-1, to=1-2]
	\arrow[shift left=3, from=1-2, to=1-3]
	\arrow[shift right=3, from=1-2, to=1-3]
	\arrow[from=1-2, to=1-3]
	\arrow[from=1-2, to=2-2]
	\arrow[from=1-3, to=2-3]
	\arrow[shift left=2, from=2-1, to=2-2]
	\arrow[shift right=2, from=2-1, to=2-2]
	\arrow[shift left=3, from=2-2, to=2-3]
	\arrow[shift right=3, from=2-2, to=2-3]
	\arrow[from=2-2, to=2-3]
\end{tikzcd}
\end{equation*}
where each arrow is a map of simplicial sets, and we've hidden the codegeneracy maps for clarity in the diagram.

Thus $X$ is an $\infty$-stack if and only if $X$ is projective fibrant, and the canonical map
\begin{equation}
    X(U) \to \cat{Tot} \left( X(\check{C}(\mathcal{U})) \right)
\end{equation}
is a weak equivalence of simplicial sets, where $X(\check{C}(\mathcal{U}))$ is the cosimplicial simplicial set defined degreewise by $X(\check{C}(\mathcal{U}))_n = \prod_{i_0 \dots i_n} X(U_{i_0 \dots i_n}).$ This concrete description has a pleasing abstract description as well.

\begin{Prop} \label{prop tot gives holim on simplicial presheaves}
If $X$ is a projective fibrant (objectwise Kan) simplicial presheaf, $U \in \cart$ and $\mathcal{U}$ is a good cover of $U$, then
\begin{equation}
\cat{Tot} \left( X(\check{C}(\mathcal{U})) \right) \simeq \text{holim}_{[n] \in \cat{\Delta}} \, X(\check{C}(\mathcal{U}))_n
\end{equation}
where the right hand side is the homotopy limit of the cosimpicial diagram of simplicial sets $X(\check{C}(\mathcal{U}))_n$ taken in the Quillen model structure on simplicial sets.
\end{Prop}

\begin{proof}
By \cite[Theorem 18.7.4]{hirschhorn2009model}, $\cat{Tot} \left( X(\check{C}(\mathcal{U})) \right) \to \text{holim}_n \, X(\check{C}(\mathcal{U}))_n$ is a weak equivalence if $X(\check{C}(\mathcal{U}))$ is a Reedy fibrant cosimplicial simplicial set. By \cite[Lemma C.5]{glass2022chern}, if $X$ is projective fibrant, then $X(\check{C}(\mathcal{U}))$ is Reedy fibrant.
\end{proof}

Thus by Proposition \ref{prop tot gives holim on simplicial presheaves}, if $X$ is a projective fibrant simplicial presheaf, then it is an $\infty$-stack if and only if the canonical map
\begin{equation} \label{eq cech descent using holim}
    X(U) \to \text{holim}_{\cat{\Delta}} \left(\begin{tikzcd}
	{\prod_i X(U_i)} & {\prod_{i,j} X(U_{ij})} & {\prod_{i,j,k}X(U_{ijk})} & {}
	\arrow[shift left=2, from=1-1, to=1-2]
	\arrow[shift right=2, from=1-1, to=1-2]
	\arrow[from=1-2, to=1-1]
	\arrow[shift left=4, from=1-2, to=1-3]
	\arrow[from=1-2, to=1-3]
	\arrow[shift right=4, from=1-2, to=1-3]
	\arrow[shift left=2, from=1-3, to=1-2]
	\arrow[shift right=2, from=1-3, to=1-2]
	\arrow["\dots"{description}, draw=none, from=1-3, to=1-4]
\end{tikzcd}\right)
\end{equation}
is a weak equivalence of simplicial sets.

If $X = {}^c F$ is a presheaf of sets, then $\u{\spre(\cart)}(yU,{}^c F) \cong F(U)$, and 
$$\u{\spre(\cart)}(\check{C}(\mathcal{U}), {}^c F) \cong \pre(\cart)(\pi_0 \check{C}(\mathcal{U}), F) \cong \text{eq} \left( \prod_i F(U_i) \rightrightarrows \prod_{i,j} F(U_{ij}) \right).$$
The right hand side is the usual equalizer one sees in the definition of a sheaf of sets. Note that if $f: {}^c X \to {}^c Y$ is a map of discrete simplicial sets, then $f$ is a weak equivalence if and only if it is an isomorphism of sets. Thus ${}^c F$ is an $\infty$-stack if and only if the canonical map
$$F(U) \to \text{eq} \left( \prod_i F(U_i) \rightrightarrows \prod_{i,j} F(U_{ij}) \right),$$
is an isomorphism of sets for every $U \in \cart$ and good open cover $\mathcal{U}$. In other words, for discrete simplicial presheaves, being an $\infty$-stack is equivalent to being a sheaf.

Suppose that $G: \cart^{op} \to \cat{Gpd}$ is a (strict) presheaf of groupoids, then it is well known that $G$ is a stack, in the classical sense, if and only if the map
$$G(U) \to \text{holim} \left(\begin{tikzcd}
	{\prod_i G(U_i)} & {\prod_{i,j} G(U_{ij})} & {\prod_{i,j,k} G(U_{ijk})}
	\arrow[shift left=2, from=1-1, to=1-2]
	\arrow[shift right=2, from=1-1, to=1-2]
	\arrow[from=1-2, to=1-1]
	\arrow[shift left=4, from=1-2, to=1-3]
	\arrow[from=1-2, to=1-3]
	\arrow[shift right=4, from=1-2, to=1-3]
	\arrow[shift left=2, from=1-3, to=1-2]
	\arrow[shift right=2, from=1-3, to=1-2]
\end{tikzcd} \right),$$
is an equivalence of groupoids, where the right hand side is a homotopy limit of groupoids as described in \cite[Section I.1.7]{carchedi2011categorical}. Now consider the nerve functor $N: \cat{Cat} \to \cat{sSet}$. Then $NG$ is a simplicial presheaf that will be projective fibrant, and it will be an $\infty$-stack if and only if $G$ is a stack in the classical sense. Thus the notion of $\infty$-stack simultaneously generalizes the notion of sheaf and stack, and provides all of the power the homotopy theory of simplicial sets has to offer.

The main example of ($\infty$-)stack we will consider in this paper is the following.

\begin{Ex} \label{ex delooping stack}
Suppose $G$ is a group object in $\cat{Sh}(\cart)$, which we will call a sheaf of groups. Consider the presheaf of groupoids on $\cart$ that sends a $U \in \cart$ to the groupoid
$$\bold{B}G(U) \coloneqq [ G(U) \rightrightarrows * ],$$
where both source and target maps are the unique map to the singleton set. Thus there is a single object in this groupoid, and for every $s \in G(U)$, there is an isomorphism from the unique object to itself. We visualize morphisms in $\B G(U)$ with diagrams like the following
\begin{equation*}
    * \xrightarrow{g} *
\end{equation*}
We define composition in this groupoid using the \textbf{opposite} of multiplication in $G$. In other words we have
\begin{equation*}
    * \xrightarrow{g} * \xrightarrow{h} * \qquad = \qquad  * \xrightarrow{hg} *
\end{equation*}
This convention might seem strange, but we have chosen it to agree with \cite[Section 1.2.5.1]{schreiber2013differential}, which was consulted often in the formulation of this paper. Note that in other conventions \cite{fiorenza2011cech} we must simply swap $G$ with $G^{op}$, the opposite of the presheaf of groupoids, or the presheaf of groupoids with one object and morphisms the sheaf of groups with multiplication defined by $g \cdot^{op} h = hg$. It is important to note that the key equations (\ref{eqn diff cocycle condition}) and (\ref{eqn map of diff cocycles}) of Section \ref{section diff principal bundles} will be altered by changing this convention.

Now if $G$ is in particular a Lie group, then we can consider it as a sheaf of groups on $\cart$. Then $\bold{B}G(U)$ will be the groupoid $[C^\infty(U,G) \rightrightarrows * ]$. Now since all principal $G$-bundles are trivial on Cartesian spaces, it is easy to see that there is an objectwise equivalence of groupoids
$$[C^\infty(U, G) \rightrightarrows *] \simeq \text{Prin}_G(U),$$
where $\text{Prin}_G(U)$ denotes the groupoid of principal $G$-bundles on $U$. Indeed, every object of the right hand groupoid is isomorphic to the trivial bundle $U \times G \to U$, and the automorphisms of a trivial bundle are in bijection with maps $U \to G$. In \cite[Section I.2]{carchedi2011categorical} it is proven that $\text{Prin}_G$ is a stack (in the classical sense) on $\cat{Man}$ and $\cart$. The argument above proves that $\bold{B}G$ is also a stack (in the classical sense) on $\cart$. However $\bold{B}G$ is not a stack on $\cat{Man}$. If we take the nerves of these presheaves of groupoids $N \bold{B} G$ and $N \text{Prin}_G$, then since they are nerves of classical stacks on $\cart$, they will be $\infty$-stacks on $\cart$, and they are objectwise weak equivalent as simplicial presheaves. See \cite[Section 3.2]{fiorenza2011cech} for more details.

We often drop the nerve $N$ from the notation of our $\infty$-stacks, and we call $\bold{B}G$ the \textbf{delooping stack} of $G$.
\end{Ex}

\begin{Rem} \label{rem about pseudofunctors}
Note that in the above example, strictly speaking $\cat{Prin}_G$ is a not a presheaf of groupoids, because if $U \xrightarrow{f} V \xrightarrow{g} W$ is a pair of composable morphisms in $\cart$, then we obtain functors $\cat{Prin}_G(W) \xrightarrow{g^*} \cat{Prin}_G(V) \xrightarrow{f^*} \cat{Prin}_G(U)$ given by pulling back the principal bundles. However if $P \in \cat{Prin}_G(W)$, then $(gf)^* P \neq f^* g^* P$, but there is an isomorphism $(gf)^* P \cong f^* g^* P$. Thus $\cat{Prin}_G$ is called a pseudofunctor $\cat{Prin}_G: \cart^{op} \to \cat{Gpd}$, where $\cat{Gpd}$ is the $2$-category of groupoids. There is an elegant theory \cite[Chapter 10]{johnson20202dimensional} relating pseudofunctors with categories fibered in groupoids, both of which can be used to develop the theory of stacks of groupoids \cite{vistoli2007notes}. However, the homotopy theory of presheaves of groupoids, pseudofunctors, and categories fibered in groupoids are all equivalent in the sense of \cite[Corollary 4.3]{hollander2007homotopy}. Furthermore, the notion of being a $1$-stack is independent across the three models. Thus in what follows we elect to use presheaves of groupoids, as they are the simplest to connect with the theory of simplicial presheaves by applying the nerve functor objectwise.
\end{Rem}

One of the most useful aspects of simplicial model categories is being able to define homotopically invariant mapping spaces.

\begin{Def}
If $X$ and $A$ are simplicial presheaves, then define
\begin{equation}
    \R \Hom(X,A) \coloneqq \u{\spre(\cart)}(QX,RA)
\end{equation}
where $QX$ is a cofibrant replacement for $X$ and $RA$ is a fibrant replacement of $A$ in the \v{C}ech model structure. We call $\R \Hom(X,A)$ the \textbf{derived mapping space} of $X$ and $A$.
\end{Def}

A key property of derived mapping spaces is their invariance under weak equivalence. Indeed suppose there is a \v{C}ech weak equivalence $f: X \to X'$, then the canonical map
$$\R \Hom(f, A) : \R \Hom(X',A) \to \R \Hom(X, A)$$
is a weak equivalence of simplicial sets, similarly for a \v{C}ech weak equivalence $g: A \to A'$, see \cite[Chapter 17]{hirschhorn2009model}.

\begin{Prop}[{\cite[Page 23]{fiorenza2011cech}}] \label{prop isomorphism classes of principal bundles on a manifold}
Given a Lie group $G$, and a finite dimensional smooth manifold $M$, there is a weak equivalence of simplicial sets
\begin{equation}
 N \text{Prin}_G(M) \simeq \R \Hom(M, \bold{B} G),   
\end{equation}
where $N \text{Prin}_G(M)$ denotes nerve of the groupoid of principal $G$-bundles on $M$.
\end{Prop}

Now if $G$ is a diffeological group, and $X$ is a diffeological space, we can consider them both as simplicial presheaves on $\cart$, and compute $\R \Hom(X,\bold{B} G)$. It would be hoped that this would in some way reproduce diffeological principal $G$-bundles, in analogy to Proposition \ref{prop isomorphism classes of principal bundles on a manifold}. One main goal of this paper is to prove that this is indeed the case. But first we must investigate $\R \Hom(X, \bold{B}G)$. If $X$ was cofibrant, and $\bold{B}G$ were fibrant in the \v{C}ech model structure, then $\R \Hom(X, \bold{B} G)$ would be computable. However diffeological spaces are not projective cofibrant in general (though cartesian spaces are). Thus we must find a projective cofibrant simplicial presheaf $QX$ which is \v{C}ech weak equivalent to $X$. This will be the subject of Section \ref{section resolutions of diffeological spaces}.

However, it is indeed the case that $\bold{B}G$ is fibrant, thanks to the following wonderful theorem.

\begin{Th}[{\cite[Lemma 3.3.29]{sati2022equivariant}, \cite[Proposition 4.13]{pavlov2022numerable}}] \label{th delooping sheaf of groups is infinity stack}
Let $G$ be a sheaf of groups on $\cart$. Then $\bold{B}G$ is an $\infty$-stack on $\cart$.
\end{Th}

Thus if $G$ is a diffeological group, then it is in particular a sheaf of groups on $\cart$, and therefore $\bold{B}G$ is an $\infty$-stack.

\subsection{Resolutions of Diffeological Spaces} \label{section resolutions of diffeological spaces}
Here we discuss three ways of ''resolving" a diffeological space into a diffeological category. One of which, $QX$ comes forth immediately from the projective model structure on simplicial presheaves. The other two, which we denote $\check{C}(X)$ and $B//M$, appear in \cite{krepski2021sheaves} and \cite{iglesias2020vcech} respectively, and are interesting in their own right. We compare these three resolutions as diffeological categories, and examine their resulting notions of diffeological \v{C}ech cohomology in Section \ref{section diffeological cech cohomologies}.

Let us start by describing the resolution $QX$ for a diffeological space $X$. Since $\cechpre(\cart)$ is a combinatorial model category, a cofibrant replacement functor $Q$ exists. However, we are in the lucky situation that there is a cofibrant replacement functor $Q$ with a relatively simple form.

\begin{Lemma}[{\cite[Lemma 2.7]{dugger2001universal}}] \label{lem cofibrant replacement of diff space}
Given a diffeological space $X$, thought of as a discrete simplicial presheaf on $\cart$, its cofibrant replacement $QX$ is given by the simplicial presheaf
\begin{equation}
    QX = \int^{[n] \in \cat{\Delta}} \Delta^n_c \times \left( \coprod_{U_{p_n} \to \dots \to U_{p_0} \to X} yU_{p_n} \right)
\end{equation}
\end{Lemma}

Let us examine this coend formula in more detail. In degree $k$, we have 
\begin{equation}
(QX)_k  \cong \coprod_{(f_{k-1}, \dots, f_0) : U_{p_k} \to \dots \to U_{p_0}} yU_{p_k}
\end{equation}
where the coproduct is taken over the set $(N\cat{Plot}(X))_k$, namely the set of $k$ composable morphisms in the category of plots over $X$. We will let $N_k \coloneqq (N\cat{Plot}(X))_k$. 

The face maps are given as follows:
\begin{equation}
d_i(x_{p_k}, f_{k-1}, f_{k-2}, \dots, f_0) = 
\begin{cases}
    (f_{k-1}(x_{p_k}), f_{k-2}, \dots, f_0) & i = 0 \\
    (x_{p_k}, f_{k-1}, \dots, f_{k-i-1} f_{k-i}, \dots, f_0) & 0 < i < k \\
    (x_{p_k}, f_{k-1}, \dots, f_1) & i = k.
\end{cases}
\end{equation}
Degeneracies insert identity maps.

For convenience, we will denote the coproduct over all plots as
$$ B \coloneqq (QX)_0 = \coprod_{p_0 \in \cat{Plot}(X)} U_{p_0}.$$
Notice that there is a canonical map
$$B \xrightarrow{\pi} X$$
given by $\pi(p_0,x_{p_0}) = p_0(x_{p_0})$.

This map induces a map $\pi: QX \to {}^c X$ of simplicial presheaves and \cite[Lemma 2.7]{dugger2001universal} proves that this is an objectwise weak equivalence, and thus a \v{C}ech weak equivalence. By construction we also have the following isomorphism of presheaves on $\cart$.

\begin{Lemma}
If $X$ is a diffeological space, then the map $\pi: QX \to {}^c X$ induces an isomorphism of presheaves on $\cart$
\begin{equation} \label{eqn pi0 of QX is X}
    \pi_0 QX \cong X,
\end{equation}
where $\pi_0: \spre(\cart) \to \pre(\cart)$ is defined in Section \ref{section smooth higher stacks}.
\end{Lemma}

\begin{proof}
This follows from the fact that every presheaf is a colimit of representables, see the discussion above \cite[Lemma 2.7]{dugger2001universal}.
\end{proof}

In low degrees, this simplicial presheaf/simplicial diffeological space looks like:
% https://q.uiver.app/?q=WzAsNCxbMCwwLCJCID0gXFxjb3Byb2Rfe1Vfe3BfMH0gXFxpbiBcXGNhdHtQbG90fShYKX0gVV97cF8wfSJdLFsxLDAsIlxcY29wcm9kX3tmXzA6IFVfe3BfMX0gXFx0byBVX3twXzB9fSBVX3twXzF9Il0sWzIsMCwiXFxjb3Byb2Rfe1Vfe3BfMn0gXFx4cmlnaHRhcnJvd3tmXzF9IFVfe3BfMX0gXFx4cmlnaHRhcnJvd3tmXzB9IFVfe3BfMH19IFVfe3BfMn0iXSxbMywwLCJcXGRvdHMiXSxbMSwwLCIiLDIseyJvZmZzZXQiOi0yfV0sWzEsMCwiIiwwLHsib2Zmc2V0IjoyfV0sWzAsMV0sWzIsMSwiIiwxLHsib2Zmc2V0IjotNH1dLFsyLDEsIiIsMSx7Im9mZnNldCI6NH1dLFsyLDFdLFsxLDIsIiIsMSx7Im9mZnNldCI6Mn1dLFsxLDIsIiIsMSx7Im9mZnNldCI6LTJ9XV0=
\begin{equation*}\begin{tikzcd}
	{B = \coprod_{p_0 \in \cat{Plot}(X)} U_{p_0}} & {\coprod_{f_0: U_{p_1} \to U_{p_0}} U_{p_1}} & {\coprod_{U_{p_2} \xrightarrow{f_1} U_{p_1} \xrightarrow{f_0} U_{p_0}} U_{p_2}} & \dots
	\arrow[shift left=2, from=1-2, to=1-1]
	\arrow[shift right=2, from=1-2, to=1-1]
	\arrow[from=1-1, to=1-2]
	\arrow[shift left=4, from=1-3, to=1-2]
	\arrow[shift right=4, from=1-3, to=1-2]
	\arrow[from=1-3, to=1-2]
	\arrow[shift right=2, from=1-2, to=1-3]
	\arrow[shift left=2, from=1-2, to=1-3]
\end{tikzcd}
\end{equation*}
where the maps $f_i: U_{p_{i+1}} \to U_{p_i}$ are understood to be morphisms in the plot category $\cat{Plot}(X)$.

In fact, $QX$ is completely determined by $QX_1$ and $QX_0$ in the following sense. Let $N: \cat{DiffCat} \to \cat{sDiff}$ denote the nerve functor from diffeological categories (categories internal to the category of diffeological spaces) to the category of simplicial diffeological spaces, defined degreewise by
$$NC_k = C_1 \times_{t, C_0, s} C_1 \times_{t, C_0, s} \dots \times_{t, C_0, s} C_1,$$
where the iterated pullback is taken $k$-times, where each $C_1 \times_{t, C_0, s } C_1$ denotes the pullback with respect to the target and source maps respectively.

For a diffeological space $X$, the first two spaces and maps between them in $QX$, namely $[QX_1 \rightrightarrows B]$ forms a diffeological category. It also is a presheaf of categories. The source, target and unit maps are defined by the simplicial structure, but we recall their definitions here. The source map $s: QX_1 \to B$ is defined by $s(x_{p_1}, f_0) = x_{p_1}$ and its target map $t: QX_1 \to B$ is defined by $t(x_{p_1}, f_0) = f_0(x_{p_1})$. The unit map $u: B \to QX_1$ is defined by $u(x_p) = (x_p, 1_{U_p})$. The composition map $c: QX_1 \times_B QX_1 \to QX_1$ is defined by $c([x_{p_2}, f_1], [x_{p_1}, f_0]) = (x_{p_2}, f_0 \circ f_1)$. With this structure it is not hard to see that $[QX_1 \rightrightarrows B]$ is a diffeological category/presheaf of categories. In fact $QX$ is completely determined by this diffeological category in the following sense.

\begin{Prop} \label{prop QX is nerve of a diff category}
If $X$ is a diffeological space, then
\begin{equation}
    QX \cong N[QX_1 \rightrightarrows QX_0],
\end{equation}
where we are thinking of $QX$ as a simplicial diffeological space or a simplicial presheaf.
\end{Prop}

\begin{proof}
Let $\varphi: QX_k \to QX_1 \times_B QX_1 \times_B \dots \times_B QX_1$ be the map defined as follows. Suppose that $(x_{p_k}, f_{k-1}, \dots, f_0) \in QX_k$. By induction, define $x_{p_{k-1}} = f_{k-1}(x_{p_k})$ and $x_{p_{k-n}} = f_{k-n}(x_{p_{k - n + 1}})$ for $1 < n \leq k$. Then set
$$\varphi(x_{p_k}, f_{k-1}, \dots, f_0) = ([x_{p_k}, f_{k-1}], [x_{p_{k-1}}, f_{k-2}], \dots, [x_{p_1}, f_0]).$$
This map is smooth, as it is built out of projection maps. Now define $\psi: QX_1 \times_B QX_1 \times_B \dots \times_B QX_1 \to QX_k$ as follows. A point of $QX_1 \times_B QX_1 \times_B \dots \times_B QX_1$ is a collection of pairs $\{ [x_{p_n}, f_{n-1}] \}_{1 \leq n \leq k}$ such that $f_{n-1}(x_{p_n}) = x_{p_{n-1}}$. Thus set
$$\psi([x_{p_k}, f_{k-1}], [x_{p_{k-1}}, f_{k-2}], \dots, [x_{p_1}, f_0]) = (x_{p_k}, f_{k-1}, \dots, f_0).$$
It is not hard to see that this map is smooth, and that $\varphi$ and $\psi$ are two-sided inverses for each other. 
\end{proof}

\begin{Lemma} \label{lem diff space X as quotient of QX}
If $X$ is a diffeological space, then we can consider the coequalizer in $\cat{Diff}$ of $[QX_1 \rightrightarrows QX_0]$ and this is isomorphic to $X$, namely
\begin{equation}
    X \cong \text{coeq} \left( \coprod_{U_{p_1} \xrightarrow{f_0} U_{p_0}} U_{p_1} \rightrightarrows \coprod_{p_0 \in \cat{Plot}(X)} U_{p_0} \right).
\end{equation}
\end{Lemma}

\begin{proof}
This is just a restatement of Lemma \ref{lem diff space is colimit of plots}.
\end{proof}

\begin{Rem} \label{rem coequalizer in diff same as pi0}
Note that the coequalizer given in (\ref{eqn pi0 of QX is X}) is taken in the category $\cat{Pre}(\cart)$, which has different colimits than $\cat{ConSh}(\cart)$. Therefore it does not immediately imply Lemma \ref{lem diff space X as quotient of QX}. However, by combining the two results we have proven that $\pi_0 QX$ is isomorphic to the coequalizer of $QX_1 \rightrightarrows QX_0$ in the category of diffeological spaces.
\end{Rem}

Now as discussed in Section \ref{section model structures on simplicial presheaves}, if $X$ is a diffeological space and $G$ is a diffeological group, then we can consider the simplicial set $\R \Hom(X, \bold{B}G).$
By Theorem \ref{th delooping sheaf of groups is infinity stack}, we know that $\bold{B}G$ is fibrant, thus
\begin{equation*}
    \R \Hom(X, \bold{B}G) = \u{\spre(\cart)}(QX, \bold{B}G) \cong \u{\spre(\cart)}(N[QX_1 \rightrightarrows B], \, N[C^\infty(-,G) \rightrightarrows *]).
\end{equation*}

In Section \ref{section principal diffeological bundles as principal infinity bundles} we will show that this simplicial set is weak equivalent to the nerve of the groupoid of diffeological principal $G$-bundles on $X$.

The fact that $QX$ is the nerve of a diffeological category is interesting, as it allows us to compare it with other diffeological categories using the homotopy theory developed in \cite{roberts2012internal} for categories internal to a site. The site in this instance is the category $\cat{Diff}$ of diffeological spaces with the coverage of subductions, see Example \ref{ex coverage of subductions on Diff}. This homotopy theory provides us with a notion of weak equivalence $f: X \to Y$ of diffeological categories, see \cite[Definition 4.14]{roberts2012internal}, which if both $X$ and $Y$ are diffeological groupoids, coincides with the notion of weak equivalence of diffeological groupoids considered in \cite{watts2022bicategories} and \cite{van2020diffeological}.

If we consider $X$ as a diffeological category $[X = X]$ with all structure maps being the identity, then the canonical map $[QX_1 \rightrightarrows QX_0] \to [X = X]$ of diffeological categories is not a weak equivalence, as it is not fully faithful.

However, there is another diffeological groupoid we can consider. Given a diffeological space $X$, we can consider the canonical map $\pi: B \to X$ as mentioned above. This can be made into a diffeological groupoid $\check{C}(X)$ by setting $\check{C}(X)_0 = B$ and $\check{C}(X)_1 = B \times_X B$, with the source and target maps being the obvious projection maps. We will call this the \textbf{\v{C}ech resolution} of $X$, as a diffeological groupoid. It is not hard then to check that the canonical map $\check{C}(X) \to [X = X]$ of diffeological groupoids is indeed a weak equivalence.

If we then take the nerve of $\check{C}(X)$, we obtain a simplicial diffeological space, which we can also consider as a simplicial presheaf. 

\begin{Prop}
The natural map $\check{C}(X) \to {}^c X$ of simplicial presheaves is a \v{C}ech weak equivalence.
\end{Prop}

\begin{proof}
It is easily checked that the map $\pi: B \to X$ is a local epimorphism, as it is objectwise a surjection. Thus \cite[Corollary A.3]{dugger2004hypercovers} proves that $\check{C}(X) \to {}^c X$ is a weak equivalence in the \v{C}ech model structure on simplicial presheaves.
\end{proof}

Therefore we have a zig-zag of \v{C}ech weak equivalences of simplicial presheaves
$$\check{C}(X) \to {}^c X \leftarrow QX.$$

However $\check{C}(X)$ will not be cofibrant in the projective model structure on simplicial presheaves in general. Thus for our purposes, $QX$ is the preferable resolution of $X$, while for the purposes of those interested in diffeological groupoid theory, $\check{C}(X)$ might be the more preferable resolution.

The final resolution we will discuss is that of the \textbf{gauge monoid} that appears in \cite{iglesias2020vcech}. Given a diffeological space $X$, its nebula $B = \coprod_{p \in \cat{Plot}(X)} U_p$ is a diffeological space, and we can consider the set of smooth maps $f: B \to B$ such that the following diagram commutes:
\begin{equation*}
    \begin{tikzcd}
	B && B \\
	& X
	\arrow["\pi"', from=1-1, to=2-2]
	\arrow["\pi", from=1-3, to=2-2]
	\arrow["f", from=1-1, to=1-3]
\end{tikzcd}
\end{equation*}
It inherits the subspace diffeology from the functional diffeology on $C^\infty(B,B)$. Notice that $M$ acts on $B$ by $B \times M \xrightarrow{\rho} B$, where $\rho(b,m) = m(b)$.

We can therefore consider the diffeological category $B//M \coloneqq [B \times M \rightrightarrows B]$, where the source map $s: B \times M \to M$ is given by $s(b,m) = b$, and the target map $t: B \times M \to M$ is given by $t(b,m) = m(b)$.

There is a map $\delta : QX \to B//M$ defined as the identity on objects and on morphisms by $\delta(x_{p_1}, f_0) = (x_{p_1}, \delta f_0)$
where $\delta f_0$ denotes the map $\delta f_0: B \to B$ that is the identity on every component $U_p$ except for $p = p_1$, in which case $\delta f_0|_{U_{p_1}} = f_0$. It is not hard to check that this defines a map of diffeological categories.

In the reverse direction, there is a map $\text{res}: B//M \to QX$ defined to be the identity on objects and on morphisms by
$$\text{res}(x_p, m) = (x_p, m|_{U_p}).$$
It is not hard to see that the composition $QX \xrightarrow{\delta} B//M \xrightarrow{\text{res}} QX$ is the identity, namely that $QX$ is a retract of $B//M$. 

There is a map $q : B//M \to \check{C}(X)$ described in \cite[Page 26]{krepski2021sheaves} which is the identity on objects and on morphisms is defined by $q(x_p, m) = (x_p, m(x_p))$. This defines a map of diffeological categories. It is also easy to see that $q \delta \text{res} = q$.

To summarize, we have the following diagram of maps of diffeological categories, all of which are the identity on objects, but none of which are fully faithful.
% https://q.uiver.app/?q=WzAsMyxbMCwwLCJRWCJdLFsyLDAsIkIvL00iXSxbMSwxLCJcXGNoZWNre0N9KEIgXFx0byBYKSJdLFswLDEsIlxcZGVsdGEiLDEseyJvZmZzZXQiOi0yfV0sWzEsMCwiXFx0ZXh0e3Jlc30iLDEseyJvZmZzZXQiOi0xfV0sWzEsMiwiXFxwaSIsMCx7Im9mZnNldCI6Mn1dLFswLDIsIlxccGkgXFxjaXJjIFxcZGVsdGEiLDJdXQ==
\begin{equation} \label{eqn maps between resolutions}
    \begin{tikzcd}
	QX && {B//M} \\
	& {\check{C}(X)}
	\arrow["\delta"{description}, shift left=2, from=1-1, to=1-3]
	\arrow["{\text{res}}"{description}, shift left=1, from=1-3, to=1-1]
	\arrow["q", shift right=2, from=1-3, to=2-2]
	\arrow["{q \circ \delta}"', from=1-1, to=2-2]
\end{tikzcd}
\end{equation}

The diffeological categories $\check{C}(X)$ and $B//M$ are used to construct \v{C}ech cohomology groups for diffeological spaces in \cite{krepski2021sheaves} and \cite{iglesias2020vcech}.

\subsection{Diffeological \v{C}ech Cohomologies} \label{section diffeological cech cohomologies}

Here we will describe three notions of \v{C}ech cohomology for diffeological spaces that results from the material in Section \ref{section resolutions of diffeological spaces}.

\begin{Rem}
In what follows we will always consider chain complexes and cochain complexes to be non-negatively graded, with differentials going down and up respectively.
\end{Rem}

Let $A$ denote a diffeological abelian group. In \cite{iglesias2020vcech}, \v{C}ech cohomology of a diffeological space $X$ is defined\footnote{Modulo some details, Iglesias-Zemmour defines diffeological spaces with open subsets of cartesian spaces and uses a generating family of open balls, but they are clearly equivalent constructions. He also only restricts to discrete abelian diffeological groups.} as follows. First consider $N (B // M)$, the simplicial diffeological space defined as the nerve of the diffeological category defined in section \ref{section resolutions of diffeological spaces}. Then 
$$A^{N(B//M)_k} = A^{B \times M^{\times k}} = C^\infty(B \times M^{\times k}, A)$$
is precisely the diffeological space of smooth maps $B \times M^{\times k} \to A$. If we forget the smooth structure, then $C^\infty(N(B//M)_k, A)$ is an abelian group by pointwise addition. Thus we obtain a cosimplicial abelian group \begin{equation*}
    \begin{tikzcd}
	{A^B} & {A^{B \times M}} & {A^{B \times M \times M}} & \dots
	\arrow[shift left=2, from=1-1, to=1-2]
	\arrow[shift right=2, from=1-1, to=1-2]
	\arrow[shift left=3, from=1-2, to=1-3]
	\arrow[shift right=3, from=1-2, to=1-3]
	\arrow[from=1-2, to=1-3]
\end{tikzcd}
\end{equation*}
and from this one can obtain a cochain complex as follows. 

If $K$ is a cosimplicial abelian group, then we can define a cochain complex $C^\co K$ called the \textbf{associated cochain complex} by
$$(C^\co K)^n = K^n, \qquad d: (C^\co K)^n \to (C^\co K)^{n+1}, \qquad d = \sum_{i = 0}^n (-1)^i d^i.$$
This definition extends to a functor $C^\co : \cat{cAb} \to \cat{CoCh}$, where $\cat{cAb}$ denotes the category of cosimplicial abelian groups and $\cat{CoCh}$ is the category of cochain complexes. Further there is a functorial direct sum decomposition as cochain complexes $C^\co K \cong N^\co K \oplus D^\co K$, where $D^\co K$ is the subcomplex consisting of degenerate simplices, and the inclusion $N^\co K \to C^\co K$ is a cochain homotopy equivalence of cochain complexes. We call $N^\co K \cong C^\co K / D^\co K$ the \textbf{normalized cochain complex} of $K$.  This is a dual version of what is called the Dold-Kan correspondence, which is an adjoint equivalence
$$N : \cat{sAb} \rightleftarrows \cat{Ch} : \Gamma,$$
where $N$ is the normalized chain complex functor and if $V$ is a chain complex, then $\Gamma V$ is defined degreewise by
$$\Gamma(V)_n = \bigoplus_{[n] \twoheadrightarrow [k]} V_k,$$
where the index is over all surjections $\varphi: [n] \to [k]$ in $\cat{\Delta}$. See \cite[Section 8.4]{weibel1995introduction} and \cite[Section III]{goerss2012simplicial} for details.

The Iglesias-Zemmour \v{C}ech cohomology of $X$ is then defined as the cohomology of this cochain complex: 
$$\check{H}^k_{PIZ}(X,A) = \check{H}^k \left( N^\co \left[ A^{B //M} \right] \right) \cong \check{H}^k \left( C^\co \left[ A^{B//M} \right] \right).$$ 
Similarly, let $N \check{C}(X)$ denote the nerve of the \v{C}ech groupoid defined in Section \ref{section resolutions of diffeological spaces}. If $A$ is an abelian diffeological group, then as above we can map $N \check{C}(X)$ into $A$ to form a cosimplicial abelian group, and taking the cohomology of the associated cochain complex gives us the Krepski-Watts-Wolbert \v{C}ech cohomology \cite{krepski2021sheaves} of $X$ with values in $A$:
$$\check{H}^k_{KWW}(X,A) = \check{H}^k \left( N^\co \left[ A^{\check{C}(X)} \right] \right) \cong \check{H}^k \left( C^\co \left[ A^{\check{C}(X)} \right] \right).$$

We will now construct a cochain complex using the cofibrant replacement $QX$ of a diffeological space $X$. Given an abelian group $A$ and a non-negative integer $n$, there exists a simplicial set $K(A,n)$, called the \textbf{$n$th Eilenberg-Maclane space}, which has trivial homotopy groups in all degrees except for $n$, which has $\pi_n(K(A,n)) = A$. One can construct this simplicial set using the Dold-Kan correspondence. Namely consider the chain complex $A[k]$, defined by
$$ \left( A[k] \right)_n = \begin{cases} A & \text{if } n = k \\
0 & \text{if } n \neq k
\end{cases}, \qquad d = 0.$$
Since $\Gamma(A[k])$ is a simplicial group, it will be a Kan complex, equipped with basepoint $*$ such that $$\pi_n(\Gamma(A[k]),*) = \begin{cases}
    A & \text{if } n = k \\
    0 & \text{if } n \neq k.
\end{cases}$$

\begin{Rem}
For future reference, if $V$ is a chain complex, then let $V[k]$ denote the chain complex such that $V[k]_n = V_{n-k}$, so that we identify an abelian group A with the chain complex $A[0]$, and then $(A[0])[k] = A[k]$.
\end{Rem}

Now we will define $\infty$-stack cohomology for simplicial presheaves. This theory, which we call $\infty$-stack cohomology, is very well developed, and generalizes many examples of cohomology found throughout mathematics, see \cite{schreiber2013differential}, \cite{lurie2009higher}, \cite{bunke2013differential}.

\begin{Def}
Let $X$ be a projective cofibrant simplicial presheaf and $A$ an $\infty$-stack. Then the zeroth \textbf{$\infty$-stack cohomology} of $X$ with values in $A$ is given by
\begin{equation} \label{eq zero infinity stack cohomology}
    \check{H}^0_\infty(X,A) \coloneqq \pi_0 \R \Hom(X,A) \cong \pi_0 \u{\spre(\cart)}(X,A).
\end{equation}
\end{Def}

Note that in the above definition, $A$ is an arbitrary $\infty$-stack. Thus $\check{H}^0_\infty(X,A)$ is an example of nonabelian cohomology. However, in order to define $\check{H}^1_\infty$, we must ask for more structure on $A$, namely that it be an $\infty$-stack, and that $A$ also be a group object in $\spre(\cart)$, namely that $A(U)$ be a simplicial group for each $U \in \cart$ and given a smooth map $f: U \to V$, the map $A(f): A(V) \to A(U)$ is a map of simplicial groups. We call group objects of $\spre(\cart)$ presheaves of simplicial groups.
\begin{Def}
Given a simplicial group $G$, let $\conj{W}G$ denote the simplicial set with 
\begin{equation} \label{eq Wbar}
\begin{aligned}
    \overline{W} G_0 & = * \\
    \overline{W} G_n & = G_{n-1} \times G_{n-2} \times \dots \times G_0 \\
    \end{aligned}
    \end{equation}
    
    with face and degeneracy maps given by
\begin{equation*}
    \begin{aligned}
d_i(g_{n-1}, \dots, g_0) & = 
    \begin{cases}
    (g_{n-2}, \dots, g_0), & \text{ if } i = 0 \\
    (d_{i-1}(g_{n-1}), \dots,  d_1(g_{n-i+1}), g_{n-i-1} \cdot d_0(g_{n-i}), g_{n-i-2},\dots, g_0)), & \text{ if } 1 \leq i \leq n
    \end{cases} \\
    s_i(g_{n-1}, \dots, g_0) &= \begin{cases}
    (1, g_{n-1}, \dots, g_0), & \text{ if } i = 0 \\
    (s_{i-1}(g_{n-1}), \dots, s_0(g_{n-i}), 1, g_{n-i-1}, \dots, g_0), & \text{ if } 1 \leq i \leq n.
    \end{cases}
\end{aligned}
\end{equation*}
\end{Def}

Simplicial sets of the form $\conj{W}G$ classify what are called principal twisted cartesian products or PTCPs in \cite{may1992simplicial}. The combinatorial structure of $\conj{W}G$ may look complicated, but it has other equivalent descriptions that are more motivated, see \cite[Chapter V]{goerss2012simplicial} and \cite{stevenson2012decalage}.

\begin{Lemma}[{\cite[Corollary 6.8]{goerss2012simplicial}}] \label{lem delooping of simplicial group is a kan complex}
If $G$ is a simplicial group, then $\conj{W}G$ is a Kan complex.
\end{Lemma}

If ${}^c G$ a discrete simplicial group, i.e. a group, then $\conj{W} {}^c G \cong N[G \rightrightarrows *]$, the nerve of $G$ thought of as a groupoid with one object. Thus $|\conj{W}{}^c G|$, the geometric realization of the delooping, is weak homotopy equivalent to the classifying space $BG$. 

Now if $A$ is a presheaf of simplicial groups, then we can apply $\conj{W}$ objectwise, and we obtain a functor $\conj{W}: \cat{sPre}(\cart, \cat{sGrp}) \to \cat{sPre}(\cart),$ where $\cat{sPre}(\cart, \cat{sGrp})$ denotes the full subcategory of presheaves of simplicial groups. Further, by Lemma \ref{lem delooping of simplicial group is a kan complex}, $\conj{W}A$ is projective fibrant, i.e. objectwise a Kan complex. 

\begin{Lemma}
Let $G$ be a sheaf of groups on $\cart$. Then the delooping stack of Example \ref{ex delooping stack} is isomorphic to its delooping as a presheaf of simplicial groups
$$\bold{B}G \cong \conj{W}{}^c G.$$
\end{Lemma}

So suppose that $A$ is an $\infty$-stack on $\cart$, and further, that it is a presheaf of simplicial groups. Then we get a new simplicial presheaf $\conj{W}A$, and it is projective fibrant. We therefore define the first $\infty$-stack cohomology group of a simplicial presheaf $X$ with values in $A$ to be
$$\check{H}^1_\infty(X,A) \cong \pi_0 \R \Hom(X, \conj{W} A).$$

In order to be able to compute this, it would be convenient to know that $\conj{W} A$ is fibrant in the \v{C}ech model structure, i.e. is an $\infty$-stack. This follows thanks to the following wonderful theorem.

\begin{Th}[{\cite[Proposition 3.3.30]{sati2022equivariant}, \cite[Proposition 4.13]{pavlov2022numerable}}] \label{th delooping infinity stack is infinity stack}
If $A$ is an $\infty$-stack on $\cart$ that is also a presheaf of simplicial groups, then $\conj{W} A$ is an $\infty$-stack on $\cart$.
\end{Th}

Thus if $A$ is an $\infty$-stack, then for any simplicial presheaf $X$, $\check{H}^0_\infty(X,A)$ is well defined, and if $A$ is also a presheaf of simplicial groups, then $\check{H}^1_\infty(X,A)$ is also well defined. To obtain higher cohomology groups, we must ask for higher deloopings of $A$ to exist. 

\begin{Def}
Let $A$ be an $\infty$-stack that is also a presheaf of simplicial groups. If $\conj{W}^k A$ is a presheaf of simplicial groups for all $1 \leq k \leq n - 1$, and $X$ is a simplicial presheaf, then let
\begin{equation}
    \check{H}^n_\infty(X,A) = \pi_0 \R \Hom(X,\conj{W}^n A)
\end{equation}
denote the $n$th $\infty$-stack cohomology of $X$ with values in $A$.
\end{Def}

It is thus important to know under what conditions will these higher deloopings $\conj{W}^n A$ exist.

\begin{Lemma}
If $A$ is a simplicial abelian group, namely $A$ is a simplicial group and $A_k$ is an abelian group for all $k$, then $\conj{W}A$ will be a simplicial group, and further it will be an abelian simplicial group.
\end{Lemma}

\begin{proof}
It follows from the isomorphism $\conj{W}A \cong TNA$ of \cite[Lemma 5.2]{stevenson2012decalage} and the discussion of $T$ in \cite[Section III]{artin1966van} that $\conj{W}A$ is a simplicial group, and that it is abelian is clear from the formula (\ref{eq Wbar}).
\end{proof}

Thus if $A$ is a simplicial abelian group, $\conj{W}^k A$ exists for all $k$.

\begin{Lemma}[{\cite[Section 4.6]{jardine1997generalized}}] \label{lem dold kan delooping}
Let $A$ be a simplicial abelian group. Then there is an isomorphism of chain complexes
$$N \conj{W}A \cong (NA)[1]$$
where $(N A)[1]$ is the chain complex $NA$ shifted up by $1$, i.e. $(NA[1])_k = (NA)_{k-1}$.
\end{Lemma}

\begin{Lemma}\label{lem delooping abelian groups}
If $A$ is an abelian group, thought of as a discrete simplicial abelian group ${}^c A$, then $\conj{W}^k A$ exists for every $k \geq 0$, and there exists an isomorphism
$$\conj{W}^k {}^cA \cong \Gamma(A[k])$$
\end{Lemma}

\begin{proof}
We proceed by induction. For the base case, we have
$$N \conj{W} {}^c A \cong (N{}^c A)([1]).$$
But $N{}^c A \cong A[0]$ as is easily checked, and $(A[0])[1] = A[1]$, so 
$ N \conj{W} {}^c A \cong A[1]$, thus 
$$ \Gamma N \conj{W} {}^c A \cong \conj{W}{}^c A \cong \Gamma A[1].$$

Now suppose $\conj{W}^{k-1}{}^c A \cong \Gamma A[k-1]$. Then by Lemma \ref{lem dold kan delooping}
$$N \conj{W} \left( \conj{W}^{k-1} {}^c A \right) \cong \left( N \Gamma A[k-1] \right) [1]$$
but $N \Gamma A[k-1] \cong A[k-1]$ since $N$ and $\Gamma$ form an adjoint equivalence, thus:
$$N \conj{W} \left( \conj{W}^{k-1} {}^c A \right) \cong N \conj{W}^k {}^c A \cong A[k-1]([1]) = A[k]$$
taking the adjoint gives
$$\conj{W}^k {}^c A \cong \Gamma A[k].$$
\end{proof}

Now if $A$ is an abelian diffeological group, then it is an $\infty$-stack on $\cart$, since it is a sheaf on $\cart$, and therefore a discrete presheaf of simplicial groups. Thus by \ref{lem delooping abelian groups} and \ref{th delooping infinity stack is infinity stack}, $\conj{W}^n A$ exists, and the $n$th $\infty$-stack cohomology of a diffeological space $X$ with values in $A$ is given by
\begin{equation} \label{eqn higher infinity sheaf cohomology}
    \check{H}^n_\infty(X,A) = \pi_0 \u{\spre(\cart)}(QX,\conj{W}^n A) \cong \pi_0\cat{Tot}( \conj{W}^n [A (QX)]),
\end{equation}
where $\cat{Tot}$ is the totalization of Definition \ref{def totalization}, $A(QX)$ is the cosimplicial abelian group which in degree $k$ is given by $C^\infty(QX_k, A)$, and $\conj{W}^n[A(QX)]$ is $\conj{W}^n$ applied to $A(QX)$ degreewise. Now $\conj{W}^n [A(QX)] \cong \Gamma [A(QX)[n]]$ by Lemma \ref{lem delooping abelian groups}.

\begin{Prop}[{\cite[Lemma 19]{jardine2019cosimplicial}}] \label{prop simplifying Tot for cosimplicial abelian groups}
If $A$ is a cosimplicial abelian group, then 
\begin{equation}
    \pi_0 \cat{Tot}(\Gamma A[k]) \cong \check{H}^k(N^{\co} A)
\end{equation}
where $\Gamma A[k]$ denotes the cosimplicial simplicial abelian group obtained by considering the abelian group $A^i$ as a simplicial abelian group $\Gamma (A^i[k])$.
\end{Prop}

So substituting for the cosimplicial abelian group $A(QX)$ in Proposition \ref{prop simplifying Tot for cosimplicial abelian groups} we have the main result of this section, which provides a concrete way of computing the $\infty$-stack cohomology of a diffeological space with values in an abelian diffeological group.

\begin{Cor} \label{cor explicit description of infinity stack cohomology}
If $X$ is a diffeological space, and $A$ is an abelian diffeological group, and we consider the cochain complex
$$C^{co}[A(QX)] = A^B \xrightarrow{d} A^{QX_1} \xrightarrow{d} A^{QX_2} \xrightarrow{d} \dots$$
then the $n$th $\infty$-stack cohomology of $X$ with values in $A$ can be computed by $C^{co}[A(QX)]$, namely
$$\check{H}^n_\infty(X,A) \cong \check{H}^n_\infty(N^{co} [A(QX)]) \cong \check{H}^n_\infty(C^{co} [A(QX)]).$$
\end{Cor}

This explicit description of $\infty$-sheaf cohomology will be useful in comparing the various \v{C}ech cohomologies.

\begin{Prop} \label{prop piz and infinity sheaf cohomology agree on 0}
For a diffeological space $X$, and a diffeological abelian group $A$, the $\infty$-sheaf cohomology and Iglesias-Zemmour cohomology agree in degree 0:
$$\check{H}^0_{PIZ}(X, A) = \check{H}^0_\infty(X,A).$$
\end{Prop}

\begin{proof}
The set of $0$-cocycles in $\infty$-sheaf cohomology is the set
$$\check{H}^0_\infty(X,A) = \left \{ \tau: B \to A \, | \, \text{if } f: U_p \to U_q \text{ is a map of plots then } \tau \circ \delta f = \tau \right \},$$
where $\delta f$ denotes the map $\delta f: B \to B$ that is the identity on every component except for $U_p$, where it is $f$. Equivalently it is the set of smooth maps $\tau: B \to A$ such that if $(x_{p_1}, f_0) \in QX_1$, then $\tau(f_0(x_{p_1})) =\tau(x_{p_1})$.
The set of $0$-cocycles in Iglesias-Zemmour cohomology is
$$\check{H}^0_{PIZ}(X, A) = \left \{ \sigma : B \to A \, | \, \text{if } m \in M, \text{then }\sigma \circ m = \sigma \right \}.$$
Equivalently it is the set of smooth maps $\sigma: B \to A$ such that if $(x_{p_1}, m) \in B \times M$, then $\sigma(m(x_{p_1})) = \sigma(x_{p_1})$. Notice that $\check{H}^0_{PIZ}(X, A) \subseteq \check{H}^0_\infty(X, A)$, since every $\delta f$ is an element of $M$. Now if $\tau \in \check{H}^0_\infty(X, A)$, and $(x_{p_1}, m) \in B \times M$, then $\tau(m(x_{p_1})) = \tau(x_{p_1})$, because $m|_{U_{p_1}}$ is a map of plots. Thus $\check{H}^0_\infty(X,A) \subseteq \check{H}^0_{PIZ}(X, A).$
\end{proof}

Note that 
\begin{equation} \label{eqn simplifying 0 infinity sheaf cohomology}
    \pi_0 \u{\spre(\cart)}(QX,{}^c A) \cong \cat{Pre}(\cart)(\pi_0 QX, A) \cong \cat{Pre}(\cart)(X, A) \cong \cat{Diff}(X,A),
\end{equation}
where the first isomorphism follows from the adjunction described in Section \ref{section smooth higher stacks}, and the second isomorphism follows from Remark \ref{rem coequalizer in diff same as pi0}.

\begin{Cor} \label{cor cech cohomology in degree 0}
If $X$ is a diffeological space, and $A$ a diffeological abelian group, then
\begin{equation}
    \check{H}^0_\infty(X,A) \cong \check{H}^0_{PIZ}(X,A) \cong \check{H}^0_{KWW}(X,A) \cong \cat{Diff}(X,A).
\end{equation}
\end{Cor}

\begin{proof}
This follows from (\ref{eqn simplifying 0 infinity sheaf cohomology}), Proposition \ref{prop piz and infinity sheaf cohomology agree on 0} and \cite[Proposition 4.6]{krepski2021sheaves}.
\end{proof}

Now recall the map $q \circ \delta : QX \to \check{C}(X)$ from (\ref{eqn maps between resolutions}). This induces a map on cohomology which we will denote by $(q \delta)^* = \varphi : \check{H}^\bullet_{KWW}(X,A) \to \check{H}^\bullet_\infty(X,A)$. In degree $1$, $\varphi$ has the following explicit description on cocycles. Namely if $\tau : \check{C}(X)_1 \to A$ is a cocycle, then $(\varphi \tau)(x_{p_1}, f_0) = \tau(x_{p_1}, f_0(x_{p_1}))$.

\begin{Prop} \label{prop kww cohomology and infinity stack cohomology agree in degree 1}
For any diffeological space $X$ and abelian diffeological group $A$, the map $\varphi: \check{H}^1_{KWW}(X,A) \to \check{H}^1_\infty(X,A)$ is an isomorphism.
\end{Prop}

\begin{proof}
Let us show that $\varphi$ is surjective. Suppose that $\sigma$ is a $1$-cocycle for the cochain complex $A^{QX}$. This means that if $f_1, f_0$ are composable maps of plots, then
$$\sigma(f_1(x_{p_2}), f_0) = \sigma(x_{p_2}, f_0 f_1) - \sigma(x_{p_2}, f_1).$$
Now if $(x_{p_1}, f_0) \in QX_1$, then notice we have the following commutative diagram of plot maps
\begin{equation*}
    \begin{tikzcd}
	& {*} \\
	{U_{p_1}} && {U_{p_0}} \\
	& X
	\arrow["{p_1}"', from=2-1, to=3-2]
	\arrow["{p_0}", from=2-3, to=3-2]
	\arrow["{x_{p_1}}"', from=1-2, to=2-1]
	\arrow["{f_0(x_{p_1})}", from=1-2, to=2-3]
	\arrow["{f_0}", from=2-1, to=2-3]
\end{tikzcd}
\end{equation*}
which implies that if $\sigma$ is a $1$-cocycle that $\sigma(x_{p_1}, f_0) = \sigma(*, f_0(x_{p_1})) - \sigma(*, x_{p_1}).$ So consider the map $\tau: \check{C}(X)_1 \to A$ defined as follows. If $(x_p, y_q) \in \check{C}(X)_1$, then let $\tau(x_p, y_q) = \sigma(*, y_q) - \sigma(*, x_p)$. Then $(\varphi \tau)(x_{p_1}, f_0) = \sigma(x_{p_1}, f_0)$ for every $(x_{p_1}, f_0) \in QX_1$.

Now we wish to show that $\varphi$ is injective. Suppose $\tau, \tau' : \check{C}(X)_1 \to A$ are $1$-cocycles such that there exists some $\alpha: B \to A$ such that for every $(x_{p_1}, f_0) \in QX_1$, 
$$\tau(x_{p_1}, f_0(x_{p_1})) - \tau'(x_{p_1}, f_0(x_{p_1})) = \alpha(f_0(x_{p_1})) - \alpha(x_{p_1}).$$
Then if $(x_p, y_q) \in \check{C}(X)_1$, we have the following commutative diagram of plot maps
\begin{equation*}
    \begin{tikzcd}
	& {*} \\
	{U_p} && {V_q} \\
	& X
	\arrow["p"', from=2-1, to=3-2]
	\arrow["q", from=2-3, to=3-2]
	\arrow["{x_p}"', from=1-2, to=2-1]
	\arrow["{y_q}", from=1-2, to=2-3]
	\arrow["z"{description}, from=1-2, to=3-2]
\end{tikzcd}
\end{equation*}
where $z = p(x_p) = q(y_q)$, and we use $z$ to refer to the point $*$ in the plot $z: * \to X$. Now since $\tau$ and $\tau'$ are $1$-cocycles, it follows that
$$\tau(x_p, y_q) = \tau(z,y_q) - \tau(z, x_p), \qquad \tau'(x_p, y_q) = \tau(z, y_q) - \tau(z, x_p).$$
Therefore
\begin{equation*}
    \begin{aligned}
    \tau(x_p, y_q) - \tau'(x_p, y_q) & = (\tau(z,y_q) - \tau(z, x_p)) - (\tau'(z,y_q) - \tau'(z,x_p)) \\
    & = (\tau(z, y_q) - \tau'(z,y_q)) - (\tau(z,x_p) - \tau'(z,x_p)) \\
    & = (\alpha(y_q) - \alpha(z)) - (\alpha(x_p) - \alpha(z)) \\
    & = \alpha(y_q) - \alpha(x_p).
    \end{aligned}
\end{equation*}
which means that $\tau$ and $\tau'$ differ by a coboundary in $A^{\check{C}(X)}$, so $\varphi$ is injective.
\end{proof}

From (\ref{eqn maps between resolutions}) we obtain the following diagram of abelian groups for every $k \geq 0$
\begin{equation} \label{eq diagram in cohomology}
    \begin{tikzcd}
	{\check{H}^k_{\infty}(X, A)} && {\check{H}^k_{PIZ}(X,A)} \\
	& {\check{H}^k_{KWW}(X,A)}
	\arrow["{(q\delta)^*}", from=2-2, to=1-1]
	\arrow["{\text{res}^*}"', shift right, from=1-1, to=1-3]
	\arrow["{\delta^*}"', shift right=2, from=1-3, to=1-1]
	\arrow["{q^*}"', from=2-2, to=1-3]
\end{tikzcd}
\end{equation}
We summarize everything we know about this diagram
\begin{enumerate}
    \item In degree $0$, all of the above maps are the identity by Corollary \ref{cor cech cohomology in degree 0},
    \item In degree $1$, the map $(q \delta)^*$ is an isomorphism by Proposition \ref{prop kww cohomology and infinity stack cohomology agree in degree 1},
    \item The map $\delta^*$ is a retraction, namely $\delta^* \text{res}^* = 1_{\check{H}^k_{\infty}(X,A)}$, so $\text{res}^*$ is injective and $\delta^*$ is surjective for all $k \geq 0$,
    \item We have $\text{res}^* \delta^* q^* = q^*$ for all $k \geq 0$, thus the above diagram actually commutes
    \item The map $q^*$ is injective for $k = 1$ by \cite[Lemma 6.9]{krepski2021sheaves}. Notice that this can also be seen by noting that since $(q \delta)^*$ is an isomorphism for $k = 1$, and $\text{res}^*$ is injective, and since $\text{res}^* \delta^* q^* = \text{res}^* (q \delta)^* = q^*$, then $q^*$ is injective for $k = 1$.

\end{enumerate}

It is not currently known if any of the above maps are isomorphisms for all $k \geq 0$.

\section{Principal Diffeological Bundles as Principal Infinity Bundles} \label{section principal diffeological bundles as principal infinity bundles} Principal Infinity Bundles were defined in \cite{NSSGeneral} and \cite{NSSPresentations}. In this section, we compare this abstract notion to diffeological principal bundles. 

\begin{Rem}
The following two definitions are needed only for Definition \ref{def principal infinity bundle} and are not used elsewhere in this paper.
\end{Rem}

\begin{Def}[{\cite[Section 3]{dugger2004hypercovers}}]
A map $f: X \to Y$ of simplicial presheaves on $\cart$ is a \textbf{local fibration} if for every $U \in \cart$, there exists a good open cover $\{U_i \subseteq U \}$ such that for every element $U_i$ of the good open cover, there is a lift in every commutative diagram of the following form.
\begin{equation}
    \begin{tikzcd}
	{\Lambda^n_k} & {X(U)} & {X(U_i)} \\
	{\Delta^n} & {Y(U)} & {Y(U_i)}
	\arrow[hook, from=1-1, to=2-1]
	\arrow[from=1-1, to=1-2]
	\arrow[from=2-1, to=2-2]
	\arrow[from=2-2, to=2-3]
	\arrow[from=1-2, to=1-3]
	\arrow["f", from=1-3, to=2-3]
	\arrow[dashed, from=2-1, to=1-3]
\end{tikzcd}
\end{equation}
\end{Def}

Note that an objectwise fibration of simplicial presheaves is a local fibration. We say that a simplicial presheaf $X$ is \textbf{locally fibrant} if the unique map $X \to *$ is a local fibration.

\begin{Def}[{\cite[Theorem 6.15]{dugger2004weak}}]
A map $f: X \to Y$ of simplicial presheaves on $\cart$ is a \textbf{local weak equivalence} if for every $U \in \cart$, there exists a good open cover $\{U_i \subseteq U \}$ such that for every element $U_i$ of the good open cover, there is a dotted arrow in every commutative diagram of the following form,
\begin{equation}
    \begin{tikzcd}
	{\partial \Delta^n} & {RX(U)} & {RX(U_i)} \\
	{\Delta^n} & {RY(U)} & {RY(U_i)}
	\arrow[hook, from=1-1, to=2-1]
	\arrow[from=1-1, to=1-2]
	\arrow[from=2-1, to=2-2]
	\arrow[from=2-2, to=2-3]
	\arrow[from=1-2, to=1-3]
	\arrow["Rf", from=1-3, to=2-3]
	\arrow[dashed, from=2-1, to=1-3]
\end{tikzcd}
\end{equation}
where $R$ is a fibrant replacement functor for $\cat{sSet}$ and the top left triangle commutes strictly, while the bottom right triangle commutes up to a homotopy relative to $\partial \Delta^n \hookrightarrow \Delta^n$.
\end{Def}

Note that an objectwise weak equivalence is a local weak equivalence, and \cite{dugger2004hypercovers} proves that \v{C}ech weak equivalences are local weak equivalences.

\begin{Def}[{\cite[Definition 3.79]{NSSPresentations}}] \label{def principal infinity bundle}
Let $G$ be a presheaf of simplicial groups acting on a simplicial presheaf $P$ by $\rho: P \times G \to P$. Then a map $\pi: P \to X$ is a \textbf{$G$-Principal $\infty$-bundle}\footnote{In \cite{NSSPresentations}, what we call principal $\infty$-bundles are known as weakly principal $G$-bundles. Also, they only define this for simplicial sheaves, but there are no problems extending the definition and all of the theorems in that paper to simplicial presheaves.} if:
\begin{enumerate}
    \item $\pi$ is a local fibration,
    \item The action of $G$ on $P$ is fiberwise, namely $\rho(g,-)$ sends fibers to fibers, and
    \item the map
    $$P \times G \to P \times_X P$$
    given by 
    $$(p,g) \mapsto (p, \rho(p,g))$$
    is a local weak equivalence.
\end{enumerate}
A map $P \xrightarrow{f} P'$ of $G$-principal $\infty$-bundles over $X$ is a map that is $G$-equivariant and commutes with the bundle projections. Namely, it is a map fitting into the following commutative diagram:
\begin{equation*}
    \begin{tikzcd}
	{P \times G} && {P' \times G} \\
	P && {P'} \\
	& X
	\arrow["f", from=2-1, to=2-3]
	\arrow["\pi"', from=2-1, to=3-2]
	\arrow["{\pi'}", from=2-3, to=3-2]
	\arrow["\rho"', from=1-1, to=2-1]
	\arrow["{\rho'}", from=1-3, to=2-3]
	\arrow["{f \times 1_G}", from=1-1, to=1-3]
\end{tikzcd}
\end{equation*}
Let $\cat{Prin}^\infty_G(X)$ denote the category of $G$-principal $\infty$-bundles on $X$.
\end{Def}

It is clear that if $\pi: P \to X$ is a diffeological principal $G$-bundle, then it is a $G$-principal $\infty$-bundle, when we think of $X$, $G$ and $P$ as discrete simplicial presheaves. This is because all maps between discrete simplicial presheaves are local fibrations, and all diffeomorphisms between diffeological spaces are local weak equivalences. Note that $\cat{DiffPrin}_G(X)$ is a groupoid, while in general $\cat{Prin}_G^\infty(X)$ is not a groupoid. So while these categories are not equivalent, we will prove that their nerves are weak homotopy equivalent.

\begin{Prop} \label{prop Rhom is principal infinity bundles}
Let $X$ be a locally fibrant simplicial presheaf and $G$ a presheaf of simplicial groups. Then, there is a weak homotopy equivalence of simplicial sets
\begin{equation}
    \R \Hom(X, \bold{B}G) \simeq N \cat{Prin}_G^\infty(X).
\end{equation}
\end{Prop}

\begin{proof}
First if $X$ and $Y$ are locally fibrant simplicial presheaves, then combining \cite[Lemma 6.4]{low2015cocycles} with \cite[Theorem 3.12]{low2015cocycles} and \cite[Corollary 4.7]{dwyer1980function} proves that
$$\R \Hom(X,Y) \simeq N \text{Cocycle}(X,Y)$$
where $\text{Cocycle}(X,Y)$ is the cocycle category as defined in \cite[Definition 3.1]{low2015cocycles}. Then \cite[Theorem 3.95]{NSSPresentations} proves that 
$$N \text{Cocycle}(X,\bold{B}G) \simeq N\cat{Prin}_G^\infty(X).$$
Since all presheaves of simplicial groups are locally fibrant, combining these gives the desired result.
\end{proof}

Since all diffeological spaces are locally fibrant, if $X$ is a diffeological space and $G$ is a diffeological group, to prove that $N \cat{DiffPrin}_G(X)$ is weak equivalent to $N \cat{Prin}_G^\infty(X)$, it suffices to show that $N \cat{DiffPrin}_G(X)$ is weak homotopy equivalent to $\R \Hom(X, \bold{B}G)$. Let us examine $\R \Hom(X, \bold{B} G)$ more deeply. In Section \ref{section resolutions of diffeological spaces} we saw that this is equal to the simplicial set $\u{\spre(\cart)}(QX, \bold{B}G)$. Now, if we consider the definition of $QX$ given in Lemma \ref{lem cofibrant replacement of diff space}, then by the same computation as (\ref{eqn spre as Tot}), we have
\begin{equation} \label{eqn spre as Tot for diff spaces and BG}
    \u{\spre(\cart)}(QX, \bold{B}G) \cong \cat{Tot}(\bold{B}G(QX)).
\end{equation}

A $k$-simplex of $\cat{Tot}(\bold{B}G(QX))$ contains a huge amount of information, but in this case, since $\bold{B}G$ is objectwise the nerve of a groupoid, most of this information will be redundant. Let us describe what a vertex of this simplicial set is. It is a map of cosimplicial simplicial sets $\Delta^\bullet \to \bold{B}G(QX)$. This means it is a commutative diagram of the form\footnote{Where we exclude the codegeneracy maps from the notation for clarity.}:
\begin{equation*}
\begin{tikzcd}
	{\Delta^0} & {\Delta^1} & {\Delta^2} & \dots \\
	{\prod_{\cat{Plot}(X)} \bold{B}G(U_{p_0})} & {\prod_{N_1} \bold{B}G(U_{p_1})} & {\prod_{N_2} \bold{B}G(U_{p_2})} & \dots
	\arrow[shift left=2, from=1-1, to=1-2]
	\arrow[shift right=2, from=1-1, to=1-2]
	\arrow["{g^1}"', from=1-2, to=2-2]
	\arrow[shift left=3, from=1-2, to=1-3]
	\arrow[from=1-2, to=1-3]
	\arrow[shift right=3, from=1-2, to=1-3]
	\arrow[shift left=3, from=2-2, to=2-3]
	\arrow[shift right=3, from=2-2, to=2-3]
	\arrow[from=2-2, to=2-3]
	\arrow["{g^2}"', from=1-3, to=2-3]
	\arrow["{g^0}"', from=1-1, to=2-1]
	\arrow[shift left=2, from=2-1, to=2-2]
	\arrow[shift right=2, from=2-1, to=2-2]
\end{tikzcd}
\end{equation*}

Let us unravel what this means. Firstly, $g^0$ contributes no information, as $\bold{B}G(U_{p_0})_0 = C^\infty(U_{p_0}, *) = *$. However, $g^1$ is the data of maps $g^1(f_0) \coloneqq g_{f_0}: U_{p_1} \to G$ for every map of plots $f_0: U_{p_1} \to U_{p_0}$. Now $g^2$ is the data of a map $g^2(f_1, f_0) \coloneqq g_{f_1, f_0}: U_{p_2} \to G \times G$ for every pair of composable maps of plots $U_{p_2} \xrightarrow{f_1} U_{p_1} \xrightarrow{f_0} U_{p_0}$. Let $g^2(f_1, f_0) = (h,k)$. The data of the above cosimplicial  map insists that
$$(d^0 g^1)(f_1,f_0) = g^1(f_0) \circ f_1 = (g^2 d^0)(f_1,f_0) = d_0 (g^2(f_1,f_0)) = d_0(h,k) = k$$
$$(d^1 g^1)(f_1,f_0) = g^1(f_0 f_1) = (g^2 d^1)(f_1, f_0) = d_1 (g^2(f_1,f_0)) = d_1(h,k) = kh\footnote{If this seems strange, see Example \ref{ex delooping stack}.}$$
$$(d^2 g^1)(f_1, f_0) = g^1(f_1) = (g^2 d^2)(f_1, f_0) = d_2 (g^2(f_1, f_0)) = d_2(h,k) = h.$$
In other words
\begin{equation*}
    g_{f_0 f_1} = (g_{f_0} \circ f_1) \cdot g_{f_1}.
\end{equation*} 
This is precisely the diffeological $G$-cocycle condition (\ref{eqn diff cocycle condition}). We can visualize this as a triangle:
\begin{equation*}
\begin{tikzcd}
	& \bullet \\
	\bullet && \bullet
	\arrow[""{name=0, anchor=center, inner sep=0}, "{g_{f_0 f_1}}"', from=2-1, to=2-3]
	\arrow["{g_{f_1}}", from=2-1, to=1-2]
	\arrow["{g_{f_0} \circ f_1}", from=1-2, to=2-3]
	\arrow["{g^2(f_1,f_0)}"{description, pos=0.6}, Rightarrow, draw=none, from=1-2, to=0]
\end{tikzcd}
\end{equation*}
which is filled in if the cocycle condition (\ref{eqn diff cocycle condition}) holds. In other words, a map $QX \to \B G$ is precisely the same information as a $G$-cocycle $g$ on $X$.

Now here's an important point: $g^3$ will provide no further data. We will explain why using the notion of coskeleton.

\begin{Def}
A simplicial set $X$ is \textbf{$k$-coskeletal} if for every boundary $\partial \Delta^n \to X$, there exists a unique $n$-simplex $\Delta^n \to X$ making the following diagram commute:
\begin{equation*}
    \begin{tikzcd}
	{\partial \Delta^n} & X \\
	{\Delta^n}
	\arrow[hook, from=1-1, to=2-1]
	\arrow[from=1-1, to=1-2]
	\arrow[from=2-1, to=1-2]
\end{tikzcd}
\end{equation*}
for all $n > k$.
\end{Def}

For any $k$, let $\cat{sSet}_{\leq k}$ denote the category of $k$-truncated simplicial sets, namely presheaves on the full subcategory $\cat{\Delta}_{\leq k}$ of $\cat{\Delta}$ whose objects are partial orders $[n]$ for $n \leq k$. There is a functor $\tau_k: \cat{sSet} \to \cat{sSet}_{\leq k}$ just given by forgetting the higher simplices of the simplicial set. This functor has a fully faithful left adjoint $\text{sk}_k$ and a fully faithful right adjoint $\text{cosk}_k$. A simplicial set $X$ is $k$-coskeletal if the unit of the adjunction $X \to \text{cosk}_k(X)$ is an isomorphism. For more details see \cite[Section VII.1]{goerss2012simplicial}.

If $X = N(\site)$ is the nerve of a category $\site$, then $X$ is $2$-coskeletal \cite[Lemma I.3.5]{goerss2012simplicial}. In our case $\bold{B}G(QX)$ is a cosimplicial simplicial set such that $\bold{B}G(QX_n)$ is the nerve of a groupoid and therefore $2$-coskeletal for every $n$. Now as we've seen, the $3$-simplex $g^3 \in \bold{B}G(QX_3)$ is required to satisfy that $\partial g^3 = (d^0 g^1, d^1 g^1, d^2 g^1, d^3 g^1)$. But that means we've just specified a $3$-boundary in a $2$-coskeletal simplicial set. Thus there exists a unique filler $g^3$. This of course continues, so that a vertex $g \in \cat{Tot}(\bold{B}G(QX))_0$ determines and is completely determined by $g^1$ and $g^2$.

Let us repeat the above analysis for a $1$-simplex in $\cat{Tot}(\bold{B}G(QX))$. This is the data of a commutative diagram:
\begin{equation*}
    \begin{tikzcd}
	{\Delta^0 \times \Delta^1} & {\Delta^1 \times \Delta^1} & {\Delta^2 \times \Delta^1} & \dots \\
	{\prod_{N_0} \bold{B}G(U_{p_0})} & {\prod_{N_1} \bold{B}G(U_{p_1})} & {\prod_{N_2} \bold{B}G(U_{p_2})} & \dots
	\arrow[shift left=2, from=1-1, to=1-2]
	\arrow[shift right=2, from=1-1, to=1-2]
	\arrow[shift left=3, from=1-2, to=1-3]
	\arrow[shift right=3, from=1-2, to=1-3]
	\arrow[from=1-2, to=1-3]
	\arrow["{h^0}"', from=1-1, to=2-1]
	\arrow["{h^1}"', from=1-2, to=2-2]
	\arrow["{h^2}"', from=1-3, to=2-3]
	\arrow[shift left=2, from=2-1, to=2-2]
	\arrow[shift right=2, from=2-1, to=2-2]
	\arrow[shift left=3, from=2-2, to=2-3]
	\arrow[shift right=3, from=2-2, to=2-3]
	\arrow[from=2-2, to=2-3]
\end{tikzcd}
\end{equation*}
Now unravelling this diagram, skipping some similar details, such a $1$-simplex consists of the following data. If $g$ and $g'$ are $0$-simplices in $\cat{Tot}(\bold{B}G(QX))$ consisting of collections of maps $\{ g_f \}$ and $\{ g'_f \}$, then a $1$-simplex is a collection of maps $\{ h_{p_0}: U_{p_0} \to G \}$ indexed by plots $p_0: U_{p_0} \to X$ such that if $f_0: U_{p_1} \to U_{p_0}$ is a map of plots, then 
\begin{equation*}
    g'_{f_0} \cdot h_{p_1} = (h_{p_0} \circ f_0) \cdot g_{f_0},
\end{equation*}
and this is precisely a morphism of diffeological $G$-cocycles (\ref{eqn map of diff cocycles}). By the same reasoning as before, the rest of the diagram provides no further conditions on this data, as the maps $\Delta^k \times \Delta^1 \to \bold{B}G(QX_k)$ will consist of $(k+1)$-simplices, and $\bold{B}G(QX_k)$ is $2$-coskeletal, so that $h$ depends only on $h^0$ and $h^1$. Namely given $h^0$ and $h^1$, the $h^k$ for $k > 1$ are fully determined.

A $2$-simplex in $\cat{Tot}(\bold{B}G(QX))$ will similarly be completely determined by its boundary. Similar reasoning also proves that there are no additional conditions coming from higher $k$-simplices of $\cat{Tot}(\bold{B}G(QX))$. In other words, $\cat{Tot}(\bold{B}G(QX))$ is $2$-coskeletal. Further, since $\cechpre(\cart)$ is a simplicial model category and $\R \Hom(X, \bold{B}G) \cong \cat{Tot}(\bold{B}G(QX))$, this implies that $\cat{Tot}(\bold{B}G(QX))$ is a Kan complex. This combined with the fact that it is $2$-coskeletal implies that for any basepoint $g$, $\pi_k(\cat{Tot}(\bold{B}G(QX)),g) = 0$ for $k > 1$.

Now that we have an explicit description of $\cat{Tot}(\bold{B}G(QX))$ it is clear that this is nothing more than a diffeological version of the cocycle construction from classical differential geometry. Let us formalize this now.

We want to construct a map $\Phi: \cat{Tot}(\bold{B}G(QX)) \to N \cat{Coc}(X,G)$. Consider the left adjoint $h: \cat{sSet} \to \cat{Cat}$ to the nerve functor $N$, that sends a simplicial set to its homotopy category \cite[Example 1.5.5]{riehl2014categorical}, namely if $X$ is a simplicial set, then $hX$ is the category whose objects are the vertices of $X$, morphisms are freely generated by the $1$-simplices of $X$ and then quotiented by the $2$-simplices, in the sense that if $\sigma$ is a $2$-simplex in $X$ with $d_0 \tau = x, d_1 \tau = y, d_2 \tau = z$, then $x \circ z = y$ in $hX$. Note that by unravelling the above definitions, the composition of two morphisms $h' \circ h$ in $h\cat{Tot}(\bold{B}G(QX))$ is given by multiplication $h' \cdot h$ as in Definition \ref{def cocycle}. 

Let $\Phi: h \cat{Tot}(\bold{B}G(QX)) \to \cat{Coc}(X,G)$ denote the functor that sends an object $g = (g^0, g^1, \dots)$ to the cocycle it defines $\{g_{f_0} \}$, and a morphism $h = (h^0, h^1, \dots)$ to the morphism of cocycles it defines $\{h_p \}$. By the above discussion it is evident that this functor defines (one half of) an isomorphism of categories. In summary we have proved the following.

\begin{Lemma} \label{lem cocycle cat isomorphic to Tot of BG}
There is an isomorphism of simplicial sets
\begin{equation}
    N \cat{Coc}(X,G) \cong \cat{Tot}(\bold{B}G(QX)).
\end{equation}
\end{Lemma}

We are now in a position to prove the main theorem of this paper.
 
\begin{Th} \label{th diffeological principal bundles are infinity bundles}
The nerve of the category of diffeological principal $G$-bundles on $X$ and the nerve of the category of $G$-principal $\infty$-bundles on $X$ are weak homotopy equivalent
\begin{equation}
    N \cat{Prin}_G^\infty(X) \simeq N \cat{DiffPrin}_G(X).
\end{equation}
\end{Th}

\begin{proof}
Since the nerve functor sends equivalences of categories to homotopy equivalences of simplicial sets, Theorem \ref{th diff cocycle theorem} implies that there is a homotopy equivalence of simplicial sets $N \cat{Coc}(X,G) \simeq N \cat{DiffPrin}_G(X)$. Combining this with Proposition \ref{prop Rhom is principal infinity bundles} and Lemma \ref{lem cocycle cat isomorphic to Tot of BG} gives the result.
\end{proof}

\begin{Cor} \label{cor nonabelian cohomology}
Given a diffeological space $X$ and a diffeological group $G$, there is an isomorphism of pointed sets
\begin{equation}
    \check{H}^1_{\infty}(X, G) \cong \pi_0 \cat{DiffPrin}_G(X),
\end{equation}
where $\pi_0 \cat{DiffPrin}_G(X)$ denotes the set of isomorphism classes of diffeological principal $G$-bundles on $X$, pointed at the isomorphism class of trivial bundles.
\end{Cor}

We can still say more about the correspondence of Theorem \ref{th diffeological principal bundles are infinity bundles}. As in the paper \cite{NSSPresentations}, it is useful to see how one can obtain an actual diffeological principal $G$-bundle $\pi : P \to X$ from a $G$-cocycle $QX \to \B G$ using simplicial presheaves. It is basically a reformulation of Theorem \ref{th diff cocycle theorem}.

Given a diffeological group $G$, consider the diffeological groupoid
$$ G \times G \underset{m}{\overset{\text{pr}_1}{\rightrightarrows}} G$$
where the source map is the first projection and the target map is the map $m(g,h) = hg$\footnote{If this seems strange, see Example \ref{ex delooping stack}}. We can visualize morphisms in this groupoid by
\begin{equation*}
g \xrightarrow{(g,h)} hg \xrightarrow{(hg, k)} khg \qquad  =  \qquad g \xrightarrow{(g, kh)} khg. 
\end{equation*}
This defines a presheaf of groupoids on $\cart$ by
$$\left[ U \mapsto \left( C^\infty(U, G \times G) \rightrightarrows C^\infty(U,G) \right) \right].$$
Applying the nerve functor objectwise gives a simplicial presheaf which we denote by $\bold{E}G$.
There is a canonical map of simplicial presheaves $\bold{E}G \to \bold{B}G$ induced by the corresponding map of diffeological groupoids/presheaves of groupoids:
\begin{equation*}
    \begin{tikzcd}
	{G \times G} & G \\
	G & {*}
	\arrow["{\text{pr}_1}", shift left=2, from=1-1, to=1-2]
	\arrow["m"', shift right=2, from=1-1, to=1-2]
	\arrow["{\text{pr}_2}"', from=1-1, to=2-1]
	\arrow[shift left=2, from=2-1, to=2-2]
	\arrow[shift right=2, from=2-1, to=2-2]
	\arrow[from=1-2, to=2-2]
\end{tikzcd}
\end{equation*}
This functor can be visualized as
\begin{equation*}
    \begin{tikzcd}
	g && hg \\
	{*} && {*}
	\arrow["{(g,h)}", from=1-1, to=1-3]
	\arrow["h", from=2-1, to=2-3]
	\arrow[shorten <=2pt, shorten >=2pt, maps to, from=1-1, to=2-1]
	\arrow[shorten <=2pt, shorten >=2pt, maps to, from=1-3, to=2-3]
\end{tikzcd}
\end{equation*}
Furthermore this map is an objectwise Kan fibration, and therefore a projective fibration of simplicial presheaves.

\begin{Rem}
The simplicial presheaf $\bold{E}G$ and the map $\bold{E}G \to \B G$ described above are well known in the literature in the form $WG \to \conj{W}G$, which can be defined when $G$ is any presheaf of simplicial groups, see \cite{stevenson2012decalage} or \cite{NSSPresentations}. In this case we are using the same convention as \cite{goerss2012simplicial} where $W G = \text{Dec}^0 \overline{W}G$ and the map $WG \to \conj{W}G$ is given degreewise by $d_0^{\conj{W}G} : G^{\times k + 1} \to G^{\times k}$.
\end{Rem}

Now with such a map $g: QX \to \bold{B}G$, we can consider the following pullback in the category of simplicial presheaves.
\begin{equation} \label{eq pullback infinity bundle}
    \begin{tikzcd}
	{\widetilde{P}} & \bold{E}G \\
	QX & \bold{B}G
	\arrow[from=1-1, to=2-1]
	\arrow["g", from=2-1, to=2-2]
	\arrow[from=1-2, to=2-2]
	\arrow[from=1-1, to=1-2]
	\arrow["\lrcorner"{anchor=center, pos=0.125}, draw=none, from=1-1, to=2-2]
\end{tikzcd}
\end{equation}

\begin{Rem}
In the situation above, $\widetilde{P}$ is a $G$-principal $\infty$-bundle, and the construction of taking this pullback is precisely the map $\text{Rec}$ described in \cite[Definition 3.93]{NSSPresentations}.
\end{Rem}

Thus $\widetilde{P}$ is a simplicial presheaf, or equivalently a simplicial diffeological space such that
$$\widetilde{P}_1 = \coprod_{U_{p_1} \xrightarrow{f_0} U_{p_0}} U_{p_1} \times_G (G \times G), \qquad \widetilde{P}_0 = \coprod_{p_0 \in \cat{Plot}(X)} U_{p_0} \times G.$$
In other words, $\widetilde{P}_0$ is precisely the diffeological space $\widehat{P}$ from Section \ref{section diff principal bundles}, and the coequalizer of the face maps $\widetilde{P}_1 \rightrightarrows \widetilde{P}_0$, which is precisely the presheaf obtained by taking $\pi_0 \widetilde{P}$, is precisely the diffeological space $P = \cat{Cons}(g)$ constructed in Section \ref{section diff principal bundles}. 

Now let us show that the canonical map $q: \widetilde{P} \to P$ is an objectwise weak equivalence. First we notice that $QX$, $\B G$, and $\bold{E} G$ are nerves of presheaves of categories. Since the nerve functor is a right adjoint, we have that
\begin{equation*}
    \widetilde{P} \cong N[QX_1 \rightrightarrows B] \times_{N[G \rightrightarrows *]} N[G \times G \rightrightarrows G] \cong N([QX_1 \rightrightarrows B] \times_{[G \rightrightarrows *]} [G \times G \rightrightarrows G]).
\end{equation*}
Thus $\widetilde{P}$ as a simplicial presheaf is the nerve of a diffeological category/presheaf of categories, a morphism of which we can visualize in the same way as in Section \ref{section diff principal bundles}
\begin{equation*}
    (x_{p_1}, k_1) \xrightarrow{f_0} (x_{p_0}, k_0),
\end{equation*}
where $f_0(x_{p_1}) = x_{p_0}$ and $k_0 = g_{f_0}(x_{p_1}) \cdot k_1$.

Now think of $P$ as a presheaf of discrete categories with objects equivalence classes $[x_{p_0}, k_0]$. The map $q: \widetilde{P} \to P$ can be seen as the nerve of the functor
\begin{equation*}
    \begin{tikzcd}
	{(x_{p_1}, k_1)} && {(x_{p_0}, k_0)} \\
	{[x_{p_1},k_1]} && {[x_{p_0}, k_0]}
	\arrow["{f_0}", from=1-1, to=1-3]
	\arrow[shorten <=18pt, shorten >=18pt, Rightarrow, no head, from=2-1, to=2-3]
	\arrow[shorten <=2pt, shorten >=2pt, maps to, from=1-1, to=2-1]
	\arrow[shorten <=2pt, shorten >=2pt, maps to, from=1-3, to=2-3]
\end{tikzcd}
\end{equation*}
Now consider the map $f: P \to \widetilde{P}$ defined as follows. If $[x_{p_0}, k_0]$ is an object in $P$, then let $x = p_0(x_{p_0})$. Then consider the pair $(*_x, e_G)$, where $x: * \to X$ is the plot sending $*$ to the point $x \in X$. There is a unique morphism
\begin{equation*}
    (*_x, g^{-1}_{f_0}(x_{p_0}) \cdot k_0) \xrightarrow{x_{p_0}} (x_{p_0}, k_0),
\end{equation*}
given by the map of plots
\begin{equation*}
\begin{tikzcd}
	{*} && {U_{p_0}} \\
	& X
	\arrow["x"', from=1-1, to=2-2]
	\arrow["{p_0}", from=1-3, to=2-2]
	\arrow["{x_{p_0}}", from=1-1, to=1-3]
\end{tikzcd}
\end{equation*}
So set $f[x_{p_0}, k_0] = (*_x, g^{-1}_{f_0}(x_{p_0}) \cdot k_0)$. Then
\begin{equation}
    \begin{aligned}
        \widetilde{P}(f[x_{p_1}, k_1], (x_{p_0}, k_0)) & = \widetilde{P}((*_x, g^{-1}_{f_0}(x_{p_0}) \cdot k_0), (x_{p_0}, k_0)) \\
        & \cong P([*_x, g^{-1}_{f_0}(x_{p_0}) \cdot k_0], [x_{p_0}, k_0]),
    \end{aligned}
\end{equation}
where the second isomorphism holds because the map $(*_x, g^{-1}_{f_0}(x_{p_0}) \cdot k_0) \xrightarrow{x_{p_0}} (x_{p_0}, k_0)$ is the unique map between the source and target, and the existence of such a map means that $[*_x, g^{-1}_{f_0}(x_{p_0}) \cdot k_0] = [x_{p_0}, k_0]$. Thus
\begin{equation*}
    \widetilde{P}((*_x, g^{-1}_{f_0}(x_{p_0}) \cdot k_0), (x_{p_0}, k_0)) \cong P([*_x, g^{-1}_{f_0}(x_{p_0}) \cdot k_0], [x_{p_0}, k_0]) \cong *.
\end{equation*}
Therefore $f$ is an objectwise left adjoint to $q$. Since the nerve functor takes adjoint functors to homotopy equivalences of simplicial sets, we have proven the following.

\begin{Lemma}
The map $q: \widetilde{P} \to P$ of simplicial presheaves is an objectwise homotopy equivalence and therefore a \v{C}ech weak equivalence.
\end{Lemma}

Now notice that $\bold{E} G$ is objectwise contractible, indeed, for every $U$ the groupoid $[G(U) \times G(U) \rightrightarrows G(U)]$ has an initial object given by the constant map at the identity element $e_G$. Thus $\bold{E} G \to *$ is an objectwise weak equivalence. The map $\B G \to \cat{DiffPrin}_G$ is also an objectwise weak equivalence, see Example \ref{ex delooping stack} and Proposition \ref{prop bundles on cartesian spaces are trivial}. Thus there is a map of diagrams, where each component is an objectwise weak equivalence of simplicial presheaves
\begin{equation*}
\begin{tikzcd}
	& {\bold{E}G} \\
	{\widetilde{P}} & {\B G} & {*} \\
	QX & P & {\cat{DiffPrin}_G} \\
	& X
	\arrow[from=3-1, to=2-2]
	\arrow[from=2-1, to=3-1]
	\arrow[from=2-1, to=1-2]
	\arrow[from=1-2, to=2-2]
	\arrow[from=2-3, to=3-3]
	\arrow[from=3-2, to=2-3]
	\arrow[from=4-2, to=3-3]
	\arrow[from=3-2, to=4-2]
	\arrow[from=3-1, to=4-2]
	\arrow[from=2-1, to=3-2]
	\arrow[from=1-2, to=2-3]
	\arrow[from=2-2, to=3-3]
\end{tikzcd} 
\end{equation*}

Thus we have proven the following result.

\begin{Cor} \label{cor homotopy pullback bundle}
Given a diffeological space $X$, diffeological group $G$, and diffeological principal $G$-bundle $\pi : P \to X$, the commutative diagram of $\infty$-stacks
\begin{equation}
   \begin{tikzcd}
	P & {*} \\
	X & {\cat{DiffPrin}_G}
	\arrow["{g}", from=2-1, to=2-2]
	\arrow["{\pi}"', from=1-1, to=2-1]
	\arrow[from=1-1, to=1-2]
	\arrow[from=1-2, to=2-2]
\end{tikzcd} 
\end{equation} 
where $g: X \to \cat{DiffPrin}_G$ sends a plot $p_0: U_{p_0} \to X$ to the diffeological principal $G$-bundle $p_0^* P$, is a homotopy pullback square in the \v{C}ech model structure on $\spre(\cart)$.
\end{Cor}

\begin{proof}
Since $f$ is an objectwise weak equivalence of simplicial presheaves, and since $\widetilde{P}$ is the actual pullback of a projective fibration, it is a homotopy pullback in the projective model structure on simplicial presheaves on $\cart$. By Proposition \cite[Proposition 11.2]{rezktoposes2010}, it is also a homotopy pullback in the \v{C}ech model structure.
\end{proof}

\appendix

\section{Comparison of Site Structures} \label{section comparison of site structures}
Here we will prove that the definition of diffeological spaces as given in Definition \ref{def diffeological space} is equivalent to that usually presented in the literature, such as \cite[Article 1.5]{iglesias2013diffeology}, in the sense that their categories are equivalent. Further we will show other possible alternative definitions that have not appeared in the literature have equivalent categories as well. An example of this is \cite[Lemma 2.9]{watts2014diffeological}. The results of this section include this result.

We will do this by exploiting Theorem \ref{thm diff = consh(cart)}, and studying concrete sheaves over the smooth sites. Now coverages are those collections of families with the least amount of structure with which we can define sheaves on $\site$. There could be many different coverages which give rise to equivalent categories of sheaves. It can therefore be difficult to see directly when coverages give rise to the same sheaves. We will define a more restricted kind of coverage, known as a Grothendieck coverage or Grothendieck topology, which will make such comparison easier.

\begin{Def}
A \textbf{sieve} $R$ is a family of morphisms that is closed under precomposition, namely if $V \xrightarrow{g} U_i$ is a map in $\site$, and $U_i \xrightarrow{r_i} X \in R$, then $V \xrightarrow{r_i g} X \in R$.
\end{Def}

Given a category $\site$ and an object $U \in \site$, there is a bijection between sieves on $X$ and subfunctors $R \hookrightarrow yU$, where $yU = \left( V \mapsto \site(V,U) \right)$ denotes the Yoneda embedding on $U$. Indeed, given a sieve $R$, we can define a subfunctor $\widetilde{R} \hookrightarrow yU$ by setting $\widetilde{R}(V) = \{ f: V \to U \, : \, f \in R \}$ and noting that being a sieve implies that $\widetilde{R}$ is functorial under precomposition, and conversely if $\widetilde{R} \hookrightarrow yU$ is a subfunctor, then we can define a sieve $R$ by setting $R = \bigcup_{V \in \site} \widetilde{R}(V)$. Thus for the rest of this section a sieve will mean both a kind of family of morphisms and a subfunctor of the Yoneda embedding. If $U \in \site$ is an object, then we call $yU$ the \textbf{maximal sieve}. This is equivalently the family of all morphisms with codomain $U$.

For any family of morphisms $r = \{ r_i: U_i \to U \}$ over $U$, we can construct the smallest sieve $R = \overline{r}$ containing it as follows. Let $R$ be the set of morphisms $f: V \to U$ such that $f$ factors as:
\begin{equation*}
    \begin{tikzcd}
	V && U \\
	& {U_i}
	\arrow["f", from=1-1, to=1-3]
	\arrow["g"', from=1-1, to=2-2]
	\arrow["{r_i}"', from=2-2, to=1-3]
\end{tikzcd}
\end{equation*}
where $r_i: U_i \to U \in r$, and $g$ is a morphism in $\site$. In this case we say that $r$ \textbf{generates} the sieve $R$.

\begin{Lemma}[{\cite[C2.1 Lemma 2.1.3]{johnstone2002sketches}}] \label{lem sheaf on covering family iff on sieve it generates}
Suppose that $j$ is a coverage on a category $\site$. Then a presheaf $F$ is a sheaf on a family of morphisms $r = \{ U_i \to U \}$ if and only if it is a sheaf on the sieve $R = \overline{r}$ it generates. 
\end{Lemma}

\begin{Def}
We say that a collection of families $j$ is \textbf{sifted} if every $r \in j(U)$ is a sieve. If $j$ is further a coverage, we call it a sifted coverage. We call covering families of sifted coverages \textbf{covering sieves}.
\end{Def}

\begin{Lemma} \label{lem convenience of using sieves for matching families}
Let $R$ be a sieve over an object $U$ in a category $\site$ and $F$ a presheaf on $\site$. A collection $\{ s_f \in F(V) \}_{f \in R}$ of sections for every $f: V \to U$ in $R$ is a matching family if and only if $F(g)(s_f) = s_{fg}$ for every morphism $g: W \to V$ in $\site$.
\end{Lemma}

\begin{proof}
$(\Rightarrow)$ Suppose $\{s_f \}$ is a matching family, then consider the commutative diagram:
\begin{equation*}
    \begin{tikzcd}
	W & W \\
	V & U
	\arrow["g"', from=1-1, to=2-1]
	\arrow["f"', from=2-1, to=2-2]
	\arrow["gf", from=1-2, to=2-2]
	\arrow[Rightarrow, no head, from=1-1, to=1-2]
\end{tikzcd}
\end{equation*}
this implies that $F(g)(s_f) = s_{fg}$.

$(\Leftarrow)$ Suppose we have a commutative diagram:
\begin{equation*}
    \begin{tikzcd}
	A & {V'} \\
	V & U
	\arrow["h", from=1-1, to=1-2]
	\arrow["g"', from=1-1, to=2-1]
	\arrow["f"', from=2-1, to=2-2]
	\arrow["{f'}", from=1-2, to=2-2]
\end{tikzcd}
\end{equation*}
where $f,f' \in R$. Then $F(g)(s_f) = s_{fg} = s_{f'h} = F(h)(s_{f'})$, thus $\{ s_f \}$ is a matching family.
\end{proof}

If $j$ is a coverage, then let $\overline{j}$ denote the collection of families where $R \in \overline{j}(U)$ if $R = \overline{r}$ for some $r \in j(U)$. We call $\overline{j}$ the \textbf{sifted closure} of $j$.

\begin{Lemma} \label{lem sifted closure is a coverage}
The collection of families $\overline{j}$ is a sifted coverage of $\site$.
\end{Lemma}

\begin{proof}
Clearly $\overline{j}$ is sifted. We wish to show it is a coverage. Suppose we have a covering family $R \in \overline{j}(U)$, and a map $g: V \to U$. We wish to show that there is a covering family $R' \in \overline{j}(V)$ such that for every map $k \in R'$, $gk$ factors through some $l \in R$. Since $R = \overline{r}$, we know that since $j$ is a coverage, there exists some covering family $t \in j(V)$ with the corresponding property. In other words, for every map $k \in R'$ there is a commutative diagram:
\begin{equation*}
    \begin{tikzcd}
	W \\
	{V_j} & {U_i} \\
	V & U
	\arrow["{k_j}", from=1-1, to=2-1]
	\arrow["{t_j}", from=2-1, to=3-1]
	\arrow["g"', from=3-1, to=3-2]
	\arrow["{r_i}", from=2-2, to=3-2]
	\arrow["{s_j}", from=2-1, to=2-2]
	\arrow["k"', curve={height=18pt}, from=1-1, to=3-1]
\end{tikzcd}
\end{equation*}
but then $l = r_i s_j k_j$ is a morphism in $R = \overline{r}$. Thus $gk$ factors through $l$ as it is equal to it. 
\end{proof}

\begin{Cor}
Given a coverage $j$ on a category $\site$, a presheaf $F$ is a sheaf on $(\site, j)$ if and only if it is a sheaf on $(\site, \overline{j})$. In other words $\cat{Sh}(\site, j) = \cat{Sh}(\site, \overline{j})$.
\end{Cor}

\begin{proof}
This follows from Lemma \ref{lem sheaf on covering family iff on sieve it generates} and Lemma \ref{lem sifted closure is a coverage}.
\end{proof}

Now if $R \hookrightarrow yU$ is a sieve, and $f: V \to U$ is a morphism in $\site$, then let $f^* R$ denote the set of morphisms $g: W \to V$ such that $fg \in R$. This is equivalently the subfunctor $f^* R \hookrightarrow yV$ given by the pullback in $\cat{Pre}(\site)$
\begin{equation*}
    \begin{tikzcd}
	{f^* R} & R \\
	yV & yU
	\arrow["f"', from=2-1, to=2-2]
	\arrow[hook, from=1-2, to=2-2]
	\arrow[hook, from=1-1, to=2-1]
	\arrow[from=1-1, to=1-2]
	\arrow["\lrcorner"{anchor=center, pos=0.125}, draw=none, from=1-1, to=2-2]
\end{tikzcd}
\end{equation*}

\begin{Def}
A \textbf{Grothendieck coverage} is a sifted collection of families $J$ on a category $\site$ satisfying the following conditions:
\begin{itemize}
    \item [(C)] $J$ is a coverage,
    \item [(M)] for any object $U \in \site$, the maximal sieve $yU \in J(U)$, and
    \item [(L)] if $R \in J(U)$ and $S$ is another sieve on $U$ such that for each $f : V \to U \in R$, the sieve $f^*(S)$ belongs to $J(V)$.
\end{itemize}
If $(\site, J)$ is a Grothendieck coverage, then we call its sieves $R \in J(U)$ \textbf{covering sieves}. 
\end{Def}

\begin{Rem} \label{rem grothendieck coverage same as grothendieck topology}
Grothendieck coverages are usually referred to as Grothendieck topologies in the literature, but are typically presented with the following condition (C'): If $R \in J(U)$ and $f: V \to U$ any morphism in $\site$, then $f^* R \in J(V)$, instead of the condition $(C)$. It is not hard to show that these are equivalent definitions, see \cite[C2.1 Page 541]{johnstone2002sketches}.
\end{Rem}

\begin{Lemma} \label{lem pullback of sieve by map in sieve is maximal sieve}
Let $\site$ be a category and $R \hookrightarrow yU$ a sieve. If $g: V \to U$ is a map in $R$, then $g^* R = yV$.
\end{Lemma}

\begin{proof}
If $R$ is a sieve on $U$, then $g^*R = \{ f: W \to V \, : \, gf \in R \}$. But $R$ is a sieve and $g \in R$, so every map $f: W \to V$ has this property, since $R$ is closed under precomposition.
\end{proof}

\begin{Lemma} \label{lem covering sieve props}
Let $(\site, J)$ be a site with a Grothendieck coverage. Then if $R,R'$ are sieves on $U$, $R \subseteq R'$ and $R$ is a covering sieve, then $R'$ is a covering sieve.
\end{Lemma}

\begin{proof}
Let $g : V \to U \in R \subseteq R'$, then by Lemma \ref{lem pullback of sieve by map in sieve is maximal sieve}, we know that $g^*R = g^*R' = yV$, which is a covering sieve of $V$ by (M). Since this is true for all $g \in R$, $R'$ is a covering sieve by (L).
\end{proof}

Given a set $\{J_\alpha \, : \, \alpha \in A \}$ of Grothendieck coverages, it is not hard to check that the collection of families $J \coloneqq \bigcap_{\alpha \in A} J_\alpha$ defined by $J(U) = \bigcap_{\alpha \in A} J_\alpha(U)$ is a Grothendieck coverage. Thus if $j$ is a coverage on $\site$, we can consider $\overline{j}$, its sifted closure. By Lemma \ref{lem sheaf on covering family iff on sieve it generates}, we can then take the intersection of the set of Grothendieck coverages that contain all of the covering sieves of $\overline{j}$, which we denote by $\tau(j)$. This will be the smallest Grothendieck coverage containing $j$ and we will call it the Grothendieck coverage generated by $j$. 

\begin{Lemma}[{\cite[C2.1 Proposition 2.1.9]{johnstone2002sketches}}] \label{lem sheaf on grothendieck coverage generated by coverage}
Given a site $(\site, j)$, a presheaf $F$ will be a sheaf on $(\site, j)$ if and only if it is a sheaf on $( \site, \tau(j) \, )$. 
\end{Lemma}

Now we are in a position to compare different coverages on the same category. Suppose that $j,j'$ are coverages on a category $\site$ such that if $r'$ is a covering family in $j'(U)$, then there exists a covering family $r \in j(U)$ and a refinement $f: r \to r'$. We will say that $j'$ is \textbf{subordinate} to $j$ and write $j' \leq j$. 

\begin{Prop} \label{lem equivalent coverages}
Suppose that $j,j'$ are coverages on a category $\site$ such that $j \leq j'$ and $j' \leq j$, then $\cat{Sh}(\site, j) = \cat{Sh}(\site, j')$.
\end{Prop}

\begin{proof}
Suppose that $j' \leq j$. Then every covering family $r' \in j'(U)$ can be refined by a covering family $r \in j(U)$. Therefore $\overline{r} \subseteq \overline{r'}$, since sieves are closed under precomposition. Now note that $\overline{r} \in \tau(j)(U)$, and thus by Lemma \ref{lem covering sieve props}, $\overline{r'} \in \tau(j)(U)$. Thus if $r'$ is a covering family of $j'$, then $\overline{r'}$ is a covering sieve of $\tau(j)$. Thus if $F$ is a sheaf on $j$, then by Lemma \ref{lem sheaf on grothendieck coverage generated by coverage}, it will be a sheaf on $\tau(j)$, so it will then be a sheaf on $\overline{r'}$, and thus by Lemma \ref{lem sheaf on covering family iff on sieve it generates} it will be a sheaf on $r'$. Since $r'$ was arbitrary, $F$ is therefore a sheaf on all of $j'$. Thus if $F$ is a sheaf on $(\site, j)$, then it will be a sheaf on $(\site, j')$. Conversely $j \leq j'$ proves that $\cat{Sh}(\site, j) = \cat{Sh}(\site, j')$.
\end{proof}

\begin{Prop}
Let $j_\good$ denote the good cover coverage on $\cat{Man}$ defined in Example \ref{ex j good coverage}, and $j_\open$ denote the open cover coverage on $\cat{Man}$ defined in Example \ref{ex j open coverage}. Then $\cat{Sh}(\cat{Man}, j_\good) = \cat{Sh}(\cat{Man}, j_\open)$. This similarly holds for $\cart$ and $\cat{Open}$.
\end{Prop}

\begin{proof}
By \cite[Corollary 5.2]{bott1982differential}, we have that $j_\open \leq j_\good$. Now $j_\good \leq j_\open$, since every good open cover is in particular an open cover.
\end{proof}

\begin{Cor}
The categories of concrete sheaves on $\cat{Man}$ with the open and good open coverages agree:
$$\cat{ConSh}(\cat{Man}, j_\good) = \cat{ConSh}(\cat{Man}, j_\open).$$
This result remains true if we replace $\cat{Man}$ with $\cat{Open}$ or $\cart$.
\end{Cor}

Now we wish to compare sites whose underlying categories differ. Let $\site$ be a category and $\site' \hookrightarrow \site$ a full subcategory. Then a sieve $R \hookrightarrow yU$ on $\site$ is said to be a $\site'$-sieve if it is generated by a family of morphisms all of whose domains are objects in $\site'$.

\begin{Def}
Let $(\site, J)$ be a category with a Grothendieck coverage, and $\site' \hookrightarrow \site$ a full subcategory. We say that $\site'$ is \textbf{$J$-dense} in $\site$ if every object $U \in \site$ has a covering sieve $R \in J(U)$ that is a $\site'$-sieve.
\end{Def}

If $(\site, j)$ is a site where $j$ is not necessarily a Grothendieck coverage, then we say that a full subcategory $\site' \hookrightarrow \site$ is $j$-dense if it is $\tau(j)$-dense in $(\site, \tau(j))$.

By \cite[Theorem 5.1]{bott1982differential}, every finite dimensional smooth manifold has a good open cover. Thus if $\mathcal{U} = \{U_i \subseteq M \}$ denotes a good open cover of $M$, then $\overline{\mathcal{U}}$ is a covering sieve of $(\cat{Man}, \tau(j_\open))$ and it is a $\cart$-sieve. Since this is true for any manifold $M$, it follows that $\cart$ is $j_\open$-dense in $(\cat{Man}, j_\open)$. By the same argument $\cart$ is also dense in $(\cat{Man}, j_\good)$. This also implies that $\cat{Open}$ is dense in $(\cat{Man}, j)$ for $j \in \{ j_\open, j_\good \}$.

Now suppose $(\site, J)$ is a site with a Grothendieck coverage. If $\site' \hookrightarrow \site$ is a full subcategory, define a collection of families $J'$ on $\site'$ by defining $J'(U)$ to be the collection of those covering sieves $R \in J(U)$ that are also $\site'$-sieves. It is not hard to show that this is also a Grothendieck coverage, called the \textbf{induced coverage} on $\site'$, and denoted $J|_{\site'}$.

The following result is well-known in the literature as the \textbf{Comparison Lemma}.

\begin{Th}[{\cite[Theorem 2.2.3]{johnstone2002sketches}}] \label{thm comparison lemma}
Let $(\site, J)$ be a site with a Grothendieck coverage and $\site' \hookrightarrow \site$ a $J$-dense full subcategory. Then the restriction functor $\text{res}: \cat{Pre}(\site) \to \cat{Pre}(\site')$ itself restricts to a functor $\text{res}: \cat{Sh}(\site, J) \to \cat{Sh}(\site', J|_{\site'})$, and this functor is an equivalence of categories. 
\end{Th}

Note that $\tau(j^{\cart}_\open) = \tau(j^\cat{Man}_\open)|_{\cart}$. This can be seen by simply noting that every sieve in $\tau(j^\cat{Man}_\open)|_{\cart}$ is generated by a open cover by cartesian spaces, and contains every such sieve. A similar argument proves the same for $j^\cart_\good$ and $j^\cat{Open}_\good, j^\cat{Open}_\open$.

\begin{Cor} \label{cor smooth sites have equivalent sheaves and concrete sheaves}
All categories of the form $\cat{ConSh}(\site, j)$ for $\site \in \{\cart, \cat{Open}, \cat{Man} \}$ and $j \in \{j_\open, j_\good \}$ are equivalent.
\end{Cor}

\begin{proof}
Theorem \ref{thm comparison lemma} implies that the categories $\cat{Sh}(\site, j)$ for $\site \in \{\cart, \Open, \Man \}$ and $j \in \{j_\open, j_\good \}$ are all equivalent. Further, using the same argument as in the proof of \cite[Lemma 2.9]{watts2014diffeological}, the above equivalences restrict to equivalences of all the full subcategories $\cat{ConSh}(\site, j)$ of concrete sheaves.
\end{proof}

Thus by Theorem \ref{thm diff' = consh(open)}, we have that $\cat{Diff}' \simeq \cat{ConSh}(\cat{Open}, j_\open)$, and by Corollary \ref{cor smooth sites have equivalent sheaves and concrete sheaves}, we have that $\cat{ConSh}(\cat{Open}, j_\open) \simeq \cat{ConSh}(\cart, j_\good) \simeq \cat{Diff}$. Thus we have proved the main proposition of this section.

\begin{Prop} \label{prop classical diff spaces equiv to diff spaces}
The category of classical diffeological spaces $\cat{Diff}'$ is equivalent to the category of diffeological spaces $\cat{Diff}$ introduced in Definition \ref{def diffeological space}.
\end{Prop}

\printbibliography

\end{document}